\documentclass{article}
\usepackage[titletoc]{appendix}
\usepackage[a4paper,top=1.50cm, bottom=1.50cm,left=1.50cm,right=1.50cm,%
includehead,includefoot]{geometry}
\usepackage{graphicx}
\usepackage{subcaption}
\newcommand{\email}[1]{\href{mailto:#1}{\texttt{#1}}}
\usepackage{amsmath,amssymb,amsfonts,amsthm}
\usepackage{graphicx}
\usepackage{float}
\usepackage{xcolor}
\usepackage{multirow}
\usepackage{booktabs}
\usepackage[linesnumbered,boxed,ruled,commentsnumbered]{algorithm2e}
\usepackage[numbers,square,sort&compress]{natbib}
\usepackage{hyperref}
  \hypersetup{colorlinks, citecolor=blue, linkcolor=blue, breaklinks=true}
\usepackage{epstopdf}
\usepackage{yhmath} 
\usepackage{indentfirst} 
\usepackage{xcolor}

\allowdisplaybreaks
\numberwithin{equation}{section} 
\newtheorem{theorem}{Theorem}[section]
\newtheorem{definition}{Definition}[section]
\newtheorem{lemma}{Lemma}[section]
\newtheorem{remark}{Remark}[section]
\newtheorem{example}{Example}[section]

\usepackage{listings}
\usepackage{matlab-prettifier}

\newcommand{\bbm}{\begin{bmatrix}}
\newcommand{\ebm}{\end{bmatrix}}

\begin{document}

\title{High-precision randomized  iterative methods for the random feature method}

\author{
  Jingrun Chen
    \thanks{School of Mathematical Sciences and Suzhou Institute for Advanced Research, University of Science and Technology of China, China (\email{jingrunchen@ustc.edu.cn})}
  \and
  Longze Tan 
    \thanks{School of Mathematical Sciences and Suzhou Institute for Advanced Research, University of Science and Technology of China, China (\email{ba23001003@mail.ustc.edu.cn})}
}

\date{}

\maketitle

\begin{abstract}
This paper focuses on solving large-scale, ill-conditioned, and overdetermined sparse least squares problems that arise from numerical partial differential equations (PDEs), mainly from the random feature method. To address these difficulties, we introduce (1) a count sketch technique  to sketch the original matrix to a smaller matrix; (2) a QR factorization or a  singular value decomposition for the smaller matrix to obtain the preconditioner, which is multiplied to the original matrix from the right-hand side; (3) least squares iterative solvers to solve the preconditioned least squares system. Therefore, the methods we develop are termed CSQRP-LSQR and CSSVDP-LSQR. Under mild assumptions, we prove that the preconditioned problem holds a condition number whose upper bound is independent of the condition number of the original matrix, and provide error estimates for both methods. Ample numerical experiments, including least squares problems arising from two-dimensional and three-dimensional PDEs and the Florida Sparse Matrix Collection, are conducted. Both methods are comparable to or even better than direct methods in accuracy and are computationally more efficient for large-scale problems. This opens up the applicability of the random feature method for PDEs over complicated geometries with high-complexity solutions.
\end{abstract}
\section{Introduction}
Numerical partial differential equations (PDEs) are widely used in scientific and engineering applications. Traditional methods, such as the finite difference and finite element methods, are robust in terms of accuracy with convergence guarantee in general. Recently, methods based on deep neural networks (DNNs) have been designed to solve PDEs in high dimensions. These methods are mesh-free and thus can be applied to problems with complex geometries. However, the highly nonlinear and nonconvex optimization problem of DNNs is difficult to solve, and the accuracy is typically around $1\%$; see \cite{EWYu, SirSp, RaPK, WAN} for examples. Therefore, recent efforts have been devoted to design methods that combine the advantages of both types of methods.
\par
One notable method is the random feature method (RFM) \cite{CCEY}. In RFM, the approximation space is defined as the composition of the partition of unity (PoU) and the random feature functions, two-layer neural networks with inner-layer parameters randomly chosen and only the out-layer parameters to be optimized. Therefore, only a convex optimization (linear least squares) problem is solved to get the numerical solution. 
Once the least squares problem is solved with high precision, RFM exhibits the exponential convergence rate as the number of random feature functions increases. RFM has been successfully applied to problems from solid mechanics, fluid mechanics, interface problems, and time-dependent problems over complex geometries \cite{CCEY, CELuo, CCYang}. 

To be specific, the least squares problem in RFM can be formulated as 
\begin{equation}\label{least-squares problem} \underset{\mathbf{x} \in \mathbb{R}^n}{\min}\|\mathbf{A}\mathbf{x}-\mathbf{b}\|_2^2, 
\end{equation}
where $\mathbf{A}\in \mathbb{R}^{m\times n}$ and $\mathbf{b}\in \mathbb{R}^m$. The geometric complexity in real applications leads to $m \gg n$, and typically $m$ is larger than $n$ by a few orders of magnitude. If $\mathbf{A}$ has full column rank, \eqref{least-squares problem} admits a unique solution, which can be obtained by direct methods such as singular value decomposition (SVD) and QR decomposition. The computational complexity is $\mathcal{O}(mn\min(m, n))$. Meanwhile, the computational complexity of iterative methods is $\mathcal{O}(mn)$ per iteration. The total cost depends on the number of iterations required for a given tolerance. An iterative method will be attractive if it can achieve the same level of accuracy with the iteration number significantly less than $\min (m, n)$. In applications, it is observed that direct methods can achieve high accuracy while iterative methods hinder the accuracy due to the ill-conditioned least squares problem in RFM, even if a significant number of iterations is executed \cite{ChenESun}. This motivates the current work.

The broadness of singular value distribution of $\mathbf{A}$ leads to the ill condition. Therefore, classical Krylov subspace iterative methods such as CGLS \cite{Golub, HeSt}, LSQR \cite{PaSa}, and LSMR \cite{FoSa} have very slow convergence rates, making them difficult to find approximate solutions with high precision. One line of work is to find an effective preconditioner, which is challenging unless the coefficient matrix exhibits a specific structure and specific properties \cite{Benzi}. Recently, randomized preconditioning techniques have shown great potential to alleviate the issues mentioned above in direct and iterative methods. For example, in \cite{RoTy}, Rokhlin and Tygert introduced a dimensionality reduction technique for overdetermined systems using a subsampled randomized Fourier transform matrix to the original matrix, got a smaller matrix, and applied the QR decomposition to this matrix to construct a preconditioner for conjugate gradient method. In \cite{AMT}, Avron, Maymounkov, and Toledo  introduced Blendenpik, an algorithm designed to solve dense, highly overdetermined least squares problems. 
In \cite{MSM}, Meng, Saunders, and Mahoney proposed a least squares solver named LSRN, which enables parallel computation of the sketched matrix $\mathbf{G}\mathbf{A}$, where $\mathbf{G}$ is a random normal matrix and then utilizes a compact SVD to construct the preconditioner in LSQR. Another line of work is to incorporate randomization directly to solve linear systems of equations \cite{DMMS, LeveLe, MaNeRa, Needell, StrVer, ZoFer}. These randomized algorithms are often as competitive as, or even superior to, current deterministic methods. For example, in \cite{StrVer}, Strohmer and Vershynin first introduced the randomized Kaczmarz (RK) method for solving linear systems. In \cite{BaiWu_4}, Bai and Wu introduced the greedy randomized Kaczmarz (GRK) method, which uses an effective greedy probability criterion to select a subset of rows in the coefficient matrix row in each iteration. Later, this greedy approach was used to select the working columns in the randomized coordinate descent method \cite{LeveLe,MaNeRa}, leading to the development of the greedy randomized coordinate descent (GRCD) method for large linear least squares problems \cite{BaiWu_5, ZhangGuo}.
For the least squares problems in RFM, our numerical experiments indicate that both lines of work, such as GRK and GRCD methods, Blendenpik, LSRN, and their variants, do not work well, and the approximation accuracy is limited.

In this work, we combine the randomization and preconditioning techniques to design randomized iterative methods for linear least squares problems.
Our methods consist of three main components: (1) a count sketch technique \cite{ChChFC, ClWo, Woodruff} is used to sketch the original matrix to a smaller matrix; (2) a QR factorization or a singular value decomposition is employed for the smaller matrix to obtain the preconditioner, which is multiplied to the original matrix from the right-hand side; (3) least squares iterative solvers are employed to solve the preconditioned least squares system. Therefore, the methods we develop are termed CSQRP-LSQR and CSSVDP-LSQR. In the framework of $\ell_2$ subspace embedding, we prove that the preconditioned problem holds a condition number whose upper bound is independent of the condition number of the original matrix, and provide error estimates for both methods. The design philosophy of our methods differs from those of Blendenpik and LSRN since we explicitly construct the preconditioned system which becomes a well-conditioned least squares problem. At first glance, this explicit construction leads to a comparable computational complexity between the CSQRP-LSQR and CSSVDP-LSQR methods and direct methods based on QR decomposition. However, the CSQRP-LSQR method is computationally faster on large-scale problems compared to the state-of-the-art sparse direct solver SPQR \cite{Davis_SPQR}. Meanwhile, the relative least squares error precision of CSQRP-LSQR and CSSVDP-LSQR methods is on par with SPQR and Householder QR \cite{Davis}. 

The paper is organized as follows. Section \ref{section:2} first outlines the random feature method for completeness, and introduces the concept of $\ell_2$-subspace embedding. The CSQRP-LSQR and CSSVDP-LSQR methods, along with their error estimates, are detailed in Section \ref{section:3}. Section \ref{section:4} provides numerous numerical examples to validate the effectiveness of the proposed methods. Conclusions are drawn in Section \ref{section:5}. 
Additional information, including numerical results for PDEs over complex geometries without explicit solutions, the setup of hyper-parameters in RFM, and 
matrix details from  the Florida Sparse Matrix Collection, are provided in 
Appendix \ref{my_experiments_without_explicit_solutions},  
Appendix \ref{my_hyperparameters_RFM}, and Appendix \ref{test_matrices_information}, respectively.

\section{Preliminaries}\label{section:2}
For completeness, in this section, we introduce the RFM for solving PDEs and the concept of $\ell_2$-subspace embedding.

\subsection{The random feature method}\label{subsection:2.1}
Consider the following boundary value problem:
\begin{equation}\label{boundary-value problem}
\begin{cases}\mathcal{L}u(\mathbf{x})=f(\mathbf{x}), & \mathbf{x} \in \Omega, \\ \mathcal{B} u(\mathbf{x})=g(\mathbf{x}), & \mathbf{x} \in \partial \Omega,\end{cases}
\end{equation}
where $f$ and $g$ are known functions, $\mathcal{L}$ and $\mathcal{B}$ are differential and boundary operators, respectively, and $\Omega$ is bounded and connected domain in $\mathbb{R}^d$. 

In RFM, the approximate solution $u_n(\mathbf{x})$ is expressed as a linear combination of $n$ random feature functions defined on $\Omega$, i.e., 
\begin{equation}\label{RFM_approximation_solution}
  u_n(\mathbf{x})=\sum_{j=1}^n u_j \phi_j(\mathbf{x}).
\end{equation}
To accommodate the local variations in the solutions of PDEs, we combine the partition of unity (PoU) functions and the random feature functions. To construct the PoU, the usual approach is to (uniformly) divide the domain $\Omega$ into $M_p$ small subdomains, denoted as $\left\{\Omega_i\right\}_{i=1}^{M_p}$. For each subdomain $\Omega_i$,  $\mathbf{r}_i=(r_{i1}, r_{i2},\ldots,r_{id})^T$ is preselected, the center point $\mathbf{x}_i$ is selected, and define the normalized coordinate
$$
\mathbf{l}_i(\mathbf{x})=\left(\frac{x_1-x_{i1}}{r_{i1}}, \frac{x_2-x_{i2}}{r_{i2}}, \ldots, \frac{x_d-x_{id}}{r_{id}}\right)^{\top}, \quad i=1, \ldots, M_p.
$$
This linear transformation $\mathbf{l}_i(\mathbf{x})$ maps $\left[x_{i1}-r_{i1}, x_{i1}+r_{i1}\right] \times \cdots \times\left[x_{id}-r_{id}, x_{id}+r_{id}\right]$ onto $[-1,1]^d$. Thus, a PoU function centered at $\mathbf{x}_i$ can be constructed. When $d=1$, two commonly used PoU functions are 
$$
\begin{aligned}
\psi_i^a(x) & =\mathbb{I}_{[-1,1]}\left(l_i(x)\right),\\
\psi_i^b(x) & =\mathbb{I}_{\left[-\frac{5}{4},-\frac{3}{4}\right]}\left(l_i(x)\right) \frac{1+\sin \left(2 \pi l_i(x)\right)}{2}+\mathbb{I}_{\left[-\frac{3}{4}, \frac{3}{4}\right]}\left(l_i(x)\right)+\mathbb{I}_{\left[\frac{3}{4}, \frac{5}{4}\right]}\left(l_i(x)\right) \frac{1-\sin \left(2 \pi l_i(x)\right)}{2}.
\end{aligned}
$$
$\psi_i^a(x)$ is discontinuous with a smaller support, while $\psi_i^b(x)$ is continuously differentiable with a larger support. In high-dimensional,  PoU function $\psi_i(\mathbf{x})$ can be obtained directly from the tensor product, namely $\psi_i(\mathbf{x})=\prod_{k=1}^d \psi_i\left(x_k\right)$.
\par
Next, $J_i$ random feature functions on each subdomain $\Omega_i$ are constructed as follows
$$
\phi_{ij}(\mathbf{x})=\sigma\left(\mathbf{k}_{ij} \cdot \mathbf{l}_i(\mathbf{x})+b_{ij}\right), \quad i=1, \ldots, J_i,
$$
where the nonlinear activation function $\sigma$ is often chosen as tanh or trigonometric functions and each component of $\mathbf{k}_{ij}$ and $b_{ij}$ is chosen uniformly from the interval $\left[-R_{ij}, R_{ij}\right]$ and is fixed. Thus, the approximate solution $u_n(\mathbf{x})$ can be represented as 
$$
u_n(\mathbf{x})=\sum_{i=1}^{M_p} \psi_i(\mathbf{x}) \sum_{j=1}^{J_i} u_{ij} \phi_{ij}(\mathbf{x}),
$$
where $u_{ij}\;(i=1,\ldots,M_p, j=1,\ldots, J_i)$ are the unknown coefficients that need to be determined, and $n = M_p \cdot J_i$ denotes the degrees of freedom. A loss function for \eqref{boundary-value problem} can be formulated as follows
\begin{equation}\label{loss_function}
  Loss=\sum_{\mathbf{x}_i \in C_I} \sum_{k=1}^{K_I} \lambda_{I i}^k\|\mathcal{L}^k u_M\left(\mathbf{x}_i\right)-f^k\left(\mathbf{x}_i\right)\|_{l^2}^2+\sum_{\mathbf{x}_j \in C_B} \sum_{\ell=1}^{K_B} \lambda_{B j}^{\ell}\|\mathcal{B}^{\ell} u_M\left(\mathbf{x}_j\right)-g^{\ell}\left(\mathbf{x}_j\right)\|_{l^2}^2. 
\end{equation}
Here $ \lambda_{I i}^k$ and $\lambda_{B j}^{\ell}$ are penalty parameters, as defined in \cite{CCEY} and $\left\{\mathbf{x}_i\right\} \subset \Omega$ and $\left\{\mathbf{x}_j\right\} \subset \partial \Omega$ are collocation points. The total number of collocation points equals $m$, leading to the least squares problem \eqref{least-squares problem}. It is worth noting that, when $M_p>1$ and PoU function $\psi^a(\mathbf{x})$ is used, smoothness conditions between the adjacent elements in the partition are explicitly imposed by adding regularization terms in loss function \eqref{loss_function} while no regularization is required when $\psi^b(\mathbf{x})$ is used for second-order equations due to its first-order continuity. Since only the coefficients $\left\{u_{ij}\right\}$ need to be solved in \eqref{loss_function}. Therefore, minimizing \eqref{loss_function} is a convex optimization problem, which can be equivalently reformulated as a linear least squares problem.

\subsection{The $\ell_2$-subspace embedding} \label{Subsection:2.2}
We first define an $\ell_2$-subspace embedding  for the column space of an $m \times n$ matrix $\mathbf{A}$.
\begin{definition}(\cite{Woodruff})\label{subspace_def}
  A $(1 \pm \epsilon) \ell_2$-subspace embedding for the column space of an $m \times n$ matrix $\mathbf{A}$ is a matrix $\mathbf{S} \in \mathbb{R}^{s \times m}$ for which for all $\mathbf{x} \in \mathbb{R}^n$
  \begin{equation}\label{L2_subspaace embedding}
    (1-\epsilon)\|\mathbf{A}\mathbf{x}\|_2^2 \leq\|\mathbf{S}\mathbf{A}\mathbf{x}\|_2^2 \leq(1+\epsilon)\|\mathbf{A}\mathbf{x}\|_2^2.
  \end{equation}
\end{definition} 
Subspace embeddings primarily aim to reduce the row count of matrix $\mathbf{S}$ and accelerate the computation of  sketched matrix $\mathbf{S} \mathbf{A}$, a common bottleneck in applications \cite{Woodruff}. Among various methods, the oblivious $\ell_2$-subspace embedding is notably effective. A precise definition is provided below.
\begin{definition}(\cite{Woodruff})
  Suppose $\prod$ is a distribution on $s \times m$ matrix $\mathbf{S}$, where $s$ is a function of $m, n, \epsilon$, and $\delta$.  Suppose that with probability at least $1-\delta$, for any fixed $m \times n$ matrix $\mathbf{A}$, a matrix $\mathbf{S}$ drawn from distribution $\prod$  has the property that $\mathbf{S}$ is a $(1 \pm \epsilon) \ell_2$-subspace embedding for $\mathbf{A}$. Then $\mathbf{S}$ is called an $(\epsilon, \delta)$ oblivious $\ell_2$-subspace embedding.
\end{definition}
An oblivious $\ell_2$-subspace embedding is a random subspace embedding matrix following a certain distribution. For simplicity, the term "oblivious" may be omitted. There are several methods to construct the sketching matrix, including the Gaussian random projection (GRP) \cite{HMT, MSM}, the subsampled randomized Hadamard transform (SRHT) \cite{LDFU, Tro, Wang, Woodruff}, and the count sketch transform \cite{ChChFC,ClWo, Woodruff}. A comparison of these methods is provided in Table \ref{tab:sketching_matrices_complexities}. Here $\mathrm{nnz}(\cdot)$ denotes the number of non-zero elements in the matrix. The $\mathcal{O}(\mathrm{nnz}(\mathbf{A}))$ complexity is indeed optimal for computing the sketched matrix, and thus we choose the count sketch method. It is important to note that there is no need to form $\mathbf{S}$ explicitly; see \cite{Wang} for the implementation details. Furthermore, in numerical experiments, we are surprised that a count sketch matrix $s=\gamma n, \gamma=3$ already yields excellent preconditioning effects.
\begin{table}[!htbp]
  \begin{center}
  \caption{Comparison of sketching methods.}\label{tab:sketching_matrices_complexities}%
  \setlength{\tabcolsep}{6.50mm}  %
  \begin{tabular}{lllll}
  \toprule
  Sketching method & GRP & SRHT &  count sketch  \\
  \midrule
  $s$ & $\mathcal{O}\left(\frac{n}{\epsilon^2}\right)$ & $\mathcal{O}\left(\frac{log(n)(\sqrt{n}+\sqrt{log(m)})^2}{\epsilon^2}\right)$ & $\mathcal{O}\left(\frac{(n^2+n)}{\epsilon^2}\right)$ \\
  Complexity & $\mathcal{O}(mns)$ & $\mathcal{O}\left(mnlog(s)\right)$ & $\mathcal{O}(\mathrm{nnz}(\mathbf{A}))$ \\
  \bottomrule
  \end{tabular}
  \end{center}
\end{table}

Next, we define the count sketch transform or sparse embedding matrix \cite{ChChFC,ClWo, Woodruff}. 
\begin{definition}\label{count sketch transform or sparse embedding matrix}
  (Count Sketch transform) A count sketch transform is defined to be $\mathbf{S}=\Phi \mathbf{D} \in$ $\mathbb{R}^{s \times m}$. Here, $\mathbf{D}$ is an $m \times m$ random diagonal matrix with each diagonal entry independently chosen to be +1 or -1 with equal probability, and $\Phi \in\{0,1\}^{s \times m}$ is a $s \times m$ binary matrix with $\Phi_{h(i), i}=1$ and all remaining entries 0 , where $h:[m] \rightarrow[s]$ is a random map such that for each $i \in[m], h(i)=j$ with probability $1 / s$ for each $j \in[s]$.
\end{definition}

An explicit example of the count sketch matrix when $s=4, m=5$ is 
$$
\mathbf{S} = \left(\begin{array}{ccccc}
0 & 0 & -1 & 1 & 0 \\
1 & 0 & 0 & 0 & 0 \\
0 & 0 & 0 & 0 & 1 \\
0 & -1 & 0 & 0 & 0
\end{array}\right),
$$
which only has one non-zero entry per column. Thus the computation of $\mathbf{SA}$ is $\mathcal{O}(\mathrm{nnz}(\mathbf{A}))$.

To end this section, we state a lemma that is crucial for the error estimates of the proposed methods.
\begin{lemma}(\cite{MM})\label{count sketch lemma}
  Suppose  that $\mathbf{S} \in \mathrm{R}^{s \times m}$ is a count sketch transform with $s=\left(n^2+n\right) /\left(\delta \varepsilon^2\right)$, where $0<\delta, \epsilon<1$. Then with probability at least $1-\delta$, we have that the inequality \eqref{L2_subspaace embedding} holds.
\end{lemma}

\section{The proposed methods}\label{section:3}
In this section, we propose two methods to solve linear least square problems. They all use a count sketch matrix  to sketch the original matrix to a smaller matrix and the LSQR method to solve the preconditioned problem. The only difference is how to get the preconditioner: one uses the QR decomposition for the smaller matrix and the other uses the SVD decomposition. Therefore, these two methods are termed the CSQRP-LSQR method and the CSSVDP-LSQR method, respectively. 

The algorithmic details of the CSQRP-LSQR method are given below.
\par
\begin{algorithm}[H]
  \caption{The CSQRP-LSQR method}\label{Alg:CSQR-PLSQR}
  \KwIn {$\mathbf{A}\in \mathbb{R}^{m\times n},\mathbf{b}\in \mathbb{R}^n$ and $\mathbf{y}^{(0)}\in \mathbb{R}^n$ }
	\KwOut {Approximate solution {$\mathbf{x}^{(k)}\in \mathbb{R}^n$}}
  Choose an oversampling factor $\gamma>1$ and set $s=\gamma n$ ($s<m$).\\
  Generate a $s$-by-$m$ count sketch matrix $\mathbf{S}$ and compute $\tilde{\mathbf{A}}=\mathbf{S}\mathbf{A}$.\\
  Apply the QR factorization to $\tilde{\mathbf{A}}$ and obtain the upper triangular matrix $\mathbf{R}$.\\
  Compute the preconditioned matrix $\mathbf{B}=\mathbf{A}\mathbf{R}^{-1}$.\\
  Use the LSQR method to find an approximate solution $\mathbf{y}^{(k)}$ of 
  \begin{equation}\label{pre_least_squares}
    \underset{\mathbf{y}\in \mathbb{R}^n}{\min}\|\mathbf{B}\mathbf{y}-\mathbf{b}\|_2.
  \end{equation}\\
  Return $\mathbf{x}^{(k)}=\mathbf{R}^{-1}\mathbf{y}^{(k)}$.
\end{algorithm}

\begin{remark}
  In \cite{Wang}, it is suggested that a more natural guess for the initial point $\mathbf{y}^{(0)}$ is readily available, namely, $\mathbf{y}^{(0)}=\mathbf{Q}^T(\mathbf{S} \mathbf{b})$. 
\end{remark}

The $(1 \pm \epsilon) \ell_2$-subspace embedding count sketch matrix of the column space of the matrix $\mathbf{A}$  is presupposed to satisfy the conditions described in Lemma \ref{count sketch lemma}, ensuring that inequality \eqref{L2_subspaace embedding}  holds with at least $1-\delta$ probability. The following Theorem \ref{theorem:3.1} plays a crucial role in analyzing the error estimate of the CSQRP-LSQR method and motivates the design of the CSSVDP-LSQR method.
\begin{theorem}\label{theorem:3.1}
  For a given $\epsilon \in(0,1)$, if the count sketch matrix $\mathbf{S}$ is an $\ell_2$-subspace embedding for a given matrix $\mathbf{A}$ and also satisfies the conditions described in Lemma \ref{count sketch lemma}, then $rank(\mathbf{A})=rank(\mathbf{S}\mathbf{A})$  holds with a probability of at least $1-\delta$, where $rank(\cdot)$ denotes the rank of  the argument matrix.
\end{theorem}
\begin{proof}
  By applying the rank-nullity theorem, we immediately conclude that 
  \begin{equation}\label{rank-nullity_equation}
    rank(\mathbf{A})+dim(ker(\mathbf{A}))=rank(\mathbf{S}\mathbf{A})+dim(ker(\mathbf{S}\mathbf{A}))=n.
  \end{equation}
  We have
  $$
  dim(ker(\mathbf{A}))\leq dim(ker(\mathbf{S}\mathbf{A})).
  $$
  If the aforementioned inequality strictly holds, i.e., $dim(ker(\mathbf{A}))< dim(ker(\mathbf{S}\mathbf{A}))$, then there must exist at least one vector $\mathbf{x}\in \mathbb{R}^n$ such that $\|\mathbf{S}\mathbf{A}\mathbf{x}\|_2=0$ but $\|\mathbf{A}\mathbf{x}\|_2\neq 0$. Hence, by Lemma \ref{count sketch lemma}, this contradicts with the inequality \eqref{L2_subspaace embedding}, thereby $ dim(ker(\mathbf{A}))=dim(ker(\mathbf{S}\mathbf{A}))$, and consequently $rank(\mathbf{A})=rank(\mathbf{S}\mathbf{A})$.
\end{proof}
\begin{theorem}\label{theorem:3.2}
  Let the matrix $\mathbf{A}$ in \eqref{least-squares problem} be of full column rank and the construction of the count sketch matrix $\mathbf{S}$ satisfy the condition in Lemma \ref{count sketch lemma}, setting $\tilde{\mathbf{A}}=\mathbf{S}\mathbf{A}$. Then, the following conclusions hold with at least a $1-\delta$ probability:
  \begin{itemize}
    \item The upper triangular matrix $\mathbf{R}$ obtained from the $\mathrm{QR}$ decomposition of $\tilde{\mathbf{A}}$ is an invertible matrix. 
    \item The condition number of the preconditioned matrix $\mathbf{A}\mathbf{R}^{-1}$ satisfies $\kappa(\mathbf{A}\mathbf{R}^{-1})\leq \sqrt{\frac{1+\epsilon}{1-\epsilon}}$.
    \item The preconditioned least squares problem \eqref{pre_least_squares} has a unique solution.
  \end{itemize}
\end{theorem}
\begin{proof}
  (1) According to Theorem \ref{theorem:3.1}, the sketched matrix $\tilde{\mathbf{A}}$ and matrix $\mathbf{A}$ both have rank $n$. Therefore, the rank of the $n \times n$ upper triangular matrix $\mathbf{R}$ must be $n$, making it invertible.
  \par (2) According to Lemma \ref{count sketch lemma}, with probability of at least $1-\delta$, it follows that
  $$
  (1-\epsilon)\|\mathbf{A}\mathbf{x}\|_2^2 \leq\|\tilde{\mathbf{A}}\mathbf{x}\|_2^2 \leq(1+\epsilon)\|\mathbf{A}\mathbf{x}\|_2^2, \quad \text { for all } \mathbf{x} \in \mathbb{R}^n.
  $$
  Thus, it yields that
  $$
  (1-\epsilon)\|\mathbf{A}(\mathbf{R}^{-1}\mathbf{x})\|_2^2 \leq\|\tilde{\mathbf{A}} (\mathbf{R}^{-1}\mathbf{x}) \|_2^2 \leq(1+\epsilon)\|\mathbf{A}(\mathbf{R}^{-1}\mathbf{x})\|_2^2, \quad \text { for all } \mathbf{x} \in \mathbb{R}^n,
  $$
  i.e.,
  $$
  (1-\epsilon) \|\mathbf{A} \mathbf{R}^{-1} \mathbf{x} \|_2^2 \leq\|(\mathbf{Q}\mathbf{R})\mathbf{R}^{-1} \mathbf{x}\|_2^2 \leq(1+\epsilon) \|\mathbf{A} \mathbf{R}^{-1}\mathbf{x}  \|_2^2.
  $$
  Here $\mathbf{Q}$ is the matrix with orthogonal columns obtained from the QR decomposition of $\tilde{\mathbf{A}}$, that is, $\tilde{\mathbf{A}}=\mathbf{QR}$.
  By the Pythagorean theorem, we obtain
  $$
  (1-\epsilon)\|\mathbf{A}\mathbf{R}^{-1}\mathbf{x}\|^2_2\leq \|\mathbf{x}\|^2_2\leq (1+\epsilon)\|\mathbf{A}\mathbf{R}^{-1}\mathbf{x}\|^2_2, 
  $$
  which can be further implied as
  $$
\frac{1}{1+\epsilon}\|\mathbf{x}\|_2^2 \leq\|\mathbf{A}\mathbf{R}^{-1} \mathbf{x}\|_2^2 \leq \frac{1}{1-\epsilon}\|\mathbf{x}\|_2^2.
$$
Dividing both sides by $\|\mathbf{x}\|_2^2$ leads to
$$
\frac{1}{1+\epsilon} \leq \frac{\|\mathbf{A}\mathbf{R}^{-1}\mathbf{x}\|_2^2}{\|\mathbf{x}\|_2^2} \leq \frac{1}{1-\epsilon}, \quad \text { for all } \mathbf{0}\neq \mathbf{x} \in \mathbb{R}^n.
$$
Furthermore, the variational definition of singular values leads to
$$
\kappa^2(\mathbf{A} \mathbf{R}^{-1})=\frac{\sigma_{\max }^2(\mathbf{A}\mathbf{R}^{-1})}{\sigma_{\min }^2(\mathbf{A}\mathbf{R}^{-1})} \leq \frac{1+\epsilon}{1-\epsilon} .
$$
Therefore, the conclusion holds.
\par
(3) Since the matrix $\mathbf{A}$ is of full column rank and the upper triangular matrix $\mathbf{R}$ is invertible, it follows that 
$$
rank(\mathbf{A}\mathbf{R}^{-1})\leq \min\left\{rank(\mathbf{A}), rank(\mathbf{R}^{-1})\right\}=n.
$$
From the generalized elementary transformations of matrices, it is demonstrated that
\begin{align*}
rank(\mathbf{A}\mathbf{R}^{-1})+n&=rank(\mathbf{A}\mathbf{R}^{-1})+rank(\mathbf{I}_n) \\
  & =rank\left(\left(\begin{array}{cc}
  \mathbf{A}\mathbf{R}^{-1} & \mathbf{0} \\
  \mathbf{0} & \mathbf{I}_n
  \end{array}\right)\right) 
   =rank\left(\left(\begin{array}{cc}
  \mathbf{A}\mathbf{R}^{-1} & \mathbf{0} \\
  \mathbf{0} & \mathbf{I}_n
  \end{array}\right)\cdot\left(\begin{array}{cc}
  \mathbf{I}_n & \mathbf{0} \\
  \mathbf{R}^{-1} & \mathbf{I}_n
  \end{array}\right)\right) \\
  & =rank\left(\left(\begin{array}{cc}
  \mathbf{A}\mathbf{R}^{-1} & \mathbf{0} \\
  \mathbf{R}^{-1} & \mathbf{I}_n
  \end{array}\right)\right)  =rank\left(\left(\begin{array}{cc}
  \mathbf{I}_m & -\mathbf{A} \\
  \mathbf{0} & \mathbf{I}_n
  \end{array}\right) \cdot\left(\begin{array}{cc}
  \mathbf{A}\mathbf{R}^{-1} & \mathbf{0} \\
  \mathbf{R}^{-1} & \mathbf{I}_n
  \end{array}\right)\right) \\
  & =rank\left(\left(\begin{array}{cc}
  \mathbf{0} & -\mathbf{A} \\
  \mathbf{R}^{-1} & \mathbf{I}_n
  \end{array}\right)\right)  \geq rank(\mathbf{A})+rank(\mathbf{R}^{-1}).
\end{align*}
From the above inequality, it is readily apparent that 
  $$
  rank(\mathbf{A}\mathbf{R}^{-1})\geq rank(\mathbf{A})+rank(\mathbf{R}^{-1})-n=n.
  $$
  Thus $rank(\mathbf{A}\mathbf{R}^{-1})=n$, and \eqref{pre_least_squares} admits a unique solution.
\end{proof}

The next theorem provides the error estimate for the CSQRP-LSQR method.
\begin{theorem}\label{CSQR-PLSQR_convergence}
  Under the same conditions as in Theorem \ref{theorem:3.2}, and letting $\mathbf{x}_{\star}$ and $\mathbf{y}_{\star}$  be the least squares solutions to \eqref{least-squares problem} and \eqref{pre_least_squares} respectively, where $\mathbf{x}_{\star}=\mathbf{A}^{\dagger}\mathbf{b}$ and $\mathbf{y}_{\star}=(\mathbf{A}\mathbf{R}^{-1})^{\dagger}\mathbf{b}$. Given a tolerance $\tau>0$, when the LSQR method with initial guess $\mathbf{y}^{(0)}=\mathbf{0}$ is used to solve \eqref{pre_least_squares}, if the numbers of iteration steps satisfies
  \begin{equation}\label{inequality:3.4}
    k\geq \frac{ln(2)+\left| ln(\tau)\right|}{\left|ln(\epsilon)\right|},
  \end{equation}
  then  
  \begin{equation}\label{convergence_rate_1}
    \|\mathbf{y}^{(k)}-\mathbf{y}_{\star}\|_{\mathbf{B}^T\mathbf{B}}\leq \tau \|\mathbf{y}_{\star}\|_{\mathbf{B}^T\mathbf{B}}
  \end{equation}
  holds with a probability of at least $1-\delta$, where $\mathbf{B}=\mathbf{A}\mathbf{R}^{-1}$ and $\mathbf{y}^{(k)}$ represents the approximate solution at the $k$-th iteration of the LSQR method. Let $\mathbf{x}^{(k)}=\mathbf{R}^{-1}\mathbf{y}^{(k)}$ be an approximate solution to the original least squares problem \eqref{least-squares problem}. Since $\mathbf{x}_{\star}=\mathbf{R}^{-1}\mathbf{y}_{\star}$, it follows that inequality \eqref{convergence_rate_1} is also equivalent to 
  $$
  \|\mathbf{x}^{(k)}-\mathbf{x}_{\star}\|_{\mathbf{A}^T\mathbf{A}}\leq \tau\|\mathbf{x}_{\star}\|_{\mathbf{A}^T\mathbf{A}}.
  $$
  If the initial guess $\mathbf{y}^{(0)}=\mathbf{Q}^T\mathbf{S}\mathbf{b}$ is used in Step 4 of the CSQRP-LSQR method. Then the following error estimate
  \begin{equation}\label{convergence_rate_2}
    \|\mathbf{A}\mathbf{x}^{(k)}-\mathbf{b}\|_2\leq \sqrt{2\left(\frac{4\tau^2}{(1-\epsilon)}+1\right)}\|\mathbf{A}\mathbf{x}_{\star}-\mathbf{b}\|_2
  \end{equation}
  holds with a probability of at least $1-\delta$.
\end{theorem}
\begin{proof}
  It is known that the convergence rate of the LSQR method is determined by the condition number of $\mathbf{B}^T\mathbf{B}$ as
  \begin{equation}\label{equation:3.5}
    \frac{\|\mathbf{y}^{(k)}-\mathbf{y}_{\star}\|_{\mathbf{B}^T \mathbf{B}}}{\|\mathbf{y}^{(0)}-\mathbf{y}_{\star}\|_{\mathbf{B}^T\mathbf{B}}}\leq 2 \left(\frac{\sqrt{\kappa(\mathbf{B}^T\mathbf{B})}-1}{\sqrt{\kappa(\mathbf{B}^T\mathbf{B})}+1}\right)^k.
  \end{equation}
  According to Theorem \ref{theorem:3.2}, it follows that
  $$
  \kappa(\mathbf{B}^T\mathbf{B})\leq \frac{1+\epsilon}{1-\epsilon}.
  $$
  Hence, we have
  \begin{equation}\label{equation:3.6}
    \begin{aligned}
      \frac{\sqrt{\kappa(\mathbf{B}^T\mathbf{B})}-1}{\sqrt{\kappa(\mathbf{B}^T\mathbf{B})}+1}&\leq \frac{\sqrt{1+\epsilon}-\sqrt{1-\epsilon}}{\sqrt{1+\epsilon}+\sqrt{1-\epsilon}}\\
      &=\frac{(\sqrt{1+\epsilon}-\sqrt{1-\epsilon})\cdot(\sqrt{1+\epsilon}+\sqrt{1-\epsilon})}{(\sqrt{1+\epsilon}-\sqrt{1-\epsilon})^2}\\
      &=\frac{\epsilon}{1+\sqrt{1-\epsilon^2}}\leq \epsilon.
    \end{aligned}
  \end{equation}
  A combination of inequalities \eqref{equation:3.5} and \eqref{equation:3.6} leads to
  \begin{equation}\label{equation:3.8}
    \|\mathbf{y}^{(k)}-\mathbf{y}_{\star}\|_{\mathbf{B}^T\mathbf{B}}\leq 2\epsilon^k \|\mathbf{y}_{\star}\|_{\mathbf{B}^T\mathbf{B}}.
  \end{equation}
  Therefore, \eqref{equation:3.8} implies that inequality \eqref{convergence_rate_1} holds, i.e.,
  \begin{equation}\label{inequality:3.9}
    \tau \geq 2 \epsilon^k \Longleftrightarrow ln(\tau)\geq ln(2)+k \cdot ln(\epsilon).
  \end{equation}
  Since $\epsilon, \tau \in (0, 1)$, it follows from \eqref{inequality:3.9} that if the number of iterations $k$ satisfies \eqref{inequality:3.4}, then the error estimate \eqref{convergence_rate_1} is necessarily satisfied.
  \par
  Set $\hat{\mathbf{x}}=\mathbf{R}^{-1}\mathbf{Q}^T\mathbf{S}\mathbf{b}$. Based on $\mathbf{x}_{\star}=\mathbf{R}^{-1}\mathbf{y}_{\star}$, it follows that
  \begin{equation}\label{Inequality:3.10}
    \begin{aligned}
      \|\mathbf{y}^{(0)}-\mathbf{y}_{\star}\|^2_{\mathbf{B}^T\mathbf{B}}&=\|\mathbf{A}\mathbf{R}^{-1}\mathbf{y}^{(0)}-\mathbf{A}\mathbf{R}^{-1}\mathbf{y}_{\star}\|^2_2=\|\mathbf{A}\hat{\mathbf{x}}-\mathbf{A}\mathbf{x}_{\star}\|_2^2=\|(\mathbf{A}\hat{\mathbf{x}}-\mathbf{b})-(\mathbf{A}\mathbf{x}_{\star}-\mathbf{b})\|^2_2\\
      & \leq \|\mathbf{A}\hat{\mathbf{x}}-\mathbf{b}\|^2_2+ 2\left|(\mathbf{A}\hat{\mathbf{x}}-\mathbf{b})^T\cdot(\mathbf{A}\mathbf{x}_{\star}-\mathbf{b})\right|+\|\mathbf{A}\mathbf{x}_{\star}-\mathbf{b}\|^2_2\\
      &\leq 2(\|\mathbf{A}\hat{\mathbf{x}}-\mathbf{b}\|^2_2+\|\mathbf{A}\mathbf{x}_{\star}-\mathbf{b}\|^2_2).
    \end{aligned}
  \end{equation}
Notice the fact that both $\mathbf{a}$ and $\mathbf{b}$ belong to $\mathbb{R}^n$. Then, $\mathbf{a}^T \cdot \mathbf{b} \leq\|\mathbf{a}\|_2\|\mathbf{b}\|_2 \leq \frac{1}{2}\left(\|\mathbf{a}\|_2^2+\|\mathbf{b}\|_2^2\right)$. Therefore, the last inequality in \eqref{Inequality:3.10} holds. Since $\hat{\mathbf{x}}$ is a solution to $\tilde{\mathbf{A}}^T\tilde{\mathbf{A}}\hat{\mathbf{x}}=\tilde{\mathbf{A}}^T\mathbf{Sb}$, $\hat{\mathbf{x}}$ is the optimal solution to the sketched least squares problem $\underset{\mathbf{x}\in\mathbb{R}^n}{\min}\|\tilde{\mathbf{A}}\mathbf{x}-\mathbf{Sb}\|_2$. 
According to Definition \ref{subspace_def}, it is evident that inequality \eqref{L2_subspaace embedding} is equivalent to
  \begin{equation}\label{L2_subspace_new_inequality}
    (1-\epsilon)\|\mathbf{x}\|^2_2 \leq \|\mathbf{S}\mathbf{x}\|^2_2\leq (1+\epsilon)\|\mathbf{x}\|^2_2, \quad \forall \mathbf{x} \in  range(\mathbf{A}).
  \end{equation}
  Since $\mathbf{x}_{\star}=\mathbf{A}^{\dagger} \mathbf{b}$, it follows that $\mathbf{b}=\mathbf{A}\mathbf{x}_{\star}$, and thus $\mathbf{b}\in range(\mathbf{A})$, where $range(\cdot)$ denotes the column space of the corresponding matrix. 
 Thus $\mathbf{A}\hat{\mathbf{x}}-\mathbf{b}\in range(\mathbf{A})$, and 
\begin{equation}\label{Inequality:3.11}
  \begin{aligned}
    \|\mathbf{A}\hat{\mathbf{x}}-\mathbf{b}\|^2_2&\leq \frac{1}{(1-\epsilon)}\|\mathbf{S}\mathbf{A}\hat{\mathbf{x}}-\mathbf{S}\mathbf{b}\|^2_2 \leq\frac{1}{(1-\epsilon)}\|\mathbf{S}\mathbf{A}\mathbf{x}_{\star}-\mathbf{S}\mathbf{b}\|^2_2 \leq \frac{(1+\epsilon)}{(1-\epsilon)}\|\mathbf{A}\mathbf{x}_{\star}-\mathbf{b}\|^2_2,
  \end{aligned}
\end{equation}
where the first and last inequalities are obtained from \eqref{L2_subspace_new_inequality}, while the second inequality is due to $\hat{\mathbf{x}}$ minimizing $\|\mathbf{SAx}-\mathbf{Sb}\|_2$.
Hence, based on inequalities \eqref{Inequality:3.10} and \eqref{Inequality:3.11}, we conclude that
\begin{equation}\label{Inequality:3.12}
  \begin{aligned}
    \|\mathbf{y}^{(0)}-\mathbf{y}_{\star}\|^2_{\mathbf{B}^T\mathbf{B}}&\leq 2 \left(1 + \frac{1+\epsilon}{1-\epsilon}\right)\|\mathbf{A}\mathbf{x}_{\star}-\mathbf{b}\|^2_2 = \frac{4}{(1-\epsilon)}\|\mathbf{A}\mathbf{x}_{\star}-\mathbf{b}\|^2_2.
  \end{aligned}
\end{equation}
From the above proof, it is evident that when the number of iterations $k$ satisfies condition \eqref{inequality:3.4}, the following inequality holds
\begin{equation}\label{Inequality:3.13}
  \|\mathbf{y}^{(k)}-\mathbf{y}_{\star}\|^2_{\mathbf{B}^T\mathbf{B}}\leq \tau^2\|\mathbf{y}^{(0)}-\mathbf{y}_{\star}\|^2_{\mathbf{B}^T\mathbf{B}}.
\end{equation}
Thus, from \eqref{Inequality:3.12} and \eqref{Inequality:3.13} , it can be concluded that 
$$
\|\mathbf{y}^{(k)}-\mathbf{y}_{\star}\|^2_{\mathbf{B}^T\mathbf{B}}\leq \frac{4 \tau^2}{(1-\epsilon)}\|\mathbf{A}\mathbf{x}_{\star}-\mathbf{b}\|^2_2.
$$
We note that $\|\mathbf{x}^{(k)}-\mathbf{x}_{\star}\|_{\mathbf{A}^T\mathbf{A}}^2=\|\mathbf{y}^{(k)}-\mathbf{y}_{\star}\|^2_{\mathbf{B}^T\mathbf{B}}$, which leads to the following error estimate
$$
\begin{aligned}
  \|\mathbf{A}\mathbf{x}^{(k)}-\mathbf{b}\|_2^2&=\|\mathbf{A}\mathbf{x}^{(k)}-\mathbf{A}\mathbf{x}_{\star}+\mathbf{A}\mathbf{x}_{\star}-\mathbf{b}\|_2^2\leq 2(\|\mathbf{A}\mathbf{x}^{(k)}-\mathbf{A}\mathbf{x}_{\star}\|^2_2+\|\mathbf{A}\mathbf{x}_{\star}-\mathbf{b}\|^2_2)\\
  & =2(\|\mathbf{x}^{(k)}-\mathbf{x}_{\star}\|_{\mathbf{A}^T\mathbf{A}}^2+\|\mathbf{A}\mathbf{x}_{\star}-\mathbf{b}\|^2_2) \leq 2\left(\frac{4\tau^2}{(1-\epsilon)}+1\right)\|\mathbf{A}\mathbf{x}_{\star}-\mathbf{b}\|^2_2,
\end{aligned}
$$
which completes the proof.
\end{proof}
\par
From Theorem \ref{theorem:3.1}, it is known that if $\mathbf{A}$ is a rank-deficient matrix, then the sketched matrix $\mathbf{S}\mathbf{A}$ is also rank-deficient, and so is the upper triangular matrix $\mathbf{R}$ obtained from the QR decomposition of $\mathbf{S} \mathbf{A}$. This prevents the preconditioning step in the CSQRP-LSQR method. A common approach for this issue is applying a threshold, treating all smaller singular values as zero (truncated SVD). In LAPACK, this threshold is the largest singular value of $\mathbf{A}$ multiplied by a user-supplied constant, referred to as RCOND. Thus, to effectively address the limitations of the CSQRP-LSQR method in handling rank-deficient or approximately rank-deficient least squares problems, we perform a truncated SVD on $\mathbf{S}\mathbf{A}$ to obtain the preconditioner. This leads to the second method, the CSSVDP-LSQR method; see Algorithm \ref{Alg:CSSVD-PLSQR} for details. It employs the same approach to compute the preconditioner, the same as that in the LSRN algorithm \cite{MSM}.

\par
\begin{algorithm}[H]
  \caption{The CSSVDP-LSQR method}\label{Alg:CSSVD-PLSQR}
  \KwIn {$\mathbf{A}\in \mathbb{R}^{m\times n},\mathbf{b}\in \mathbb{R}^n$, $\mathbf{y}^{(0)}\in \mathbb{R}^n$}
	\KwOut {Approximate solution {$\mathbf{x}^{(k)}$}}
  Choose an oversampling factor $\gamma>1$ and set $s=\gamma n$ ($s<m$).\\
  Generate a $s$-by-$m$ count sketch matrix $\mathbf{S}$ and compute $\tilde{\mathbf{A}}=\mathbf{S}\mathbf{A}$.\\
  Compute the compact SVD $\mathbf{U}\mathbf{\mathbf{\Sigma}}\mathbf{V}^T$ of matrix $\tilde{\mathbf{A}}$, where $r=\operatorname{rank}(\tilde{\mathbf{A}}), \mathbf{U} \in \mathbb{R}^{s \times r}, \mathbf{\mathbf{\Sigma}} \in \mathbb{R}^{r \times r}$, $\mathbf{V} \in \mathbb{R}^{n \times r}$, and let $\mathbf{P}=\mathbf{V}\mathbf{\mathbf{\Sigma}}^{-1}$.\\
  Compute the preconditioned matrix $\mathbf{B}=\mathbf{A}\mathbf{P}$.\\
  Use the LSQR method to compute an approximate solution $\mathbf{y}^{(k)}$ for 
  \begin{equation}\label{Pre_least_squares}
    \underset{\mathbf{y}\in \mathbb{R}^n}{\min}\|\mathbf{B}\mathbf{y}-\mathbf{b}\|_2.
  \end{equation}\\
  Return $\mathbf{x}^{(k)}=\mathbf{P}\mathbf{y}^{(k)}$.
\end{algorithm}
\begin{remark}
Similar to the LSRN method, if the preconditioner obtained in Step 3 of Algorithms \ref{Alg:CSQR-PLSQR} and \ref{Alg:CSSVD-PLSQR} is directly used as the preconditioner for the LSQR method, the methods are referred to as the CSQR-PLSQR method and the CSSVD-PLSQR method, respectively.
\end{remark}
\begin{remark}
  In exact arithmetic, the CSQRP-LSQR and CSQR-PLSQR methods and the CSSVDP-LSQR and CSSVD-PLSQR methods give identical solutions. However, in dealing with highly ill-conditioned problems in floating-point algorithms, we observe significant discrepancies in numerical experiments. Specifically, CSQR-PLSQR is less accurate than CSQRP-LSQR, and LSRN and CSSVD-PLSQR are less accurate than CSSVDP-LSQR.
\end{remark}
\par
Following \cite{MSM}, let $\mathbf{A} \in \mathbb{R}^{m \times n}$, $m \gg n$, be an approximately rank-$k$ matrix, decomposed as $\mathbf{A}_k + \mathbf{E}$ with $\mathbf{A}_k$ as the best rank-$k$ approximation of $\mathbf{A}$ and $k < n$. Assume $\sigma_1 \geq \sigma_k \gg c \sigma_1 \gg \sigma_{k+1} = \|\mathbf{E}\|_2$, where $c > 0$ is a known constant. This constant $c$ helps determine the effective rank $k$ during a truncated SVD, where the solution to the approximate or exact rank-deficiency least squares problem is $\mathbf{x}_{\star} = \mathbf{A}_k^{\dagger}\mathbf{b}$. In the CSSVDP-LSQR method, the truncated SVD of $\mathbf{S} \mathbf{A}$ uses $c$ to determine the effective rank $\tilde{k}$ of $\tilde{\mathbf{A}}$. Assuming exact arithmetic and exact solving of the preconditioned system, consistent with the error analysis for the LSRN algorithm in \cite{MSM}, we show that $\tilde{k} = k$ and derive the approximation error of $\hat{\mathbf{x}}$. According to \eqref{L2_subspaace embedding}, the following inequality holds:
\begin{equation}\label{equation:3.16}
  \frac{1}{\sqrt{1+\epsilon}} \|\mathbf{SAx}\|_2 \leq \|\mathbf{Ax}\|_2 \leq \frac{1}{\sqrt{1-\epsilon}} \|\mathbf{SAx}\|_2, \quad \forall \mathbf{x} \in \mathbb{R}^n.
\end{equation}
\begin{theorem}\label{theorem:3.4}
 If \eqref{equation:3.16} holds, $\sigma_{k+1}<c \sigma_1 \sqrt{\frac{1-\epsilon}{1+\epsilon}}$, and $\sigma_k>c \sigma_1\left(1+\sqrt{\frac{1+\epsilon}{1-\epsilon}}\right)$, we have $\tilde{k}=k$, i.e., the CSSVDP-LSQR method determines the effective rank of $\mathbf{A}$ correctly, and
$$
\|\hat{\mathbf{x}}-\mathbf{x}_{\star}\|_2 \leq \frac{6\|b\|_2}{\left(\sigma_k \sqrt{\frac{1-\epsilon}{1+\epsilon}}\right)\left(\sigma_k \sqrt{\frac{1-\epsilon}{1+\epsilon}}-\sigma_{k+1}\right)} \cdot \sigma_{k+1},
$$
where $\hat{\mathbf{x}}=\mathbf{P}(\mathbf{AP})^{\dagger}\mathbf{b}$.
\end{theorem}
\begin{proof}
According to Theorem 4.6 in \cite{MSM} (the error estimate of the LSRN method applied to approximately rank-deficient least squares problems), we only need to replace the random normal projection matrix $\mathbf{G}$ in \cite{MSM} with the count sketch matrix $\mathbf{S}$ to complete the proof.
\end{proof}
\begin{remark}
For $\forall \mathbf{x} \in \mathbb{R}^r$, we have $\mathbf{Px}\in \mathbb{R}^n$ and $\mathbf{APx}\in range(\mathbf{A})$. Thus, following the proof of Theorem \ref{theorem:3.2}, it can similarly be derived that $\kappa(\mathbf{AP})\leq \sqrt{\frac{1+\epsilon}{1-\epsilon}}$.
\end{remark}
\section{Numerical experiments}\label{section:4}
\par
This section provides ample examples to check the performance of the CSQRP-LSQR and CSSVDP-LSQR methods. These least squares problems are ill-conditioned, chosen from the RFM for solving two-dimensional and three-dimensional PDEs, or have infinite condition numbers or rank deficiencies, selected from the SuiteSparse Matrix Collection \cite{DavisHu}.

\subsection{Setup and notation}\label{Subsection:4.1}
All experiments are conducted on an Ubuntu 22.04 operating system equipped with two 64-core AMD 7763 CPU processors. Two-dimensional PDE problems are implemented with Python 3.8, and three-dimensional PDEs are implemented with Eigen \cite{GueJac}, a C++ template library for linear algebra. We compare the CSQRP-LSQR and CSSVDP-LSQR methods with available methods for sparse systems:
\begin{itemize}
  \item CSQR-PLSQR and CSSVD-PLSQR: Count sketch + QR or SVD + LSQR without explicitly computing the preconditioned matrix.
  \item  Householder QR (HHQR) algorithm \cite{Davis} for solving least squares problems.
  \item  SuiteSparseQR (SPQR) \cite{Davis_SPQR}, which utilizes multifrontal sparse QR factorization and serves as a state-of-the-art direct solver for sparse systems.
  \item Randomized iterative methods, including GRK \cite{BaiWu_4}, GRCD\cite{BaiWu_5}, and their variants GBK \cite{NiuZhe}, FGBK \cite{XiaoYin}, AMDCD \cite{ZhangGuo}, and FBCD \cite{ChHu}. We choose the same setup as in \cite{NiuZhe} for $\eta$ in the GBK method and set $p=2$ and $\eta=0.05$ in the FGKB method \cite{XiaoYin}.
  \item  Krylov subspace iterative solvers LSQR \cite{PaSa} and LSMR \cite{FoSa}.
  \item  LSRN \cite{MSM}, which processes the sketched matrix $\mathbf{G}\mathbf{A}$ serially. 
  In LSRN, LSQR is the solver for the preconditioned system in LSRN-LSQR.
\end{itemize}

In all experiments, unless noted otherwise, iterative methods begin with the initial vector $\mathbf{x}^{(0)}=\mathbf{0}$. The termination criteria for GRK, GRCD, GBK, FGBK, AMDCD, and FBCD are the relative residual (RR)
$$
\text{RR}=\frac{\|\mathbf{b}-\mathbf{A}\mathbf{x}^{(k)}\|_{2}}{\|\mathbf{b}\|_{2}}<10^{-8},
$$ or the total number of iteration steps exceeds 200,000, or the computation time surpasses 18,000 seconds. For the latter two cases, the corresponding CPU time or the number of iterations is denoted as ``$--$''. In LSQR and LSMR, $\tau$ is set to be $10^{-8}$ and, unless otherwise specified, the maximum number of iterations is $n$, the number of columns of $\mathbf{A}$. To ease the comparison of different methods, we introduce the following notations:
\begin{itemize}
  \item $\kappa(\cdot)$  denotes the condition number of a matrix.
  \item PCPU denotes the CPU time of the preprocessing part in a method.  
  \item CPU and IT denote the CPU time elapsed and the number of iterations needed, respectively. 
  \item TCPU denotes the total time elapsed by a method.
  \item $\frac{\|\mathbf{r}_k\|_2^2}{\|\mathbf{b}\|_2^2}$ denotes the relative least squares error.
  \item $\mathbf{A}\mathbf{R}^{-1}$ and $\mathbf{AP}$ denote the preconditioned matrices in CSQRP-LSQR and CSSVDP-LSQR, respectively.
\end{itemize}
We use the relative $L^2$  error to measure the difference between approximate solutions and exact solutions for PDE problems.

\subsection{Condition number, oversampling factor, randomized and Krylov methods}
In this subsection, we experimentally determine a reasonable value for the oversampling factor $\gamma$ for the CSQRP-LSQR and CSSVDP-LSQR methods and assess how different values of $\gamma$ affect the  CSQRP-LSQR, CSSVDP-LSQR, CSQR-PLSQR, and CSSVD-PLSQR methods, including the condition number of the preconditioned matrix, accuracy, and convergence rate of the algorithms discussed in Section \ref{subsubsection:4.2.1}. Results of the randomized iterative methods, GRK, GRCD, and their variants, as well as Krylov subspace methods LSQR and LSMR, are presented in Section \ref{subsubsection:4.2.2}.

To ease the description, we put the detailed setup of hyper-parameters, sizes, and condition numbers for different problems in Table \ref{tab:Numerical_parameters_table}, Appendix \ref{my_hyperparameters_RFM}.

\subsubsection{Condition number and oversampling factor}\label{subsubsection:4.2.1}
As stated in Theorem \ref{theorem:3.2}, the condition number $\kappa(\cdot)$ of the preconditioned least squares problem is typically bounded by $\sqrt{\frac{1+\epsilon}{1-\epsilon}}$ when $s=\mathcal{O}\left(n^2\right)$. Practically, the dependence of $s$ on $n$ may not be optimal. Therefore, we conduct experiments to determine a reasonable $\gamma$ for the count sketch matrix. In addition, we compare the accuracy of randomized iterative methods that explicitly compute the preconditioned matrix with those that do not.

\begin{example}\label{example:4.1}
Consider a two-dimensional elasticity problem of the following form
\begin{equation}
  \begin{cases}-\operatorname{div}(\boldsymbol{u}(\mathbf{x}))=\boldsymbol{B}(\mathbf{x}), & \mathbf{x} \in \Omega, \\ \mathbf{\sigma}(\boldsymbol{u}(\mathbf{x})) \cdot \mathbf{n}=N(\mathbf{x}), & \mathbf{x} \in \Gamma_N, \\ \boldsymbol{u}(\mathbf{x}) =\boldsymbol{U}(\mathbf{x}), & \mathbf{x} \in \Gamma_D,\end{cases}
  \end{equation}
where $\mathbf{\sigma}: \mathbb{R}^2 \rightarrow \mathbb{R}^2$ is the stress tensor induced by the displacement field $\boldsymbol{u}: \Omega \rightarrow \mathbb{R}^2, \boldsymbol{B}$ is the body force over $\Omega$, $\boldsymbol{N}$ is the surface force on $\Gamma_N$, $\boldsymbol{U}$ is the displacement on $\Gamma_D$, and $\partial \Omega=\Gamma_N \cup \Gamma_D$. Firstly, following \cite{NRBD}, we consider the Timoshenko beam problem with size $L \times$ $D$, where the free end is subjected to a parabolic traction as shown in Figure \ref{fig:Timoshenko_beam}. The exact displacement field is
\begin{equation}\label{Timoshenko_beam_exact_displacement_solution}
  \left\{\begin{array}{l}
  u=-\frac{P y}{6 E I}\left[(6 L-3 x) x+(2+v)\left(y^2-\frac{D^2}{4}\right)\right], \\
  v=\frac{P}{6 E I}\left[3 v y^2(L-x)+(4+5 v) \frac{D^2 x}{4}+(3 L-x) x^2\right],
  \end{array}\right.
  \end{equation}
where $I=\frac{D^3}{12}$. Homogeneous Dirichlet boundary conditions are applied on the left boundary at $x=0$, and homogeneous Neumann boundary conditions are applied on all other boundaries. The material parameters specified are as follows: Young's modulus $E = 3 \times 10^7\mathrm{Pa}$ and Poisson's ratio $\nu = 0.3$. Dimensions are set with $D = 10$ and $L = 10$, and a shear force of $P = 1000 \mathrm{Pa}$ is applied. 
\end{example}

\begin{example}\label{example:4.2}
  Consider the Stokes flow problem
  \begin{equation}\label{Stokes flow}
    \begin{cases}-\Delta \boldsymbol{u}(\mathbf{x})+\nabla p(\mathbf{x})=\boldsymbol{f}(\mathbf{x}), & \mathbf{x} \in \Omega, \\ \nabla \cdot \boldsymbol{u}(\mathbf{x})=0, & \mathbf{x} \in \Omega, \\ \boldsymbol{u}(\mathbf{x})=\boldsymbol{U}(\mathbf{x}), & \mathbf{x} \in \partial \Omega .\end{cases}
    \end{equation}
    In this problem, $p$ is only determined up to a constant. To avoid difficulties, we fix the value of $p$ at the left-bottom corner.
    We consider \eqref{Stokes flow} with an explicit solution
    and inhomogeneous boundary, where $\Omega$ is the square $(0,1)\times (0,1)$ with three holes centered at $(0.5,0.2),(0.2,0.8),(0.8,0.8)$ of radius 0.1. The exact displacement field for the Stokes flow is given by 
    \begin{equation}\label{Stokes_flow_solution}
      \left\{\begin{array}{l}
      u=x+x^2-2 x y+x^3-3 x y^2+x^2 y, \\
      v=-y-2 x y+y^2-3 x^2 y+y^3-x y^2, \\
      p=x y+x+y+x^3 y^2-\frac{4}{3} .
      \end{array}\right.
    \end{equation}
  \end{example}

\begin{example}\label{example:4.3}
Consider the two-dimensional elasticity problem over the domain $\Omega$ characterized as a square $(0,8) \times (0,8)$, containing 40 holes with radii ranging from 0.3 to 0.6. It is important to note the presence of a cluster of closely spaced holes, highlighted in the inset. The exact displacement field for the two-dimensional elasticity problem with complex
geometry shown in \ref{fig:complex_geometry_domain_elasticity} is 
\begin{equation}\label{two-dimensional_elasticity_problem_exact_displacement_field}
  \begin{aligned}
    & u(x, y)=\frac{1}{10} y((x+10) \sin y+(y+5) \cos x), \\
    & v(x, y)=\frac{1}{60} y\left((30+5 x \sin (5 x))\left(4+e^{-5 y}\right)-100\right) .
    \end{aligned}
\end{equation}
Dirichlet boundary condition is imposed on the lower boundary $y=0$, while Neumann boundary conditions are applied to the remaining boundaries and the internal holes. The material parameters specified are: Young's modulus $E=3 \times 10^7 \mathrm{~Pa}$ and Poisson's ratio $v=0.3$.
\end{example}

We use ill-conditioned least squares problems obtained from applying the RFM for solving Examples \ref{example:4.1}, \ref{example:4.2}, and \ref{example:4.3} to illustrate the effect of the oversampling factor $\gamma$, where $\gamma \in \left\{2, 2.5, 3, 3.5, 4\right\}$. Table \ref{tab:Numerical_parameters_table} lists these problems with condition numbers of $9.1689 \times 10^{13}$, $8.2740 \times 10^{17}$, $1.8662 \times 10^{13}$, and $5.5545 \times 10^{12}$, and corresponding sizes of $139,000 \times 10,000$, $169,564 \times 14,400$, $122,238 \times 9,000$, and $140,060 \times 10,800$. The condition numbers of $\mathbf{A} \mathbf{R}^{-1}$ and $\mathbf{AP}$ are recorded in Tables \ref{Timoshenko_beam_table}, \ref{complex_geometric_2D_Stokes_flow_problem}, and \ref{complex_elasticity_table}, with ranges of $3.5473\sim6.8624$ and $3.4068\sim6.2002$, $4.0634\sim6.5427$ and $3.9192\sim6.3172$, $3.7788\sim6.4491$ and $3.7711\sim6.3951$, respectively. In Table \ref{CSSVD-PLSQR_QR_SVD_for_Stokes_flow}, the condition numbers of $\mathbf{A P}$ range from $2.8009$ to $5.6533$. These data indicate that even with a linear sampling size (although the theoretical bound is quadratic), our methods can successfully transform ill-conditioned problems into well-conditioned ones during the preconditioning stage.

\begin{figure}[!htbp]
  \renewcommand\figurename{Figure}
    \centering
    \includegraphics[scale=0.6]{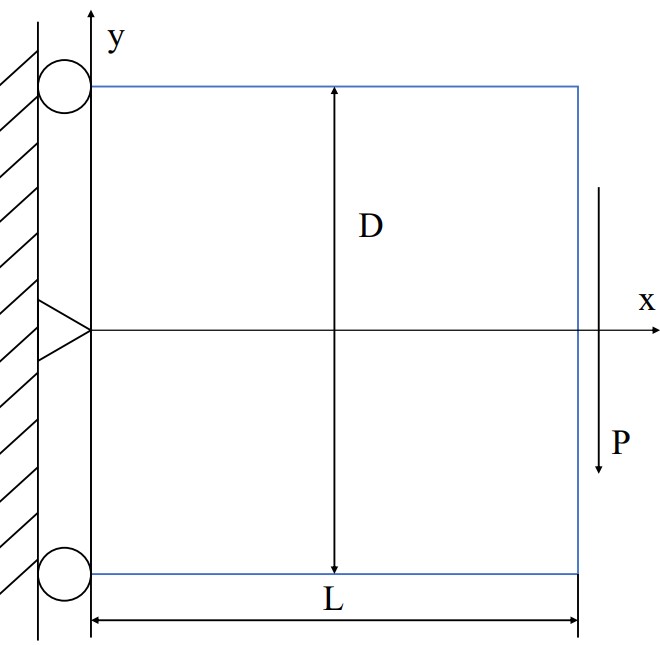}
  \caption{The Timoshenko beam problem.}
  \label{fig:Timoshenko_beam}
\end{figure}
\begin{figure}[!htbp]
  \renewcommand\figurename{Figure}
    \centering
    \includegraphics[scale=0.30]{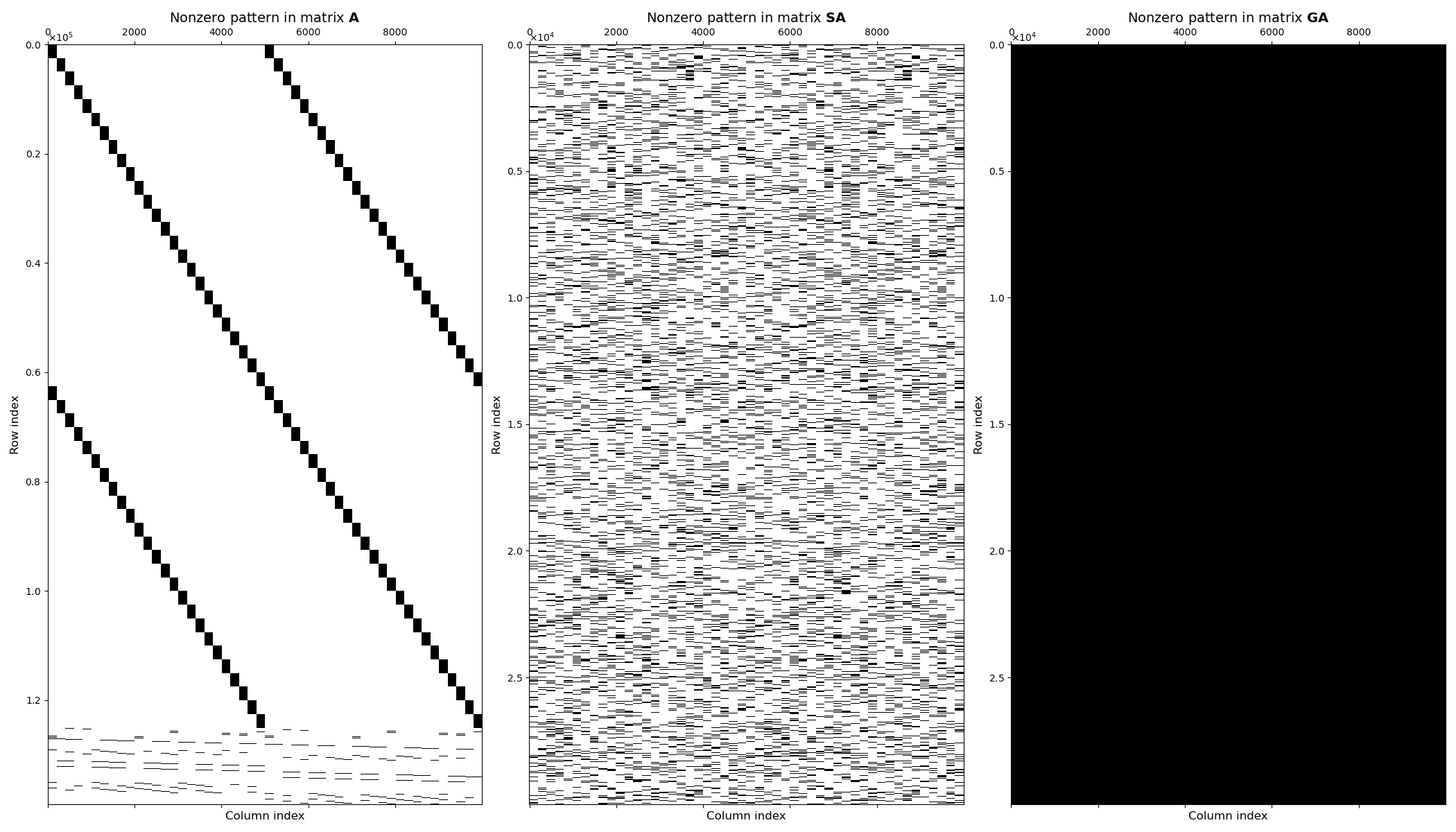}
  \caption{Sparsity pattern. Left: $\mathbf{A}$, middle: $\mathbf{S}\mathbf{A}$, right: $\mathbf{G}\mathbf{A}$.}\label{fig:count_sketch_pattern_in_A_and_B_gamma_3}
\end{figure}
\begin{figure}[!htbp]
  \renewcommand\figurename{Figure}
    \centering
    \includegraphics[scale=0.28]{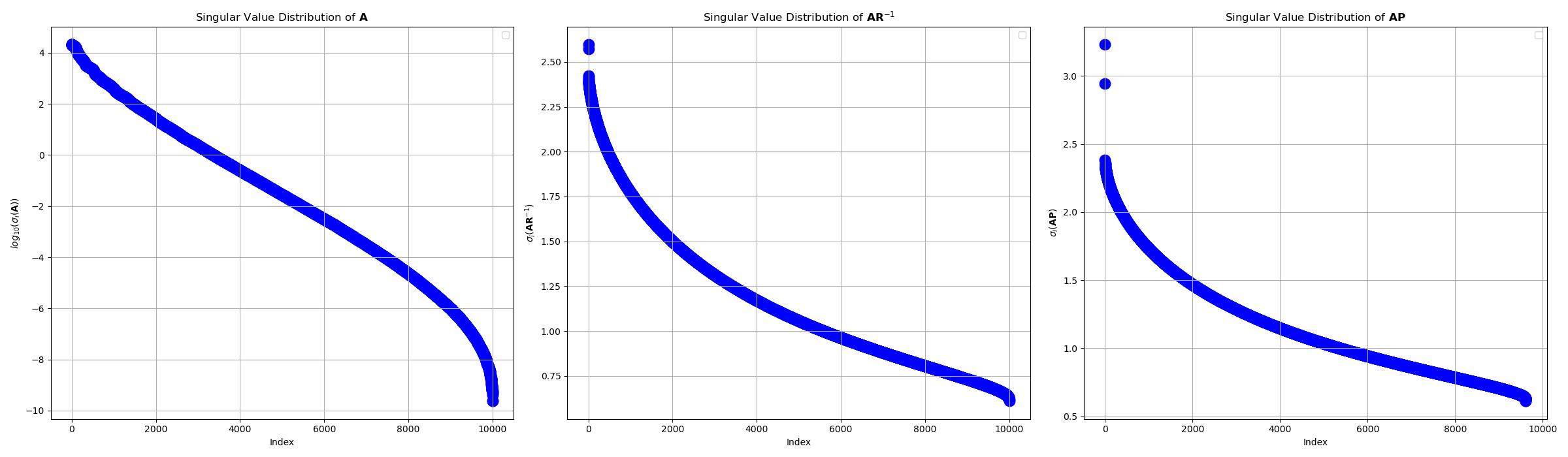}
  \caption{Singular value distribution of  $\mathbf{A}$, $\mathbf{A}\mathbf{R}^{-1}$, and $\mathbf{AP}$ for the Timoshenko beam problem.}\label{fig:count_sketch_pattern_in_A_and_B_singular_values_gamma_3}
\end{figure}

Tables \ref{Timoshenko_beam_table} and \ref{CSSVD-PLSQR_QR_SVD_for_Stokes_flow} show that the relative errors of the LSRN-LSQR and CSSVD-PLSQR methods range from $10^{-13} \sim 10^{-12}$, with the CSQRP-LSQR method achieving around $10^{-13}$ and the CSSVDP-LSQR method reaching $10^{-15}$ in Table \ref{CSSVD-PLSQR_QR_SVD_for_Stokes_flow} and $10^{-17}$ in Table \ref{Timoshenko_beam_table}. The higher accuracy of CSQRP-LSQR and CSSVDPLSQR methods is due to the explicit computation of the preconditioned matrix, enabling LSQR to effectively solve well-conditioned problems. Table \ref{complex_geometric_2D_Stokes_flow_problem} further demonstrates that methods with explicit preconditioning achieve better accuracy than those without.

\begin{table}[!htbp]
  \begin{center}
  \caption{Comparison of the LSRN-LSQR, CSQR-PLSQR, CSQRP-LSQR, CSSVD-PLSQR, and CSSVDP-LSQR methods with different $\gamma$ values for the Timoshenko beam problem.}\label{Timoshenko_beam_table}%
  \setlength{\tabcolsep}{3.50mm}  
  \begin{tabular}{lllllll}
\toprule
Method &$\gamma $ & $2$ & $2.5$ & $3$ & $3.5$ & $4$\\
\midrule
        LSRN-LSQR &PCPU  & $174.95$ & $213.08$ & $241.42$ & $270.64$ & $284.84$\\
        &CPU        & $ 62.55$ & $ 49.45$ & $53.52$ & $48.63$ & $33.75$\\
        &TCPU       & $237.50$ & $262.53$ & $294.94$ & $319.27$ & $318.59$\\
        &IT         & $81$ & $63$ & $54$ & $48$ & $ 44$\\
        &$\frac{\|\mathbf{r}_k\|_2^2}{\|\mathbf{b}\|_2^2}$ & $1.57e-12$ & $1.02e-12$ & $ 8.11e-13$ & $6.06e-13$ & $4.85e-13$\\
\midrule
CSQR-PLSQR &PCPU  & $32.82$ & $35.89$ & $37.74$ & $42.40$ & $48.34$\\
        &CPU         & $ 68.94$ & $58.65$ & $45.68$ & $ 41.17$ & $34.12$\\
        &TCPU       & $101.76$ & $94.54$ & $83.42$ & $83.57$ & $82.46$\\
        &IT         & $79$ & $66$ & $54$ & $48$ & $44$\\
        &$\frac{\|\mathbf{r}_k\|_2^2}{\|\mathbf{b}\|_2^2}$  & $7.55e-08$ & $1.37e-08$ & $2.04e-08$ & $1.63e-08$ & $1.72e-08$\\
        \midrule
 CSQRP-LSQR &$\kappa(\mathbf{A}\mathbf{R}^{-1})$ & $4.2651$ & $5.9474$ & $4.2672$ & $3.8820$ & $3.5473$\\
          &PCPU       & $50.88$ & $53.84$ & $56.08$ & $61.82$ & $65.06$\\
          &CPU        & $58.14$ & $47.58$ & $38.28$ & $ 35.55$ & $32.71$\\
          &TCPU       & $109.02$ & $101.42$ & $94.36$ & $97.37$ & $97.77$\\
          &IT         & $82$ & $67$ & $56$ & $50$ & $46$\\
            &$\frac{\|\mathbf{r}_k\|_2^2}{\|\mathbf{b}\|_2^2}$ & $5.62e-13$ & $3.29e-13$ & $2.96e-13$ & $2.72e-13$ & $1.66e-13$\\
          &$u$ error & $ 9.98e-06$ & $3.90e-06$ & $4.56e-06$ & $2.78e-06$ & $2.73e-06$\\
          &$v$ error & $ 1.35e-06$ & $5.23e-07$ & $ 2.65e-07$ & $2.23e-07$ & $2.75e-07$\\
          &$\sigma_x$ error & $1.12e-05$ & $ 7.92e-06$ & $4.67e-06$ & $5.55e-06$ & $ 3.94e-06$\\
          &$\tau_{xy}$ error & $4.13e-06$ & $3.02e-06$ & $2.69e-06$ & $ 2.75e-06$ & $2.00e-06$\\
\midrule
CSSVD-PLSQR &PCPU & $73.27$ & $79.07$ & $83.74$ & $97.79$ & $106.30$\\
                  &CPU         & $66.06$ & $ 55.22$ & $ 44.98$ & $40.97$ & $37.57$\\
                  &TCPU       & $139.33$ & $134.29$ & $128.72$ & $138.76$ & $143.87$\\
                  &IT         & $86$ & $71$ & $59$ & $53$ & $49$\\
                  &$\frac{\|\mathbf{r}_k\|_2^2}{\|\mathbf{b}\|_2^2}$  & $1.18e-12$ & $1.21e-12$ & $6.95e-13$ & $6.08e-13$ & $5.27e-13$\\
\midrule
CSSVDP-LSQR &$\kappa(\mathbf{A}\mathbf{P})$ & $5.2875$ & $5.4989$ & $4.0863$ & $3.7168$ & $3.4068$\\
            &PCPU       & $94.21$ & $101.17$ & $111.23$ & $124.36$ & $128.54$\\
            &CPU        & $37.08$ & $25.45$ & $21.83$ & $19.22$ & $17.94$\\
            &TCPU        & $131.29$ & $126.62$ & $133.06$ & $143.58$ & $146.48$\\
            &IT         & $47$ & $37$ & $32$ & $28$ & $26$\\
            &$\frac{\|\mathbf{r}_k\|_2^2}{\|\mathbf{b}\|_2^2}$ & $5.71e-17$ & $5.50e-17$ & $1.77e-17$ & $1.84e-17$ & $1.34e-17$\\
            &$u$ error & $3.12e-09$ & $ 1.01e-08$ & $6.48e-09$ & $1.31e-08$ & $6.90e-09$\\
            &$v$ error & $3.18e-09$ & $7.57e-09$ & $1.41e-09$ & $1.45e-08$ & $7.67e-09$\\
            &$\sigma_x$ error & $1.03e-08$ & $1.35e-08$ & $9.58e-09$ & $1.37e-08$ & $7.99e-09$\\
            &$\tau_{xy}$ error & $1.30e-08$ & $1.52e-08$ & $9.17e-09$ & $1.74e-08$ & $1.08e-08$\\
            \bottomrule
  \end{tabular}
  \end{center}
\end{table}
\begin{table}[!htbp]
  \begin{center}
  \caption{Comparison of the LSRN-LSQR, CSSVD-PLSQR, CSSVDP-LSQR methods with $\gamma=2,2.5,3,3.5,4$ for the Stokes flow problem.}\label{CSSVD-PLSQR_QR_SVD_for_Stokes_flow}%
  \setlength{\tabcolsep}{3.50mm}  
  \begin{tabular}{llllllll}
\toprule
Method & $\kappa(\mathbf{A}\mathbf{P})$ & TCPU & IT & $\frac{\|\mathbf{r}_k\|_2^2}{\|\mathbf{b}\|_2^2}$ & $u$ error & $v$ error & $p$ error \\
\midrule
LSRN-LSQR & $--$   & $417.04$  & $56$ & $ 1.67e-12$ & $ 3.87e-07$ & $2.26e-07$ & $3.48e-05$ \\
& $--$    & $477.06$  & $47$ & $1.23e-12$ & $2.98e-07$ & $1.80e-07$ & $4.45e-05$ \\
& $--$    & $507.78$  & $42$ & $8.23e-13$ & $2.96e-07$ & $1.54e-07 $ & $5.80e-05$ \\
& $--$    & $561.60$  & $38$ & $7.34e-13$ & $2.70e-07$ & $1.59e-07$ & $4.59e-05$ \\
& $--$    & $683.08$  & $36$ & $3.86e-13$ & $1.95e-07$ & $1.38e-07$ & $3.79e-05$ \\
\midrule
CSSVD-PLSQR & $--$   & $314.12$  & $64$ & $ 1.08e-12$ & $4.92e-07$ & $1.90e-07$ & $3.95e-05$ \\
& $--$    & $285.86$  & $49$ & $ 1.60e-12$ & $ 3.22e-07$ & $2.06e-07$ & $9.38e-05$ \\
& $--$    & $302.69$  & $53$ & $ 6.80e-13$ & $2.93e-07 $ & $1.33e-07$ & $2.89e-05$ \\
& $--$    & $307.96$  & $52$ & $5.03e-13$ & $2.45e-07$ & $1.78e-07$ & $5.04e-05$ \\
& $--$    & $314.06$  & $40$ & $ 9.39e-13$ & $3.25e-07$ & $1.47e-07$ & $2.49e-05$ \\
\midrule
CSSVDP-LSQR & $5.2884$    & $251.68$  & $33$ & $ 8.27e-15$ & $ 1.57e-08 $ & $ 8.32e-09 $ & $1.37e-05$ \\
&$3.3746$     & $257.01$  & $27$ & $3.64e-15$ & $1.26e-08$ & $ 6.61e-09$ & $7.45e-07$ \\
&$4.3693$     & $261.39$  & $26$ & $ 5.75e-15$ & $1.29e-08$ & $8.13e-09$ & $1.32e-05$ \\
&$5.6533$     & $266.60$  & $26$ & $6.42e-15$ & $1.62e-08 $ & $7.45e-09$ & $3.63e-05$ \\
&$2.8009$     & $280.19$  & $21$ & $ 2.77e-15$ & $8.31e-09$ & $5.93e-09$ & $1.64e-05$ \\
\bottomrule
\end{tabular}
\end{center}
\end{table}

\begin{figure}[!htbp]
  \renewcommand\figurename{Figure}
    \centering
    \includegraphics[scale=0.28]{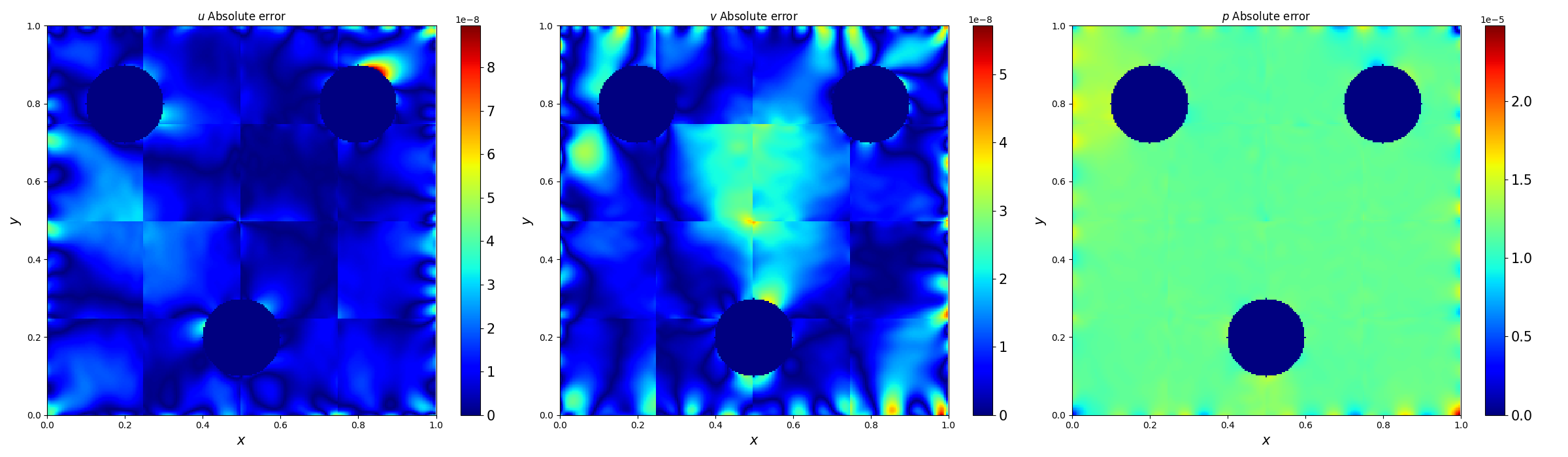}
  \caption{Distribution of absolute errors in $u$, $v$, and $p$ for two-dimensional Stokes flow problem using the CSSVDP-LSQR method with $\gamma=3$.}
  \label{fig:RFM_2dFluid_3}
\end{figure}

For all methods, as $r$ increases, PCPU gradually increases, while CPU, IT, and the condition number of the preconditioned matrix decrease. SVD-based preconditioning methods, such as the CSSVDP-LSQR method, require longer PCPU times due to the higher computational complexity of SVD compared to QR decomposition, making them more time-consuming than the CSQRP-LSQR method. However, the CSSVDP-LSQR method offers superior accuracy and requires fewer iterations; for example, its errors in $u, v, \sigma_x$, and $\tau_{x y}$ are significantly smaller than those in the CSQRP-LSQR method (see Table \ref{Timoshenko_beam_table}). The shear stress $\sigma_y$ is zero and hence omitted from the analysis. For the Stokes flow problem (see Table \ref{CSSVD-PLSQR_QR_SVD_for_Stokes_flow} ), the LSRN-LSQR and CSSVD-PLSQR methods yield errors of approximately $10^{-7}$ in $u, 10^{-7}$ in $v$, and $10^{-5}$ in $p$; whereas the CSSVDP-LSQR method produces errors of $10^{-8}$ in $u, 10^{-9}$ in $v$, and $10^{-5}$ in $p$. Hence, randomized methods that explicitly form the preconditioned matrix achieve higher accuracy, as similarly observed in Table  \ref{complex_geometric_2D_Stokes_flow_problem}. Additionally, in terms of computational efficiency, Tables  Tables \ref{Timoshenko_beam_table}, \ref{CSSVD-PLSQR_QR_SVD_for_Stokes_flow}, and \ref{complex_geometric_2D_Stokes_flow_problem} show that although these algorithms require more preprocessing time, they have fewer iterations, resulting in a shorter time in total. Table \ref{complex_elasticity_table} compares the CSQRP-LSQR and CSSVDP-LSQR methods with different $\gamma$ values for the elasticity problem over a complex geometry in Example \ref{example:4.3}.  Both methods achieve relative errors up to $10^{-11}$ and errors for $u, v, \sigma_x, \sigma_y$, and $\tau_{x y}$ generally between $10^{-6}$ and $10^{-5}$.

\begin{table}[!htbp]
  \begin{center}
  \caption{Comparison of the LSRN-LSQR, CSQR-PLSQR, CSQRP-LSQR, CSSVD-PLSQR, and CSSVDP-LSQR methods with different $\gamma$ values for the Stokes flow over a complex geometry.}\label{complex_geometric_2D_Stokes_flow_problem}%
  \setlength{\tabcolsep}{3.50mm}  
  \begin{tabular}{lllllll}
\toprule
Method &$\gamma $ & $2$ & $2.5$ & $3$ & $3.5$ & $4$\\
\midrule
          LSRN-LSQR  &PCPU  & $139.48$ & $163.26$ & $174.51$ & $195.69$ & $217.82$\\
          &CPU         & $74.58$ & $42.24$ & $ 37.61$ & $31.05$ & $29.35$\\
          &TCPU        & $213.06$ & $205.50$ & $212.12$ & $226.74$ & $247.17$\\
          &IT         & $87$ & $68$ & $58$ & $51$ & $46$\\
          &$\frac{\|\mathbf{r}_k\|_2^2}{\|\mathbf{b}\|_2^2}$ & $ 1.55e-13$ & $ 1.06e-13$ & $ 7.10e-14$ & $ 4.55e-14$ & $4.57e-14$\\
\midrule
          CSQR-PLSQR &PCPU  & $27.43$ & $33.49$ & $36.05$ & $40.90$ & $45.43$\\
          &CPU         & $69.32$ & $49.78$ & $38.33$ & $37.09$ & $35.70$\\
          &TCPU       & $96.75$ & $82.27$ & $74.38$ & $79.99$ & $81.13$\\
          &IT         & $87$ & $74$ & $63$ & $61$ & $53$\\
          &$\frac{\|\mathbf{r}_k\|_2^2}{\|\mathbf{b}\|_2^2}$  &$3.50e-10$ & $ 2.27e-10$ & $ 8.60e-11$ & $2.40e-11$ & $5.16e-11$\\
\midrule
 CSQRP-LSQR &$\kappa(\mathbf{A}\mathbf{R}^{-1})$ & $6.4527$ & $5.5251$ & $5.0128$ & $5.4155$ & $4.0634$\\
            &PCPU        & $35.63$ & $41.88$ & $45.67$ & $48.33$ & $51.03$\\
            &CPU         & $27.28$ & $25.42$ & $19.90$ & $18.19$ & $16.40$\\
            &TCPU        & $62.91$ & $67.30$ & $65.57$ & $66.52$ & $67.43$\\
            &IT         & $48$ & $40$ & $35$ & $32$ & $29$\\
            &$\frac{\|\mathbf{r}_k\|_2^2}{\|\mathbf{b}\|_2^2}$  & $3.03e-14$ & $1.20e-14$ & $8.65e-15$ & $1.31e-14$ & $7.74e-15$\\
            &$u$ error  & $1.51e-07$ & $2.19e-07$ & $1.67e-07$ & $7.16e-08$ & $8.47e-08$\\
            &$v$ error & $2.77e-08$ & $2.63e-08$ & $2.36e-08$ & $3.05e-08$ & $2.12e-08$\\
            &$p$ error  & $2.86e-05$ & $3.20e-06$ & $7.01e-06$ & $1.65e-05$ & $9.97e-06$\\
\midrule
CSSVD-PLSQR &PCPU  & $48.57$ & $56.21$ & $61.69$ & $70.95$ & $75.38$\\
            &CPU         & $53.67$ & $47.04$ & $39.91$ & $32.06$ & $ 34.68$\\
            &TCPU       & $102.24$ & $103.25$ & $101.60$ & $103.01$ & $110.06$\\
            &IT         & $88$ & $77$ & $65$ & $53$ & $56$\\
            &$\frac{\|\mathbf{r}_k\|_2^2}{\|\mathbf{b}\|_2^2}$  &$1.36e-13$ & $ 8.47e-14$ & $ 6.48e-14$ & $ 6.95e-14$ & $ 5.46e-14$\\
\midrule
CSSVDP-LSQR &$\kappa(\mathbf{A}\mathbf{P})$ & $6.3173$ & $5.2832$ & $4.9515$ & $5.3320$ & $3.9192$\\
            &PCPU        & $57.54$ & $66.32$ & $72.65$ & $83.57$ & $85.68$\\
            &CPU         & $26.11$ & $21.66$ & $19.07$ & $ 17.88$ & $16.12$\\
            &TCPU        & $83.65$ & $87.98$ & $91.72$ & $101.45$ & $101.80$\\
            &IT         & $47$ & $39$ & $34$ & $32$ & $29$\\
            &$\frac{\|\mathbf{r}_k\|_2^2}{\|\mathbf{b}\|_2^2}$   & $3.01e-14$ & $1.28e-14$ & $1.13e-14$ & $1.00e-14$ & $6.84e-15$\\
            &$u$ error  & $2.95e-08$ & $ 1.70e-08$ & $ 1.79e-08$ & $1.87e-08$ & $9.23e-09$\\
            &$v$ error & $1.76e-08$ & $ 1.20e-08$ & $1.19e-08$ & $1.09e-08$ & $5.87e-09$\\
            &$p$ error  & $1.56e-05$ & $9.70e-06$ & $8.42e-06$ & $9.92e-06$ & $5.03e-06$\\
  \bottomrule
  \end{tabular}
  \end{center}
\end{table}

Figure \ref{fig:count_sketch_pattern_in_A_and_B_gamma_3} shows the sparsity patterns of matrices related to this problem: the coefficient matrix $\mathbf{A}$ exhibits a structured sparse pattern; $\mathbf{S}\mathbf{A}$, using a count sketch matrix, maintains this sparsity; whereas $\mathbf{G}\mathbf{A}$, with a Gaussian random matrix, becomes fully dense. This demonstrates that using a count sketch matrix in large-scale sparse problems can conserve memory, unlike the Gaussian matrix used in LSRN, which leads to increased density. Moreover, when computing $\mathbf{S}\mathbf{A}$, it is not necessary to explicitly form $\mathbf{S}$.  Figure \ref{fig:count_sketch_pattern_in_A_and_B_singular_values_gamma_3} indicates that the singular values of preconditioned matrices $\mathbf{A R}^{-1}$ and $\mathbf{A P}$ are more closely clustered and have a narrower range than those of $\mathbf{A}$, suggesting better conditioning and enhanced numerical stability and convergence for methods like CSQRP-LSQR and CSSVDP-LSQR.

For the above three examples, we evaluate the preconditioning effectiveness of these methods using a count sketch matrix with row counts $s=\gamma n$ when $\gamma=2,2.5,3,3.5,4$. It is found that $\gamma=3$ yields nearly optimal total time and accuracy for most randomized iterative methods. Thus, $\gamma=3$ is used in all subsequent experiments as a generally reasonable choice.
\begin{table}[!htbp]
  \begin{center}
  \caption{Comparison of the CSQRP-LSQR and CSSVDP-LSQR methods with different $\gamma$ values for the elasticity problem over a complex geometry.}\label{complex_elasticity_table}%
  \setlength{\tabcolsep}{3.50mm}  
  \begin{tabular}{lllllll}
\toprule
Method &$\gamma $ & $2$ & $2.5$ & $3$ & $3.5$ & $4$\\
\midrule
 CSQRP-LSQR &$\kappa(\mathbf{A}\mathbf{R}^{-1})$  & $6.4491$ & $5.5522$ & $4.8243$ & $4.2486$ & $3.7788$\\
            &PCPU        & $60.68$ & $63.27$ & $65.38$ & $69.15$ & $73.83$\\
            &CPU         & $67.18$ & $50.02$ & $44.07$ & $41.14$ & $38.29$\\
            &TCPU        & $127.86$ & $113.29$ & $109.45$ & $110.29$ & $112.12$\\
            &IT          & $88$ & $69$ & $61$ & $53$ & $50$\\
            &$\frac{\|\mathbf{r}_k\|_2^2}{\|\mathbf{b}\|_2^2}$  & $5.31e-11$ & $5.23e-11$ & $5.15e-11$ & $5.16e-11$ & $5.14e-11$\\
            &$u$ error & $4.00e-05$ & $3.12e-05$ & $4.00e-05$ & $ 3.73e-05$ & $2.12e-05$\\
            &$v$ error & $3.01e-05$ & $2.12e-05$ & $2.20e-05$ & $2.13e-05$ & $ 2.32e-05$\\
            &$\sigma_x$ error & $2.77e-05$ & $2.94e-05$ & $3.26e-05$ & $2.74e-05$ & $2.32e-05$\\
            &$\sigma_y$ error & $5.78e-05$ & $6.09e-05$ & $5.93e-05$ & $5.52e-05$ & $5.27e-05$\\
            &$\tau_{xy}$ error & $8.62e-06$ & $7.37e-06$ & $7.69e-06$ & $8.83e-06$ & $7.35e-06$\\
\midrule
CSSVDP-LSQR &$\kappa(\mathbf{A}\mathbf{P})$ &$6.3951$ & $5.5258$ & $4.8064$ & $4.2287$ & $3.7711$\\
            &PCPU        & $130.85$ & $143.98$ & $150.09$ & $158.66$ & $163.64$\\
            &CPU         & $67.21$ & $55.41$ & $46.26$ & $ 39.92$ & $38.70$\\
            &TCPU        & $199.06$ & $199.39$ & $196.35$ & $198.58$ & $202.34$\\
            &IT          & $88$ & $69$ & $61$ & $52$ & $50$\\
            &$\frac{\|\mathbf{r}_k\|_2^2}{\|\mathbf{b}\|_2^2}$  & $5.20e-11$ & $5.18e-11$ & $ 5.18e-11$ & $5.17e-11$ & $5.17e-11$\\
            &$u$ error & $1.82e-05$ & $1.77e-05$ & $7.12e-06$ & $1.67e-05$ & $1.67e-05$\\
            &$v$ error & $2.30e-05$ & $2.31e-05$ & $2.25e-05$ & $2.28e-05$ & $2.28e-05$\\
            &$\sigma_x$ error & $2.05e-05$ & $ 2.04e-05$ & $2.04e-05$ & $2.05e-05$ & $2.05e-05$\\
            &$\sigma_y$ error & $5.24e-05$ & $5.23e-05$ & $ 5.23e-05$ & $5.23e-05$ & $5.24e-05$\\
            &$\tau_{xy}$ error & $7.12e-06$ & $7.11e-06$ & $7.12e-06$ & $7.12e-06$ & $7.13e-06$\\
  \bottomrule
  \end{tabular}
  \end{center}
\end{table}
\begin{figure}[!htbp]
  \renewcommand\figurename{Figure}
    \centering
    \includegraphics[scale=0.6]{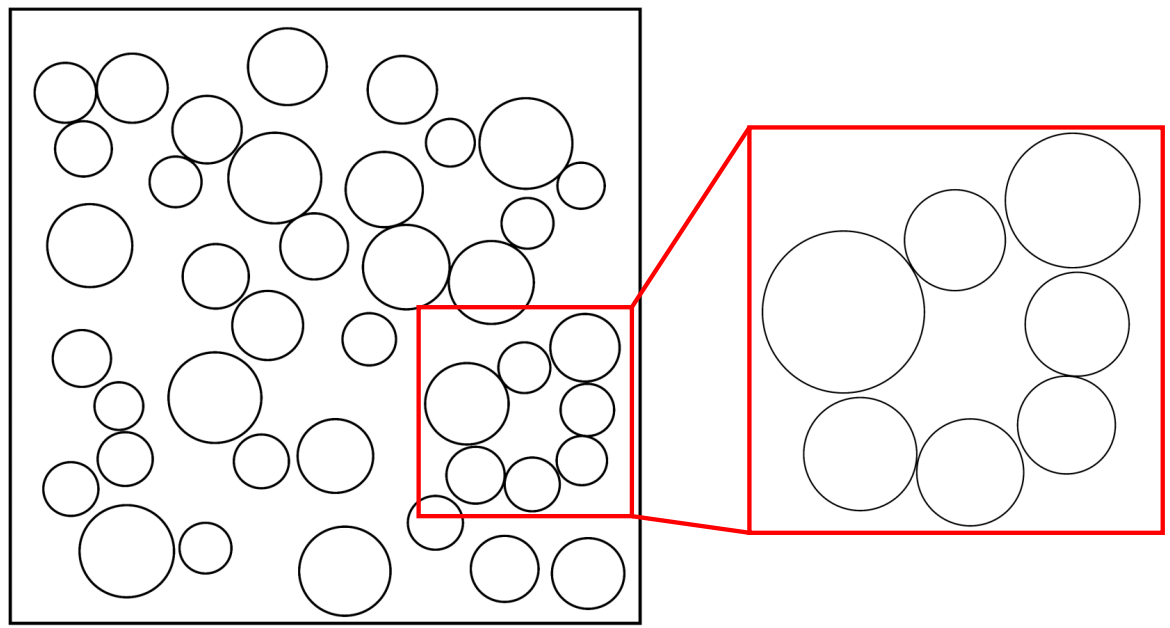}
  \caption{A two-dimensional complex geometry.}
  \label{fig:complex_geometry_domain_elasticity}
\end{figure}

\subsubsection{Results of randomized iterative and Krylov subspace methods} \label{subsubsection:4.2.2}

In this section, we demonstrate that methods including GRK, GRCD, and their variants, as well as LSQR and LSMR, are unsuitable to solve such least squares problems due to their low accuracy and long execution times. Specifically, for the elasticity problem (see Table \ref{Randomized_iterative_LSQR_LSMR_for_elasticity_numerical_results}), LSQR and LSMR methods reach their iteration limits after over 7000 seconds, with relative errors around $10^{-4}$ and errors in $u, v, \sigma_x$, and $\tau_{x y}$ around $10^{-3}$ to $10^{-2}$. Similarly, in the Stokes flow problem (see Table \ref{Randomized_iterative_LSQR_LSMR_for_Stokes_flow}), both methods hit their maximum iterations of 14,000, with running times exceeding 17,000 seconds, achieving relative errors around $10^{-7}$ and errors in $u, v$, and $p$ at $10^{-4}, 10^{-5}$, and $10^{-2}$, respectively. The numerical performance of Kaczmarz-type methods (GRK, GBK, FGBK) and coordinate descent methods (GRCD, AMDCD, FBCD) is even worse. For example, FGBK takes over 18,000 seconds with a relative error of 0.7316, while GRCD exceeds 200,000 iterations, resulting in a relative error of $2.2631 \times 10^{-3}$ and errors around $10^{-1}$.

\begin{table}[!htbp]
  \begin{center}
  \caption{Comparison of the LSQR, LSMR, GRK, GBK, FGBK, GRCD, AMDCD, and FBCD methods for the Timoshenko beam problem.}\label{Randomized_iterative_LSQR_LSMR_for_elasticity_numerical_results}%
  \setlength{\tabcolsep}{2.50mm}  
  \begin{tabular}{llllllll}
\toprule
Method & TCPU & IT & $\frac{\|\mathbf{r}_k\|_2^2}{\|\mathbf{b}\|_2^2}$ & $u$ error & $v$ error & $\sigma_x$ error & $\tau_{xy}$ error\\
\midrule
LSQR  &$7088.21$  & $10000$ & $1.70e-04$ & $8.12e-03$ & $8.12e-03$ & $7.67e-03$ & $7.51e-03$\\
LSMR  &$ 7100.33$  & $10000$ & $2.38e-04$ & $1.87e-02$ & $1.90e-02$ & $ 1.76e-02$ & $1.77e-02$\\
GRK &$>5h$  & $--$ & $2.93$ & $  1.03e+00$ & $ 1.01e+00$ & $9.65e-01$ & $9.76e-01$\\
GBK &$>5h$  & $--$ & $ 2.61$ & $1.02e+00$ & $1.01e+00$ & $9.72e-01$ & $1.01e+00$\\
FGBK &$>5h$  & $--$ & $ 0.73$ & $1.04e+00$ & $ 1.00e+00$ & $  9.74e-01$ & $ 9.28e-01$\\
GRCD &$--$  & $>200000$ & $2.26e-03$ & $ 1.03e+00$ & $ 9.63e-01$ & $ 9.49e-01$ & $ 8.28e-01$\\
AMDCD &$--$ & $>200000$  & $5001.03$ & $1.49e+01$ & $4.45e+01$ & $9.26e+01$ & $1.71e+02$ \\
FBCD &$>5h$  & $--$ & $3.29e-02$ & $1.05e+00$ & $9.74e-01$ & $ 9.66e-01$ & $8.35e-01$\\
\bottomrule
\end{tabular}
\end{center}
\end{table}

\begin{table}[!htbp]
  \begin{center}
  \caption{Comparison of LSQR, LSMR,  GRK, GBK, FGBK, GRCD, AMDCD, and FBCD methods for the Stokes flow problem.}\label{Randomized_iterative_LSQR_LSMR_for_Stokes_flow}%
  \setlength{\tabcolsep}{2.50mm}  
  \begin{tabular}{llllllll}
\toprule
Method & $\kappa(\mathbf{A}\mathbf{P})$ & TCPU & IT & $\frac{\|\mathbf{r}_k\|_2^2}{\|\mathbf{b}\|_2^2}$ & $u$ error & $v$ error & $p$ error \\
\midrule
LSQR  &$--$  & $17607.62$  & $14400$ & $3.36e-07$ & $1.05e-04$ & $ 8.09e-05$ & $2.36e-02$ \\
LSMR  &$--$  & $17735.96$  & $ 14400$ & $4.50e-07$ & $ 8.22e-05$ & $5.89e-05$ & $1.72e-02$ \\
GRK  &$--$   & $>5h$  & $--$ & $1.19e+00$ & $2.64e-01$ & $3.14e-01$ & $9.27e-01$ \\
GBK &$--$    & $>5h$ & $--$ & $2.10e-01$ & $9.82e-02$ & $9.22e-02$ & $ 9.72e-01$\\
FGBK &$--$   & $>5h$ & $--$ & $1.89e-01$ & $1.35e-01$ & $1.31e-01 $ & $1.01e+00$\\
GRCD  &$--$  & $--$  & $>2\cdot 10^5$ & $3.28e-04$ & $2.34e-03$ & $1.63e-03$ & $ 1.73e-01$ \\
AMDCD &$--$  & $--$  & $>2\cdot 10^5$ & $1.71e-02$ & $1.64e-02$ & $1.29e-02$ & $  8.62e-01$ \\
FBCD &$--$   & $>5h$ & $--$ & $6.25e-04$ & $ 9.19e-03$ & $6.92e-03$ & $ 5.07e-01$\\
\bottomrule
\end{tabular}
\end{center}
\end{table}

\subsection{Three-dimensional PDEs}\label{Subsection:4.3}

In this subsection, we further apply the randomized iterative methods LSRN-LSQR, CSQRP-LSQR, CSSVDP-LSQR, CSQR-PLSQR, and CSSVD-PLSQR, as well as the direct methods SPQR and HHQR, to solve ill-conditioned least squares problems arising from three-dimensional PDEs to evaluate their performance. Hyper-parameters for the corresponding least squares problems are listed in Table \ref{tab:hyper-parameters_three-dimensional_PDEs}, Appendix \ref{my_hyperparameters_RFM}.

\par
\begin{example}\label{example:3D Poisson}
Consider the three-dimensional Poisson equation with Dirichlet boundary condition over $\Omega=[0,1] \times[0,1]\times[0,1]$ 
\begin{equation}\label{eq:Poisson}
  \begin{cases}-\Delta u(x,y,z)=f(x,y,z), & (x, y, z) \in \Omega, 
    \\ u(x,y,z)=g(x,y,z), & (x, y, z) \in \partial \Omega,\end{cases}
\end{equation}
where the explicit solution is 
\begin{equation}\label{Poisson_ture_solution}
  \begin{aligned}
    u(x,y,z) =& \left( \cos(\pi x + \frac{2\pi}{5}) + \cos(2\pi x - \frac{\pi}{5}) + \cos(3\pi x + \frac{\pi}{10}) \right) \\
    &\cdot \left( \cos(\pi y + \frac{2\pi}{5}) + \cos(2\pi y - \frac{\pi}{5}) + \cos(3\pi y + \frac{\pi}{10}) \right) \\
    &\cdot \left( \cos(\pi z + \frac{2\pi}{5}) + \cos(2\pi z - \frac{\pi}{5}) + \cos(3\pi z + \frac{\pi}{10}) \right).
  \end{aligned}
\end{equation}
\end{example}

\begin{example}
  Consider the three-dimensional Helmholtz equation with Dirichlet boundary condition over $\Omega=[0,1] \times[0,1]\times[0,1]$ 
  \begin{equation}\label{eq:Helmholtz_eq}
    \begin{cases}\Delta u(x,y,z)+k^{2} u(x,y,z)=f(x,y,z), & (x, y, z) \in \Omega, 
      \\ u(x,y,z)=g(x,y,z), & (x, y, z) \in \partial \Omega,\end{cases}
  \end{equation}
  where $k=0.5$ is the wave number. The explicit solution is
  \begin{equation}
    \begin{aligned}
      u(x,y,z) = &-A \left(\cos\left(\pi x + \frac{2\pi}{5}\right) \cdot \cos\left(\pi y + \frac{2\pi}{5}\right) \cdot \cos\left(\pi z + \frac{2\pi}{5}\right)\right)\\
      & -B \left(\cos\left(\pi x + \frac{2\pi}{5}\right) \cdot \cos\left(2\pi y + \frac{2\pi}{5}\right) \cdot  \cos\left(4\pi z + \frac{2\pi}{5}\right)\right).
    \end{aligned}
  \end{equation}
  The following three cases are considered: the low-frequency problem when $A=1.0$ and $B=0.0$; the high-frequency problem when $A=0.0$ and $B=1.0$; the mixed-frequency problem when $A=0.5$ and $B=0.5$.
  \end{example}

Table \ref{tab:three-dimensional_Poisson_equation} presents the results for the CSQRP-LSQR and CSSVDP-LSQR methods. As $m$ and $n$ increase, the problem size grows with gradual increases in PCPU, CPU, and TCPU. Both methods converge within a reasonable number of iterations, indicating that they effectively transform ill-conditioned problems into well-conditioned ones. The relative error $\frac{\left\|\mathbf{r}_k\right\|_2^2}{\|\mathbf{b}\|_2^2}$ decreases from $1.34 \times 10^{-6}$ to $2.93 \times 10^{-14}$ for the CSQRP-LSQR method and from $1.34 \times 10^{-6}$ to $1.74 \times 10^{-14}$ for the CSSVDP-LSQR method. Better solution accuracy is achieved for larger $m$ and $n$, and both methods can solve the corresponding least squares problems with high precision. Table \ref{tab:the three-dimensional Helmholtz equation} records the numerical results of the CSQRP-LSQR method for the three-dimensional Helmholtz equation. The numerical performance of the CSQRP-LSQR method is consistent with that in Table \ref{tab:three-dimensional_Poisson_equation}. For instance, for the low-frequency case, the relative error $\frac{\left\|\mathbf{r}_k\right\|_2^2}{\|\mathbf{b}\|_2^2}$ reduces from $1.99 \times$ $10^{-8}$ to $4.56 \times 10^{-14}$ as the problem size increases. Similar trends are observed for the high- and mixed-frequency cases. Table \ref{tab:three-dimensional_Helmholtze_low_quation} compares the performance of LSRN-LSQR, CSSVD-PLSQR, and CSSVDP-LSQR methods for the three-dimensional Helmholtz equation with a low-frequency solution. Results show that only the CSSVDP-LSQR method performs well, while the other two methods, which do not explicitly compute the preconditioned matrix, exhibit significantly poorer numerical performance.

\begin{table}[!htbp]
  \begin{center}
    \caption{Numerical results of CSQRP-LSQR and CSSVDP-LSQR methods for the three-dimensional Poisson equation.}\label{tab:three-dimensional_Poisson_equation}%
  \setlength{\tabcolsep}{3.50mm}  
  \begin{tabular}{llllllll}
  \toprule
  Method &$m$ &  $92800 $ &  $ 92800 $ & $170000$  & $ 170000 $ & $280800$ & $280800$ \\
   &$n$ &  $ 3200$ &  $ 6400 $ & $6400$  & $ 9600$ & $ 12800$ & $ 16000$ \\
  \midrule 
  CSQRP-LSQR &PCPU & $18.80$ & $45.18$ & $71.08$ & $122.05$ & $279.18$ & $393.26$ \\
  &CPU  &  $2.87$ & $6.47$ & $11.25$ & $16.10$ & $29.06$ & $40.50$ \\
  &TCPU & $21.67$ & $51.65$ & $82.33$ & $138.15$ & $308.24$ & $433.76$ \\
  &IT   &  $43$ & $51$ & $49$ & $54$ & $57$ & $56$ \\
  &$\frac{\|\mathbf{r}_k\|_2^2}{\|\mathbf{b}\|_2^2}$ & $1.34e-06$ & $1.50e-09$ & $1.65e-09$ & $1.10e-11$ & $4.78e-13$ & $2.93e-14$ \\
  &$u$ error &  $1.40e-03$ & $3.58e-05$ & $4.00e-05$ & $2.94e-06$ & $4.32e-06$ & $1.34e-07$ \\
  \midrule 
  CSSVDP-LSQR &PCPU & $141.54$ & $ 709.85$ & $ 737.29$ & $1574.37$ & $2932.64$ & $4703.31$ \\
  &CPU  &   $2.85$ & $6.44$ & $10.90$ & $17.15$ & $39.57$ & $49.16$ \\
  &TCPU &  $144.39$ & $716.29$ & $748.19$ & $1591.52$ & $2972.21$ & $4752.47$ \\
  &IT   &   $42$ & $51$ & $49$ & $54$ & $58$ & $59$ \\
  &$\frac{\|\mathbf{r}_k\|_2^2}{\|\mathbf{b}\|_2^2}$ &  $1.34e-06$ & $ 1.50e-09$ & $1.65e-09$ & $1.10e-11$ & $1.78e-13$ & $1.79e-14$ \\
  &$u$ error &  $1.40e-03$ & $3.58e-05$ & $4.00e-05$ & $2.93e-06$ & $3.74e-07$ & $1.07e-07$ \\
  \bottomrule
  \end{tabular}
  \end{center}
\end{table}

\begin{table}[!htbp]
  \begin{center}
    \caption{Numerical results of CSQRP-LSQR method for the three-dimensional Helmholtz with the different solution frequencies.}\label{tab:the three-dimensional Helmholtz equation}%
  \setlength{\tabcolsep}{3.50mm}  
  \begin{tabular}{llllllll}
  \toprule
  Frequency  &$m$ &  $92800 $ &  $ 92800 $ & $280800$  & $280800$ & $431200 $ & $431200$ \\
   &$n$ &  $ 3200$ &  $ 6400 $ & $8000$  & $9600$ & $9600$ & $11200$ \\
  \midrule 
  Low  &PCPU & $17.94$ & $45.16$ & $146.11$ & $184.82$ & $264.16$ & $379.50$ \\
  &CPU  &  $2.68$ & $7.26$ & $22.72$ & $29.12$ & $45.09$ & $51.35$ \\
  &TCPU & $20.62$ & $52.42$ & $168.83$ & $213.94$ & $309.25$ & $430.85$ \\
  &IT   & $46$ & $55$ & $56$ & $59$ & $59$ & $59$ \\
  &$\frac{\|\mathbf{r}_k\|_2^2}{\|\mathbf{b}\|_2^2}$ & $1.99e-08$ & $5.75e-12$ & $2.14e-13$ & $1.73e-14$ & $1.58e-14$ & $4.56e-14$ \\
  &$u$ error &  $3.43e-05$ & $5.10e-07$ & $1.02e-07$ & $ 1.48e-07$ & $9.22e-08$ & $4.85e-07$ \\
  \midrule 
  High  &PCPU & $18.42$ & $45.23$ & $148.98$ & $183.79$ & $265.41$ & $368.58$ \\
  &CPU  &  $2.64$ & $6.77$ & $17.45$ & $26.80$ & $42.49$ & $48.35$ \\
  &TCPU & $21.06$ & $52.00$ & $166.43$ & $210.59$ & $307.90$ & $416.93$ \\
  &IT   & $42$ & $49$ & $51$ & $53$ & $53$ & $54$ \\
  &$\frac{\|\mathbf{r}_k\|_2^2}{\|\mathbf{b}\|_2^2}$ & $2.68e-06$ & $2.36e-09$ & $1.92e-10$ & $1.46e-11$ & $1.54e-11$ & $1.93e-12$ \\
  &$u$ error &  $2.91e-03$ & $7.71e-05$ & $ 2.38e-05$ & $5.76e-06$ & $ 6.20e-06$ & $3.66e-06$ \\
  \midrule 
  Mixed &PCPU & $18.36$ & $44.94$ & $149.60$ & $181.76$ & $270.91$ & $375.78$ \\
  &CPU  &  $2.69$ & $6.34$ & $16.05$ & $28.99$ & $44.06$ & $ 50.16$ \\
  &TCPU & $21.05$ & $51.28$ & $165.65$ & $210.75$ & $314.97$ & $425.94$ \\
  &IT   & $42$ & $50$ & $51$ & $53$ & $52$ & $54$ \\
  &$\frac{\|\mathbf{r}_k\|_2^2}{\|\mathbf{b}\|_2^2}$ & $2.63e-06$ & $2.31e-09$ & $ 1.88e-10$ & $1.42e-11$ & $1.50e-11$ & $1.95e-12$ \\
  &$u$ error &  $2.05e-03$ & $5.43e-05$ & $1.68e-05$ & $4.02e-06$ & $4.37e-06$ & $3.93e-06$ \\
  \bottomrule
  \end{tabular}
  \end{center}
\end{table}
\begin{table}[!htbp]
  \begin{center}
    \caption{Comparison of LSRN-LSQR, CSSVD-PLSQR, and CSSVDP-LSQR methods for the three-dimensional Helmholtz with a low-frequency solution.}\label{tab:three-dimensional_Helmholtze_low_quation}%
    \setlength{\tabcolsep}{4.50mm}  
    \begin{tabular}{lllllll}
    \toprule
    Method &  \multicolumn{2}{c}{LSRN-LSQR}  &  \multicolumn{2}{c}{CSSVD-PLSQR} &  \multicolumn{2}{c}{CSSVDP-LSQR} \\
    \cmidrule(lr){2-3} \cmidrule(lr){4-5} \cmidrule(lr){6-7}
    $m$ &$92800$ &  $92800$ & $92800$  & $92800$ & $92800$ & $92800$  \\
    $n$ & $3200$ &  $6400$ & $3200$  & $6400$ & $3200$ & $6400$  \\
    PCPU &  $262.56$ &  $1010.95$ & $121.32$  & $696.23$ & $138.42$ & $706.23$  \\
    CPU & $  75.23$ &  $ 113.41$ & $ 70.92$  & $100.45$ & $2.37$ & $6.84$  \\
    TCPU &  $337.79$ &  $1124.36$ & $192.24$  & $796.68$ & $140.79$ & $712.63$  \\
    IT & $1000$ &  $1000$ & $1000$  & $1000$ & $46$ & $55$  \\
    $\frac{\|\mathbf{r}_k\|_2^2}{\|\mathbf{b}\|_2^2}$ &  $1.00e+00$ &  $1.00e+00$ & $9.84e-01$  & $9.84e-01$ & $1.99e-08$ & $5.75e-12$  \\
    $u$ error & $1.00e+00$ &  $1.00e+00$ & $9.92e-01$  & $9.92e-01$ & $3.43e-05$ & $ 5.10e-07$  \\
    \bottomrule 
    \end{tabular}
  \end{center}
\end{table}

\begin{example}\label{3D_Elasticity}
Consider the three-dimensional elasticity problem
\begin{equation}
  \begin{cases}-\operatorname{div}(\mathbf{u}(\mathbf{x}))=\mathbf{B}(\mathbf{x}), & \mathbf{x} \in \Omega, \\ \sigma(\mathbf{u}(\mathbf{x})) \cdot \mathbf{n}=\mathbf{N}(\mathbf{x}), & \mathbf{x} \in \Gamma_N, \\ \mathbf{u}(\mathbf{x})=\mathbf{U}(\mathbf{x}), & \mathbf{x} \in \Gamma_D,\end{cases}
  \end{equation}
  where $\mathbf{\sigma}: \mathbb{R}^3 \rightarrow \mathbb{R}^3$ is the stress tensor induced by the displacement field $\mathbf{u}: \Omega \rightarrow \mathbb{R}^3, \mathbf{B}$ is the body force over $\Omega, \mathbf{N}$ is the surface force on $\Gamma_N, U$ is the displacement on $\Gamma_D$, and $\partial \Omega=\Gamma_N \cup \Gamma_D$. The exact displacement field is  
  \begin{equation}\label{three-dimensional_elasticity_problem_true_solution}
    \begin{aligned}
      u(x, y, z)&=-\frac{P y}{6 E I}\left((6 L-3 x)x+(2+\mu)\left(y^2-\frac{D^2}{4}\right)+(6 L-3 z) z\right),\\
    v(x, y, z)&=\frac{P}{6 E I}\left(3 \mu y^2(L-x+z)+(4+5 \mu) \frac{D^2 x z}{4}+(3 L-x) x z^2+z^3\right),\\
      w(x, y, z)&=-\frac{P}{6 E I}\left(3 \mu z^2(L-x+y)+(4+5 \mu) \frac{D^2 z y}{4}+(3 L-z) y z^2\right),
    \end{aligned}
  \end{equation}
  where $\Omega=[0,10] \times[0,10]\times[0,10]$, the material parameters are as follows: the Young's modulus $E=3 \times 10^7$ Pa, Poisson ratio $\mu = 0.3$, and we choose $D=10$, $L=10$, $I=\frac{D^3}{12}$ and the shear force as $P=1000$ Pa. Homogeneous Dirichlet boundary condition is applied on the lower boundary $z=0$ and Homogeneous Neumann boundary condition is applied on the other boundaries.
\end{example}
\begin{table}[!htbp]
  \begin{center}
    \caption{Numerical results of CSQRP-LSQR, CSQR-PLSQR, SPQR and HHQR methods for the three-dimensional elasticity problem.}\label{tab:three-dimensional_elasticity_problem}%
    \setlength{\tabcolsep}{3.00mm}  
    \begin{tabular}{llllllll}
    \toprule
    Method &$m$ &  $31200 $ &  $ 97200 $ & $ 278400 $  & $ 278400 $ & $510000$ & $510000$ \\
     &$n$ &  $ 1200$ &  $ 1800 $ & $14400$  & $ 19200$ & $ 24000 $ & $ 28800$ \\
    \midrule 
    CSQRP-LSQR &PCPU & $10.76$ & $55.04$ & $407.45$ & $503.25$ & $1257.97$ & $ 1499.98$ \\
    &CPU  & $0.62$ & $ 3.62$ & $44.77$ & $55.70$ & $130.32$ & $172.77$\\
    &TCPU & $11.38$ & $58.66$ & $452.22$ & $558.95$ & $1388.29$ & $1672.75$\\
    &IT   & $48$ & $47$ & $53$ & $58$ & $60$ & $64$ \\
    &$\frac{\|\mathbf{r}_k\|_2^2}{\|\mathbf{b}\|_2^2}$ & $1.37e-09$ & $2.51e-11$ & $2.01e-11$ & $2.81e-13$ & $1.68e-14$ & $5.03e-15$\\
    &$u$ error & $6.33e-04$ & $ 5.77e-05$ & $8.85e-05$ & $3.92e-06$ & $ 2.75e-06$ & $1.31e-07$\\
    &$v$ error & $2.69e-04$ & $4.92e-05$ & $7.49e-05$ & $2.61e-06$ & $1.39e-06$ & $ 1.11e-07$\\
    &$w$ error & $4.24e-04$ & $ 2.52e-05$ & $3.69e-05$ & $3.05e-06$ & $1.64e-06$ & $ 3.53e-08$\\
    \midrule 
    CSQR-PLSQR &PCPU & $9.42$ & $45.10$ & $ 209.38$ & $348.27$ & $773.25$ & $ 1039.61$ \\
    &CPU  & $0.99$ & $9.77$ & $ 67.20 $ & $85.50$ & $ 240.19$ & $265.43$ \\
    &TCPU & $10.41$ & $54.87$ & $276.58$ & $433.77$ & $1013.44$ & $1305.04$ \\
    &IT   & $89$ & $ 146$ & $83$ & $83$ & $106$ & $106$ \\
    &$\frac{\|\mathbf{r}_k\|_2^2}{\|\mathbf{b}\|_2^2}$ & $ 1.37e-09$ & $2.63e-11$ & $2.03e-11$ & $4.47e-13$ & $4.53e-13$ & $7.55e-13$ \\
    &$u$ error & $6.32e-04$ & $ 5.73e-05$ & $ 8.90e-05$ & $5.44e-06$ & $ 3.99e-06$ & $9.75e-06$ \\
    &$v$ error & $2.70e-04$ & $5.12e-05$ & $7.59e-05$ & $ 3.07e-06$ & $2.08e-06$ & $8.04e-06$ \\
    &$w$ error & $4.24e-04$ & $2.58e-05$ & $3.65e-05$ & $4.10e-06$ & $3.24e-06$ & $8.37e-06$ \\
    \midrule
    SPQR &TCPU& $8.50$ & $ 41.44$ & $458.61$ & $745.07$ & $2047.42$ & $2336.98$\\
    &$\frac{\|\mathbf{r}_k\|_2^2}{\|\mathbf{b}\|_2^2}$& $1.37e-09$ & $2.51e-11$ & $2.01e-11$ & $2.81e-13$& $1.83e-14$ & $ 1.83e-15$\\
    \midrule
    HHQR &TCPU& $3.95$ & $18.4743$ & $694.76$ & $ 1120.25$ & $ 3016.92$ & $3541.24$\\
    &$\frac{\|\mathbf{r}_k\|_2^2}{\|\mathbf{b}\|_2^2}$& $1.37e-09$ & $2.51e-11$ & $2.01e-11$ & $2.81e-13$ & $1.66e-14$ & $ 5.60e-16$\\
    \bottomrule
  \end{tabular}
  \end{center}
\end{table}
\begin{table}[!htbp]
  \begin{center}
    \caption{Numerical results of LSRN-LSQR, CSSVD-PLSQR, and CSSVDP-LSQR methods for the three-dimensional elasticity problem.}\label{tab:three-dimensional_elasticity}%
    \setlength{\tabcolsep}{4.50mm}  
    \begin{tabular}{lllllll}
    \toprule
    Method &  \multicolumn{2}{c}{LSRN-LSQR}  &  \multicolumn{2}{c}{CSSVD-PLSQR} &  \multicolumn{2}{c}{CSSVDP-LSQR} \\
    \cmidrule(lr){2-3} \cmidrule(lr){4-5} \cmidrule(lr){6-7}
    $m$ &$ 31200$ &  $97200$ & $31200$  & $97200$ & $ 31200 $ & $97200$  \\
    $n$ & $1200$ &  $1800$ & $1200$  & $1800$ & $ 1200 $ & $ 1800 $  \\
    PCPU &  $ 31.38$ &  $113.40$ & $24.21$  & $71.65$ & $ 26.04$ & $78.36 $  \\
    CPU &  $10.35$ &  $61.02$ & $0.94$  & $ 6.68$ & $0.46$ & $  2.33$  \\
    TCPU &  $41.73$ &  $174.42$ & $25.15$  & $78.33$ & $22.65$ & $80.69$  \\
    IT &  $ 1000$ &  $1000$ & $90$  & $145$ & $48$ & $ 51 $  \\
    $\frac{\|\mathbf{r}_k\|_2^2}{\|\mathbf{b}\|_2^2}$ &  $9.60e-03$ &  $1.22e-01$ & $1.37e-09$  & $2.63e-11$ & $1.37e-09$ & $ 2.51e-11$  \\
    $u$ error &  $6.10e-01$ &  $2.36e+00$ & $6.29e-04$  & $5.01e-05$ & $6.33e-04$ & $5.77e-05$  \\
    $v$ error &  $2.17e-01$ &  $8.16e-01$ & $2.67e-04$  & $4.01e-05$ & $2.69e-04$ & $4.92e-05$  \\
    $w$ error &  $4.18e-01$ &  $1.08e+00$ & $4.21e-04$  & $2.60e-05$ & $4.24e-04$ & $2.52e-05$  \\
    \bottomrule 
    \end{tabular}
  \end{center}
\end{table}
\begin{example}\label{3D_Stokes_flow}
  Consider the three-dimensional Stokes flow problem
  \begin{equation}
    \begin{cases}-\Delta \boldsymbol{u}(\mathbf{x})+\nabla p(\mathbf{x})=f(\mathbf{x}), & \mathbf{x} \in \Omega, \\ \nabla \cdot \boldsymbol{u}(\mathbf{x})=0, & \mathbf{x} \in \Omega, \\ \boldsymbol{u}(\mathbf{x})=\boldsymbol{U}(\mathbf{x}), & \mathbf{x} \in \partial \Omega,\end{cases}
    \end{equation}
    where $\Omega=[0,10] \times[0,10]\times[0,10]$, and velocity Dirichlet boundary conditions are applied to all boundaries, specifically imposed only on the displacement function. The exact solution is
\begin{equation}\label{three-dimensional_ Stokes_flow_true_solution}
  \begin{aligned}
    u(x, y, z)&=\sin (x+y+z)-\tanh (x+y+z)+x^2 y+x z^2,\\
    v(x, y, z)&=-\sin (x+y+z)+\cos (x+y+z)-x y^2+y^2 z,\\
    w(x, y, z)&=-\cos (x+y+z)+\tanh (x+y+z)-\frac{z^3}{3}-y z^2,\\
    p(x, y, z)&=\frac{P}{6 E I}\left(3 \mu x^2(L-x+y)+(4+5 \mu) \frac{D^2 x y}{4}+(3 L-z) z^2\right).
  \end{aligned}
\end{equation}
\end{example}
\begin{table}[!htbp]
  \begin{center}
    \caption{Comparison of CSQRP-LSQR, CSQR-PLSQR, SPQR, and HHQR methods for the three-dimensional Stokes flow problem.}\label{tab:three-dimensional Stokes flow}%
  \setlength{\tabcolsep}{3.00mm} 
  \begin{tabular}{llllllll}
  \toprule
  Method &$m$ & $39200$ &  $124200$ & $347200$  & $347200$ & $642500$ & $642500$ \\
  &$n$ & $ 3200 $ &  $ 3200$ & $19200$  & $ 25600$ & $32000$ & $ 38400$ \\
  \midrule 
  CSQRP-LSQR &PCPU & $22.91$ & $60.78$ & $436.61$ & $782.58$ & $1801.06$ & $2572.67$\\
  &CPU  & $2.06$ & $4.32$ & $ 59.52$ & $102.31$ & $204.40$ & $279.06$\\
  &TCPU & $24.97$ & $65.10$ & $496.13$ & $884.89$ & $2005.46$ & $2851.73$\\
  &$\frac{\|\mathbf{r}_k\|_2^2}{\|\mathbf{b}\|_2^2}$& $4.44e-12$ & $6.11e-12$ & $9.95e-12$ & $2.05e-13$ & $ 9.39e-15$ & $7.92e-16$\\
  &IT   & $57$ & $53$ & $54$ & $60$ & $58$ & $63$\\
  &$u$ error  & $1.10e-06$ & $1.36e-06$ & $6.05e-07$ & $9.76e-08$ & $6.91e-08$ & $5.55e-09$\\
  &$v$ error  & $5.14e-07$ & $6.30e-07$ & $2.96e-07$ & $4.89e-08$ & $ 2.64e-08$ & $2.61e-09$\\
  &$w$ error & $7.59e-07$ & $9.03e-07$ & $4.00e-07$ & $8.58e-08$ & $ 1.31e-08$ & $3.19e-09$\\
  \midrule 
  CSQR-PLSQR &TCPU & $18.88$ & $51.97$ & $296.98$ & $529.51$ & $1107.82$ & $1723.14$\\
  &CPU  & $38.89$ & $ 78.85$ & $938.39$ & $ 1280.56$ & $2725.97$ & $3281.50$\\
  &TCPU & $57.77$ & $130.82$ & $1235.37$ & $1810.07$ & $3833.79$ & $5004.64$\\
  &$\frac{\|\mathbf{r}_k\|_2^2}{\|\mathbf{b}\|_2^2}$& $9.79e-01$ & $9.93e-01$ & $9.92e-01$ & $9.92e-01$ & $9.95e-01$ & $9.95e-01$\\
  &IT   & $1000$ & $1000$ & $1000$ & $1000$ & $1000$ & $1000$\\
  &$u$ error  & $9.90e-01$ & $9.96e-01$ & $9.96e-01$ & $9.96e-01$ & $9.97e-01$ & $9.97e-01$\\
  &$v$ error  & $9.90e-01$ & $9.96e-01$ & $9.96e-01$ & $9.96e-01$ & $9.97e-01$ & $9.97e-01$\\
  &$w$ error & $9.90e-01$ & $9.97e-01$ & $9.96e-01$ & $9.96e-01$ & $9.97e-01$ & $9.97e-01$\\
  \midrule
  SPQR &TCPU& $18.95$ & $54.31$ & $636.23$ & $1197.22$ & $2755.07$ & $3846.36$\\
  &$\frac{\|\mathbf{r}_k\|_2^2}{\|\mathbf{b}\|_2^2}$& $4.44e-12$ & $6.14e-12$ & $ 9.96e-12$ & $2.12e-13$ & $9.62e-15$ & $7.92e-16$\\
  \midrule
  HHQR &TCPU& $12.34$ & $35.86$ & $ 1184.43$ & $1905.83$ & $5051.38$ & $7247.11$\\
  &$\frac{\|\mathbf{r}_k\|_2^2}{\|\mathbf{b}\|_2^2}$& $4.44e-12$ & $6.11e-12$ & $ 9.95e-12$ & $2.05e-13$ & $ 7.78e-15$ & $2.90e-16$\\
  \bottomrule
  \end{tabular}
  \end{center}
\end{table}
\begin{table}[!htbp]
  \begin{center}
    \caption{Comparison of LSRN-LSQR, CSSVD-PLSQR, and CSSVDP-LSQR methods for the three-dimensional Stokes flow problem.}\label{tab:3Dtokes}%
    \setlength{\tabcolsep}{4.50mm}  
    \begin{tabular}{lllllll}
    \toprule
    Method &  \multicolumn{2}{c}{LSRN-LSQR}  &  \multicolumn{2}{c}{CSSVD-PLSQR} &  \multicolumn{2}{c}{CSSVDP-LSQR} \\
    \cmidrule(lr){2-3} \cmidrule(lr){4-5} \cmidrule(lr){6-7}
    $m$ &$39200$ &  $124200$ & $39200$  & $124200$ & $39200$ & $124200$  \\
    $n$ & $3200$ &  $3200$ & $3200$  & $3200$ & $3200$ & $3200$  \\
    PCPU &  $165.27$ &  $293.49$ & $125.46$  & $158.27$ & $130.79$ & $ 165.34$  \\
    CPU &  $52.61$ &  $81.10$ & $49.86$  & $85.98$ & $ 1.90$ & $4.33$  \\
    TCPU &  $217.88$ &  $374.59$ & $175.32$  & $244.25$ & $132.69$ & $169.67$  \\
    IT & $1000$ &  $1000$ & $1000$  & $1000$ & $55$ & $53$  \\
    $\frac{\|\mathbf{r}_k\|_2^2}{\|\mathbf{b}\|_2^2}$ &  $1.00e+00$ &  $1.00e+00$ & $ 9.80e-01$  & $ 9.92e-01$ & $4.44e-12$ & $6.11e-12$  \\
    $u$ error &  $1.00e+00$ &  $1.00e+00$ & $ 9.90e-01$  & $9.96e-01$ & $1.10e-06$ & $1.36e-06$  \\
    $v$ error &  $1.00e+00$ &  $1.00e+00$ & $9.90e-01$  & $9.96e-01$ & $5.14e-07$ & $6.30e-07$  \\
    $w$ error &  $1.00e+00$ &  $1.00e+00$ & $9.89e-01$  & $9.96e-01$ & $7.59e-07$ & $9.03e-07$  \\
    \bottomrule 
    \end{tabular}
  \end{center}
\end{table}

Next, we further compare the accuracy, computational efficiency, and other numerical performance aspects of the LSRN-LSQR, CSQRP-LSQR, CSSVDP-LSQR, CSQR-PLSQR, CSSVD-PLSQR, SPQR, and HHQR methods for the three-dimensional elasticity equation (Example \ref{3D_Elasticity}) and the three-dimensional Stokes flow equation (Example \ref{3D_Stokes_flow}). Numerical results are recorded in Tables \ref{tab:three-dimensional_elasticity_problem} and \ref{tab:three-dimensional_elasticity} for the elasticity problem, and in Tables \ref{tab:three-dimensional Stokes flow} and \ref{tab:3Dtokes} for the Stokes flow problem. The error precision of the CSQRP-LSQR and CSSVDP-LSQR methods on ill-conditioned least squares problems is comparable to those of HHQR and SPQR, as shown in Tables \ref{tab:three-dimensional_elasticity_problem}-\ref{tab:3Dtokes}. As the problem size increases, the CSQRP-LSQR method consumes less time compared to the HHQR and SPQR methods; see Tables \ref{tab:three-dimensional_elasticity_problem} and \ref{tab:three-dimensional Stokes flow}. The CSQR-PLSQR method demonstrates good convergence for elasticity problems, though its precision on larger scales is slightly lower than that of CSQRP-LSQR. Its advantage lies in reduced preconditioning time due to not forming an explicit preconditioned matrix, as shown in Table \ref{tab:three-dimensional_elasticity_problem}. However, for the Stokes flow problems, the errors and iteration count in Table \ref{tab:three-dimensional Stokes flow} indicate that the CSQR-PLSQR method exhibits very poor numerical convergence. In Table \ref{tab:three-dimensional_elasticity_problem}, the relative $L^2$ errors for $u, v$, and $w$ in the elasticity problem range from approximately $10^{-7} \sim 10^{-4}$ for CSQRP-LSQR and $10^{-6} \sim 10^{-4}$ for CSQR-PLSQR, while in Table \ref{tab:three-dimensional Stokes flow}, for the Stokes flow problem, the relative $L^2$ errors of the CSQRP-LSQR method range from approximately $10^{-9} \sim 10^{-6}$. Table \ref{tab:3Dtokes} indicates  that the numerical results of the LSRN-LSQR, CSSVD-PLSQR, and CSSVDP-LSQR methods exhibit the same experimental phenomena as their performance in Table \ref{tab:three-dimensional_Helmholtze_low_quation} for the Helmholtz equation, where only the CSSVDP-LSQR method performs well for the Stokes flow problem, while the LSRN-LSQR and CSSVD-PLSQR methods perform poorly.

In summary, the numerical results in this subsection demonstrate that the randomized methods CSQRP-LSQR and CSSVDP-LSQR, which explicitly compute the preconditioned matrix, consistently achieve high accuracy comparable to the direct solvers SPQR and HHQR. As the problem size increases, CSQRP-LSQR outperforms SPQR and HHQR in terms of computational time. In contrast, methods that do not explicitly form the preconditioned matrix, such as LSRN-LSQR, CSQR-PLSQR, and CSSVD-PLSQR, generally exhibit significantly poorer numerical performance.

\subsection{Least squares problems with rank deficiency or infinite condition number}\label{Subsection:4.4}

In this subsection, we check the performance of the CSSVDP-LSQR method for least squares problems with infinite condition numbers or rank deficiency, sourced from the Florida Sparse Matrix Collection \cite{DavisHu}. Table \ref{florida_sparse_matrix_information} lists their size, rank, density, and condition number. In these problems, $\mathbf{b}=\mathbf{A} \mathbf{x}_{\star}$, where $\mathbf{x}_{\star}$ is the exact solution generated by the ``numpy''  library using the ``randn'' function in Python. Tables \ref{tab:Numerical_florida_sparse_matrices_table_1}-\ref{tab:Numerical_florida_sparse_matrices_table_3} present the numerical results for CSQRP-LSQR and CSSVDP-LSQR methods. In these tables, ``$--$'' indicates the failure of the CSQRP-LSQR method in the preconditioning phase.
\begin{table}[!htbp]
  \begin{center}
  \caption{Comparison of CSQRP-LSQR and CSSVDP-LSQR methods for matrices from the Florida Sparse Matrix Collection.}\label{tab:Numerical_florida_sparse_matrices_table_1}%
  \setlength{\tabcolsep}{2.00mm} 
  \begin{tabular}{llllllll}
  \toprule
  Method &  Name & GL7d11 & rel6 & aa4${^T}$ & relat6  & air05${^T}$ & ch8-8-b1 \\
  \midrule 
  CSQRP-LSQR  &$\kappa(\mathbf{A}\mathbf{R}^{-1})$ & $ 60.6420$ & $--$ & $--$ & $--$ & $--$ & $1.83e+15$\\
              &PCPU & $0.125$ & $--$ & $--$ & $--$ & $--$ & $0.1189$\\
              &CPU & $0.0025$ & $--$ & $--$ & $--$ & $--$ & $0.0026$\\
              &TCPU & $0.1276$ & $--$ & $--$ & $--$ & $--$ & $0.7204$\\
              &IT& $33$ & $--$ & $--$ & $--$ & $--$ & $24$\\
              &$\frac{\|\mathbf{r}_k\|^2_2}{\|\mathbf{b}\|_2^2}$ & $6.44e-15$ & $--$ & $--$ & $--$ & $--$ & $2.20e-15$\\
  \midrule
  CSSVDP-LSQR &$\kappa(\mathbf{A}\mathbf{P})$ & $3.8815$ & $3.1253$ & $ 3.4820$ & $3.2876$ & $ 3.4820$ & $3.5831$\\
              &PCPU & $0.1536$ & $0.1725$ & $0.4726$ & $0.1857$ & $0.5233$ & $0.1467$\\
              &CPU & $  0.0019$ & $0.0058$ & $0.0386$ & $0.0060$ & $0.0572$ & $ 0.0026$\\
              &TCPU & $0.1555$ & $0.1783$ & $0.5112$ & $0.1917$ & $0.5805$ & $0.1493$\\
              &IT& $24$ & $24$ & $27$ & $25$ & $27$ & $24$\\
              &$\frac{\|\mathbf{r}_k\|_2^2}{\|\mathbf{b}\|_2^2}$ & $3.27e-15$ & $4.60e-15$ & $3.10e-15$ & $ 4.90e-15$ & $3.10e-15$ & $ 1.53e-15$\\
  \bottomrule
  \end{tabular}
  \end{center}
\end{table}
\begin{table}[!htbp]
  \begin{center}
  \caption{Comparison of CSQRP-LSQR and CSSVDP-LSQR methods for matrices from the Florida Sparse Matrix Collection.}\label{tab:Numerical_florida_sparse_matrices_table_2}%
  \setlength{\tabcolsep}{1.50mm}  
  \begin{tabular}{llllllll}
  \toprule
  Method &  Name & shar\_te2-b1 & us04${^T}$ & ch8-8-b2 & rel7  & kl02${^T}$ & relat7b \\
  \midrule 
  CSQRP-LSQR  &$\kappa(\mathbf{A}\mathbf{R}^{-1})$ & $2.18e+14$ & $ 4.52e+43$ & $1.39e+04$ & $--$ & $46.1324$ & $--$\\
              &PCPU & $0.4507$ & $ 0.5356$ & $ 1.4143$ & $--$ & $0.6256$ & $--$\\
              &CPU & $0.2823$ & $0.0450$ & $ 13.9163$ & $--$ & $0.1398$ & $--$\\
              &TCPU & $0.7330$ & $0.9856$ & $15.3306$ & $--$ & $2.0615$ & $--$\\
              &IT& $28$ & $6$ & $884$ & $--$ & $71$ & $--$\\
              &$\frac{\|\mathbf{r}_k\|_2^2}{\|\mathbf{b}\|_2^2}$ & $ 5.98e-15$ & $ 0.99$ & $2.01e-08$ & $--$ & $1.53e-04$ & $--$\\
  \midrule
  CSSVDP-LSQR &$\kappa(\mathbf{A}\mathbf{P})$ & $3.7358$ & $ 2.7751$ & $3.6138$ & $3.6682$ & $ 3.1171$ & $3.6051$\\
              &PCPU & $ 0.4749$ & $ 0.5202$ & $2.0042$ & $1.3262$ & $ 0.6224$ & $1.4965$\\
              &CPU & $0.2755$ & $0.0418$ & $  0.6461$ & $0.6281$ & $ 0.0355$ & $0.5367$\\
              &TCPU & $2.1006$ & $0.5620$ & $2.6503$ & $1.9543$ & $1.5740$ & $4.5166$\\
              &IT& $28$ & $21$ & $29$ & $29$ & $22$ & $29$\\
              &$\frac{\|\mathbf{r}_k\|_2^2}{\|\mathbf{b}\|_2^2}$ & $4.83e-15$ & $ 2.10e-15$ & $3.65e-15$ & $5.04e-15$ & $4.06e-15$ & $ 4.09e-15$\\
  \bottomrule
  \end{tabular}
  \end{center}
\end{table}
\begin{table}[!htbp]
  \begin{center}
  \caption{Comparison of CSQRP-LSQR, CSSVDP-LSQR and QR methods for matrices from the Florida Sparse Matrix Collection.}\label{tab:Numerical_florida_sparse_matrices_table_3}%
  \setlength{\tabcolsep}{1.50mm}  
  \begin{tabular}{llllllll}
  \toprule
  Method &  Name & stat96v5$^T$ & shar\_te2-b2 & connectus${^T}$ & rel8  & relat8 & 12month1$^T$ \\
  \midrule 
  CSQRP-LSQR  &$\kappa(\mathbf{A}\mathbf{R}^{-1})$ & $--$ & $7.05e+14$ & $--$ & $--$ & $--$ & $\dagger$\\
              &PCPU & $--$ & $166.0192$ & $--$ & $--$ & $--$ & $279.8977$\\
              &CPU & $--$ & $630.5781$ & $--$ & $--$ & $--$ & $>30000$\\
              &TCPU & $--$ & $796.5973$ & $--$ & $--$ & $--$ & $--$\\
              &IT & $--$ & $402$ & $--$ & $--$ & $--$ & $--$\\
              &$\frac{\|\mathbf{r}_k\|_2^2}{\|\mathbf{b}\|_2^2}$ &$--$ & $3.16e-04$ & $--$ & $--$ & $--$ & $--$\\
  \midrule
  CSSVDP-LSQR &$\kappa(\mathbf{A}\mathbf{P})$ & $3.7392$ & $ 3.7476$ & $3.3782$ & $3.7503$ & $ 3.7195$ & $ 5.7268$\\
              &PCPU & $ 5.7638$ & $512.0306$ & $9.1436$ & $296.6587$ & $303.0039$ & $415.9679$\\
              &CPU & $2.8392$ & $51.9230$ & $ 4.4282$ & $77.0743$ & $78.1589$ & $280.6610$\\
              &TCPU & $8.6030$ & $563.9536$ & $13.5718$ & $373.7330$ & $381.1628$ & $696.6289$\\
              &IT & $30$ & $30$ & $26$ & $30$ & $30$ & $39$\\
              &$\frac{\|\mathbf{r}_k\|_2^2}{\|\mathbf{b}\|_2^2}$ & $4.86e-15$ & $5.16e-15$ & $4.45e-15$ & $5.08e-15$ & $ 4.49e-15$ & $ 1.05e-14$\\
  \midrule
  QR        &TCPU & $--$ & $962.9913$ & $--$ & $--$ & $--$ & $2311.7958$\\
            &$\frac{\|\mathbf{r}_k\|_2^2}{\|\mathbf{b}\|_2^2}$ & $--$ & $1.05e-28$ & $--$ & $--$ & $--$ & $\text{nan}$\\
  \bottomrule
  \end{tabular}
  \end{center}
\end{table}

In Tables \ref{tab:Numerical_florida_sparse_matrices_table_1}-\ref{tab:Numerical_florida_sparse_matrices_table_3}, the condition number of $\mathbf{AP}$ ranges from 3 to 6. Upon satisfying the termination criteria, the CSSVDP-LSQR method consistently achieves a relative around $10^{-15}$ with 24 to 40 iterations. In contrast, the CSQRP-LSQR method is less effective for rank-deficient and infinite condition number problems. In Table \ref{tab:Numerical_florida_sparse_matrices_table_3}, ``$\dagger$'' indicates that calculating the condition number for $\mathbf{A} \mathbf{R}^{-1}$ causes a non-convergence exception in SVD. Additionally, results in Table \ref{tab:Numerical_florida_sparse_matrices_table_3} suggest that QR decomposition is less effective for these problems.

It is known that the convergence of most iterative methods depends on the distribution of eigenvalues or singular values of the preconditioned matrix. Concentration around $1$ with few outliers leads to fast convergence \cite{GroHuc}. Figure \ref{fig:shar_te2-b1_in_A_and_B_singular_values} shows that for shar\_te2-b1, singular values of $\mathbf{A}\mathbf{R}^{-1}$ cluster near 1 with one outlier, indicating the improved distribution due to CSQRP-LSQR preconditioning. Despite a condition number of $2.18 \times$ $10^{14}$, LSQR converges.  
Similarly, Table \ref{tab:Numerical_florida_sparse_matrices_table_1} shows that for ch8-8-b1, $\mathbf{A}\mathbf{R}^{-1}$ has a condition number of $1.83 \times 10^{15}$, but CSQRP-LSQR still converges. 
Figure \ref{fig:connectus_in_A_and_B_singular_values} also shows that $\mathbf{AP}$ has no outlier singular values due to a threshold (RCOND) in the preprocessing stage of the CSSVDP-LSQR method, leading to a more concentrated singular value distribution and enhancing both the stability and convergence of LSQR.
\begin{figure}[!htbp]
  \renewcommand\figurename{Figure}
    \centering
    \includegraphics[scale=0.30]{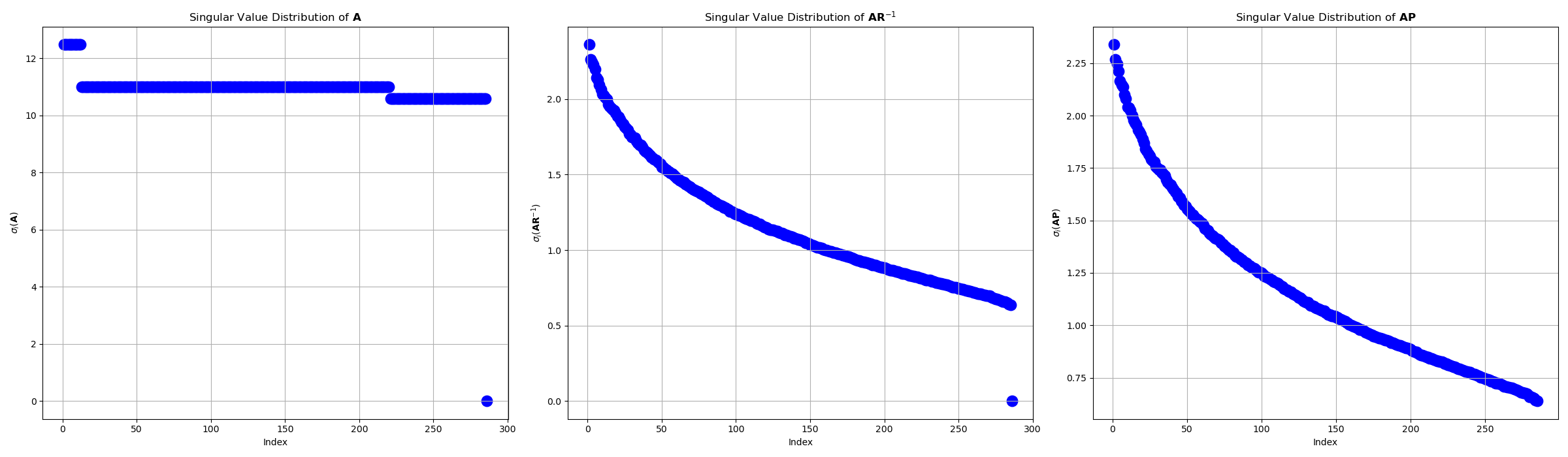}
  \caption{shar\_te2-b1: Singular value distribution of $\mathbf{A}$, $\mathbf{A}\mathbf{R}^{-1}$, and $\mathbf{A}\mathbf{P}$.}
  \label{fig:shar_te2-b1_in_A_and_B_singular_values}
\end{figure}
\begin{figure}[!htbp]
  \renewcommand\figurename{Figure}
    \centering
    \includegraphics[scale=0.45]{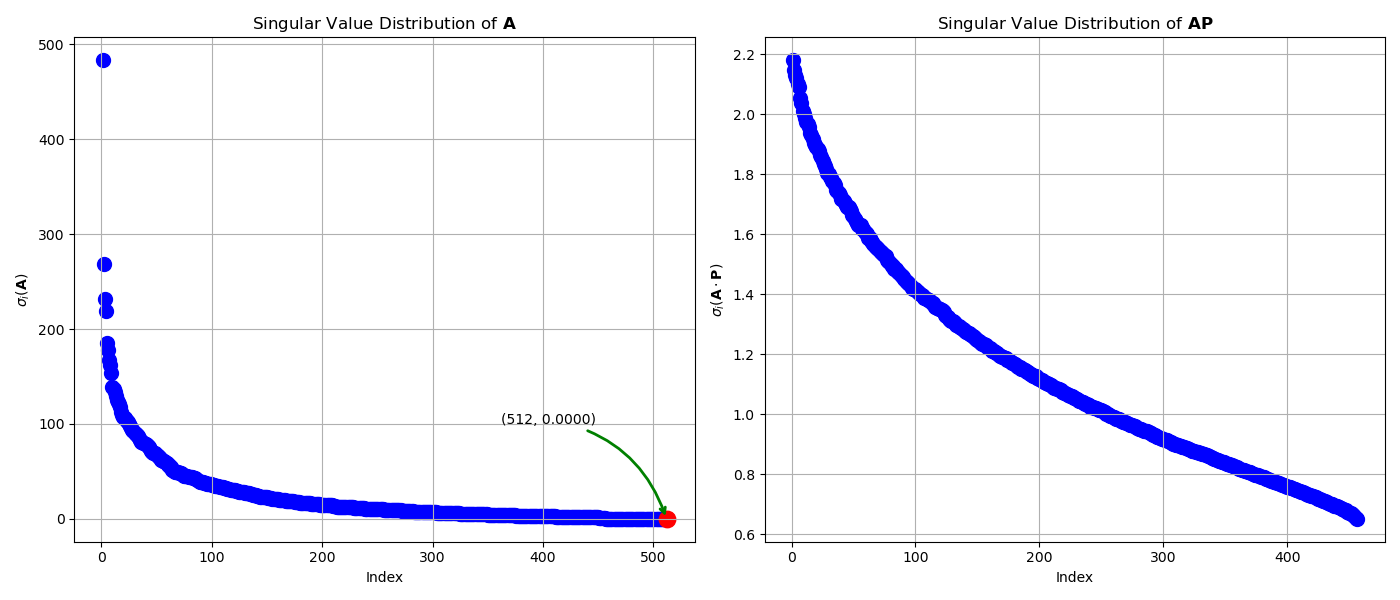}
  \caption{connectus: Singular value distribution of $\mathbf{A}$, $\mathbf{A}\mathbf{R}^{-1}$, and $\mathbf{A}\mathbf{P}$.}
  \label{fig:connectus_in_A_and_B_singular_values}
\end{figure}

\section{Concluding remarks} \label{section:5}
In this work, we have combined the randomization and preconditioning techniques to design randomized iterative methods for linear least squares problems.
Our methods consist of three main components: (1) a count sketch technique is used to sketch the original matrix to a smaller matrix; (2) a QR factorization or a singular value decomposition is employed for the smaller matrix to obtain the preconditioner, which is multiplied to the original matrix from the right-hand side; (3) least squares iterative solvers are employed to solve the preconditioned least squares system. Therefore, the methods we have developed are termed CSQRP-LSQR and CSSVDP-LSQR. In the framework of $\ell_2$ subspace embedding, we have proved that the preconditioned problem holds a condition number whose upper bound is independent of the condition number of the original matrix, and have provided error estimates for both methods. The CSQRP-LSQR method is computationally faster on large-scale problems compared to the state-of-the-art sparse direct solver SPQR \cite{Davis_SPQR}. Meanwhile, the relative least squares error precision of CSQRP-LSQR and CSSVDP-LSQR methods is on par with SPQR and Householder QR. For some problems, we have observed that the explicit construction of the preconditioned matrix enables high precision in CSQRP-LSQR and CSSVDP-LSQR methods, a significant advantage over LSRN-LSQR, CSQR-PLSQR, and CSSVD-PLSQR methods. At last, the CSSVDP-LSQR method can also solve large-scale, sparse, severely overdetermined least squares problems with rank deficiency and infinite condition numbers with high accuracy.

Theorem \ref{theorem:3.2} suggests a sample size of $s=\mathcal{O}(n^2)$. However, extensive numerical experiments have demonstrated that linear sample size has already worked well in practice. It remains unclear how to close this gap. Nevertheless, our discovery greatly enhances the practical utility of the count sketch transformation. Experiments also show that such a technique to compute the sketched matrix significantly reduces fill-in, and requires high memory demands, which shall be explored in the future.  

 \section*{Acknowledgements}
 This work was partially supported by the NSFC Major Research Plan - Interpretable and General Purpose Next-generation Artificial Intelligence (Nos. 92270001 and 92370205).
 
 \bibliographystyle{siam}
 \bibliography{paper_references}

\newpage  
\begin{appendices}
\section{Results for PDEs without explicit  solutions}\label{my_experiments_without_explicit_solutions}

In \ref{Subsection:4.3}, randomized iterative methods, such as LSRN-LSQR, CSQR-PLSQR, and CSSVD-PLSQR, perform poorly on highly ill-conditioned least squares problems with exact solutions. In this section, we check their performance for PDE problems without explicit solutions and demonstrate the effectiveness of CSQRP-LSQR and CSSVDP-LSQR methods. Hyper-parameters in the RFM  are recorded in Appendix \ref{my_hyperparameters_RFM}, Tables \ref{tab:hyper-parameters_without_analytical_solutions_2DPDEs} and \ref{tab:hyper-parameters_without_analytical_solutions_3DPDEs}, respectively.

Consider the two-dimensional elasticity problem 
\begin{equation}
  \begin{cases}-\operatorname{div}(\mathbf{u}(\mathbf{x}))=\mathbf{B}(\mathbf{x}), & \mathbf{x} \in \Omega \\ \sigma(\mathbf{u}(\mathbf{x})) \cdot \mathbf{n}=N(\mathbf{x}), & \mathbf{x} \in \Gamma_N \\ \mathbf{u}(\mathbf{x}) =\mathbf{U}(\mathbf{x}), & \mathbf{x} \in \Gamma_D\end{cases}
  \end{equation}

\begin{example}\label{example:4.8}
First, the domain is given by a square $(-1,1.5) \times(-0.5,0.5)$ jointed by a semi-disk centered at $(1.0,0.0)$ with radius 0.5, with two disks centered at $(1.2,0.0),(-0.5,0.0)$ with radius 0.2 removed (see Fig.\ref{fig:double_holes_elasticity}). The left boundary $x=-0.5$ is fixed and a load $P=10^7$ $\mathrm{Pa}$ is applied on the upper half of the semicircle. Dirichlet boundary condition is applied on the left boundary $x=-0.5$ and Neumann boundary condition is applied on the other boundaries.
\end{example}

\begin{example}\label{example:4.9}
Second, we consider a complex geometry; see Figure \ref{fig:complex_geometry_domain_elasticity}. Here $\Omega$ is defined as a square $(0,8) \times(0,8)$ with 40 holes of radius between 0.3 and 0.6 inside. Note that there is a cluster of nearly touching holes, as shown in the inset. 
Dirichlet boundary condition is applied on the lower boundary $y=0$ and Neumann boundary condition on the other boundaries and the holes inside. The material constants are: Young's modulus $E=3 \times 10^7 \mathrm{~Pa}$ and Poisson ratio $v=0.3$.
\end{example}
\begin{example}\label{example:4.10}
Consider the homogenization equation over the unit disk
  \begin{equation}\label{2D_multi-scale_solutions}
    \begin{cases}-\operatorname{div}(a(\mathbf{x}) \nabla u(\mathbf{x}))=f(\mathbf{x}), & \mathbf{x} \in \Omega, \\ u(\mathbf{x})=0, & \mathbf{x} \in \partial \Omega,\end{cases}
    \end{equation}
  where $a(\mathbf{x})=e^{h(\mathbf{x})}, h(\mathbf{x})=\sum_{|\mathbf{k}| \leq R}\left(a_{\mathbf{k}} \sin (2 \pi \mathbf{k} \cdot \mathbf{x})+b_{\mathbf{k}} \cos (2 \pi \mathbf{k} \cdot \mathbf{x})\right), R=6$, and $\left\{a_{k}\right\}$ and $\left\{b_k\right\}$ are independent, identically distributed random variables with the distribution $\mathbb{U}[-0.3,0.3]$. This is chosen so there is no clear scale separation in the coefficient \cite{OwhadiZhang}.
\end{example}
\begin{example}
  Consider the Stokes flow defined by
\begin{equation}
  \begin{cases}-\Delta \boldsymbol{u}(\mathbf{x})+\nabla p(\mathbf{x})=f(\mathbf{x}), & \mathbf{x}\in \Omega \\ \nabla \cdot \boldsymbol{u}(\mathbf{x})=0, & \mathbf{x} \in \Omega \\ \boldsymbol{u}(\mathbf{x})=\boldsymbol{U}(\mathbf{x}), & \mathbf{x} \in \partial \Omega \end{cases}
\end{equation}
with four sets of complex obstacles under inhomogeneous boundary conditions
\begin{equation}
  \left.(u, v)\right|_{\partial \Omega}= \begin{cases}(y(1-y), 0), & \text { if } x=0,\\ (y(1-y), 0), & \text { if } x=1, \\ (0,0), & \text { otherwise.}\end{cases}
  \end{equation}
\end{example}
\begin{example}
Consider the three-dimensional homogenization equation with Dirichlet boundary condition over $\Omega=[0,1] \times[0,1]\times[0,1]$ 
\begin{equation}
  \left\{\begin{array}{c}
  -\operatorname{div}(a(\mathbf{x}) u(\mathbf{x}))=f(\mathbf{x}), \quad \mathbf{x} \in \Omega,  \\
  u(\mathbf{x})=0 ,\quad \mathbf{x} \in \partial \Omega,
  \end{array}\right.
\end{equation}
where $a(\mathbf{x})=e^{h(\mathbf{x})}$, 
$$
\begin{aligned}
& h(\mathbf{x})=\sum_{|k| \leq R}\left(a_k \sin (2 \pi \mathbf{k} \cdot \mathbf{x})+b_k \cos (2 \pi \mathbf{k} \cdot \mathbf{x})\right) \\
& =\sum_{|k| \leq R}\left(a_k \sin \left(2 \pi\left(k_x x+k_y y+k_z z\right)\right)+b_k \cos \left(2 \pi k_x x+k_y y+k_z z\right)\right),
\end{aligned}
$$
and $\left\{a_k\right\}$ and $\left\{b_k\right\}$ are independent, identically distributed random variables with the distribution $\mathbb{U}[-0.3, 0.3]$. When $R$ takes the values $2$, $2.5$, $3$, and $4$, $h(\mathbf{x})$ is respectively the sum of $33,81,123$, and 257 distinct functions.
\end{example}
\begin{example}
Consider the three-dimensional Poisson equation with complex geometry
\begin{equation}\label{eq:complex_Poisson}
  \begin{cases}\Delta u(x,y,z)= 1 & (x, y, z) \in \Omega, 
    \\ u(x,y,z)=0 & (x, y, z) \in \partial \Omega,\end{cases}
\end{equation}
where $\Omega$ is defined by rotating a square region $[0,1] \times[1,2]$, from which four circular holes have been removed and five circular holes have been filled, around the z-axis to form a complete revolution. Figure \ref{fig:three-dimensional_Poisson_equation_solution} illustrates the specific structure of this square region, and the center coordinates and radii of these circular holes are provided in Table \ref{tab:Filled_Removed_holes}.
\end{example}

Tables \ref{tab:my_double_holes} and \ref{tab:complex_domain_elasticity} compare the numerical performance of the CSQRP-LSQR method with different initial values for Examples  \ref{example:4.8} and \ref{example:4.9}. From these two tables, it is evident that using $\mathbf{x}^{(0)}=\mathbf{Q}^T(\mathbf{S b})$ as the initial guess in LSQR extends preprocessing time compared to $\mathbf{x}^{(0)}=\mathbf{0}$, due to the additional computational effort required to compute $\mathbf{Q}^T(\mathbf{S b})$. This approach reduces the number of iterations for some problems, though it does not significantly affect the iteration count for most. As the size of the least squares problem increases, the relative $L^2$ errors of the displacement field $u, v$, and the stress field $\sigma_x, \tau_{x y}, \sigma_y$ decrease. For Example \ref{example:4.8}, using the numerical solution from the largest parameter settings as the reference, the relative $L^2$ errors for $u, v, \sigma_x, \sigma_y$, and $\tau_{x y}$ are respectively $4.17 \times 10^{-3}, 2.83 \times 10^{-3}, 4.34 \times 10^{-3}, 8.88 \times 10^{-3}$, and $3.97 \times 10^{-3}$. For Example \ref{example:4.9}, the errors of $u, v, \sigma_x, \sigma_y$, and $\tau_{x y}$ are respectively $4.17 \times 10^{-3}, 2.83 \times 10^{-3}, 4.34 \times 10^{-3}, 8.88 \times 10^{-3}$, and $3.97 \times 10^{-3}$. Figures \ref{fig:double_holes_elasticity} and \ref{fig:complex_geometry_domain} depict the displacement field $u, v$ and the stress field $\sigma_x, \tau_{x y}, \sigma_y$ for Examples \ref{example:4.8} and \ref{example:4.9}, respectively.
\begin{table}[!htbp]
  \begin{center}
  \caption{Results of the elasticity problem over a complex geometry in Figure \ref{fig:double_holes_elasticity} using the CSQRP-LSQR method. The number of columns $n$ of $\mathbf{A}$ is 12800.}\label{tab:my_double_holes}%
  \setlength{\tabcolsep}{3.50mm} 
  \begin{tabular}{llllllll}
  \toprule
  Initial&  $m$ & $164806$ & $267658$ & $349992$  &$493988$  & $724344$ & $856012$\\
  \midrule
  $\mathbf{0}$ &PCPU & $119.73$ & $153.67$ &$213.26$  & $288.52$ & $379.02$ & $420.17$\\
  &CPU & $118.16$ & $179.39$ &$138.02$  & $204.90$ & $272.47$ & $405.88$\\
  &TCPU & $237.89$ & $333.06$ &$351.28$  & $493.42$ & $651.49$ & $826.05$\\
  &IT & $78$ & $65$ &$52$  & $46$ & $43$ & $59$\\
  &$u$ error  & $1.00e-01$ & $5.98e-02$ &$2.17e-02$  & $7.04e-03$ & $4.17e-03$ & Reference\\
  &$v$ error  & $1.07e-01$ & $6.60e-02$ &$2.31e-02$  & $9.55e-03$ & $2.83e-03$ & Reference\\
  &$\sigma_x$ error  & $1.01e-01$ & $6.10e-02$ &$2.21e-02$  & $8.72e-03$ & $4.34e-03$ & Reference\\
  &$\sigma_y$ error  & $1.01e-01$ & $6.90e-02$ &$3.22e-02$  & $1.51e-02$ & $8.88e-03$ & Reference\\
  &$\tau_{xy}$ error & $1.12e-01$ & $7.44e-02$ &$2.82e-02$  & $1.62e-02$ & $3.97e-03$ & Reference\\
  \midrule 
  $\mathbf{Q}^T(\mathbf{Sb})$ &PCPU & $130.23$ & $174.91$ &$218.26$  & $302.12$ & $395.13$ & $446.18$\\
  &CPU & $46.05$ & $80.80$ &$88.93$  & $138.97$ & $304.29$ & $184.46$\\
  &TCPU & $176.28$ & $255.71$ &$307.19$  & $441.09$ & $699.42$ & $630.64$\\
  &IT & $41$ & $33$ &$33$  & $42$ & $45$ & $29$\\
  \bottomrule
  \end{tabular}
  \end{center}
\end{table}
\begin{table}[!htbp]
  \begin{center}
  \caption{Results of the elasticity problem over a complex geometry in Figure \ref{fig:complex_geometry_domain_elasticity} using the CSQRP-LSQR method. The number of columns $n$ of $\mathbf{A}$ is 25600.}\label{tab:complex_domain_elasticity}%
  \setlength{\tabcolsep}{3.50mm}  
  \begin{tabular}{llllllll}
  \toprule
  Initial &  $m$ & $344126$ & $520582$ & $732988$  &$1072526$  & $1167574$ & $1266678$\\
  \midrule
  $\mathbf{0}$ &PCPU & $681.43$ & $796.20$ &$1233.40$  & $1422.82$ & $1633.54$ & $1726.92$\\
  &CPU & $186.02$ & $275.43$ &$329.71$  & $487.12$ & $530.29$ & $722.10$\\
  &TCPU & $867.45$ & $1071.63$ &$1563.11$  & $1909.94$ & $2163.83$ & $2449.10$\\
  &IT & $41$ & $40$ &$34$  & $32$ & $32$ & $32$\\
  &$u$ error & $1.28e-01$ & $8.69e-02$ &$5.59e-02$  & $3.70e-02$ & $9.26e-03$ & Reference\\
  &$v$ error & $8.79e-02$ & $1.09e-01$ &$4.54e-02$  & $7.15e-02$ & $1.34e-02$ & Reference\\
  &$\sigma_x$ error & $5.82e-02$ & $4.49e-02$ &$2.90e-02$  & $2.80e-02$ & $4.50e-03$ & Reference\\
  &$\sigma_y$ error & $1.14e-01$ & $9.97e-02$ &$7.05e-02$  & $2.80e-02$ & $1.13e-02$ & Reference\\
  &$\tau_{xy}$ error & $1.20e-01$ & $9.03e-02$ &$5.87e-02$  & $5.39e-02$ & $9.50e-03$ & Reference\\
  \midrule
  $\mathbf{Q}^T(\mathbf{Sb})$ &PCPU & $704.40$ & $806.10$ &$1286.44$  & $1482.52$ & $1656.59$ & $1738.34$\\
  &CPU & $173.82$ & $252.80$ &$339.99$  & $466.93$ & $528.95$ & $728.64$\\
  &TCPU & $878.22$ & $1058.90$ &$1626.43$  & $1949.45$ & $2185.54$ & $2466.98$\\
  &IT & $38$ & $35$ &$35$  & $30$ & $31$ & $32$\\
  \bottomrule
  \end{tabular}
  \end{center}
\end{table}
\begin{figure}[!htbp]
  \centering
  \begin{subfigure}[b]{0.8\textwidth}
    \includegraphics[width=\textwidth]{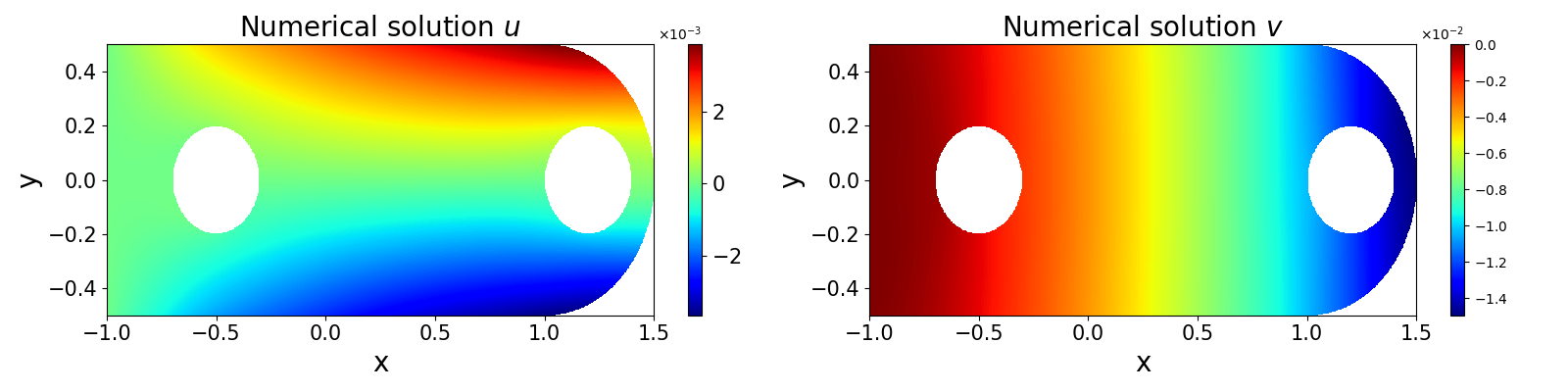}
  \end{subfigure}
  \medskip 
  \begin{subfigure}[b]{1.0\textwidth}
    \includegraphics[width=\textwidth]{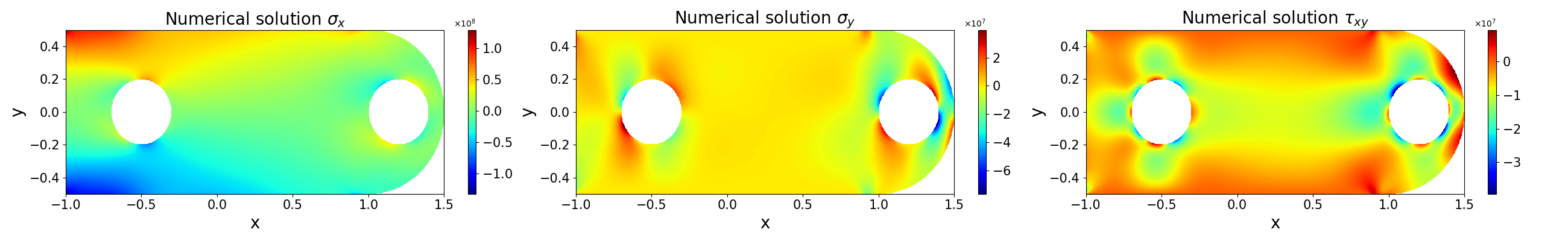}
  \end{subfigure}
  \caption{Reference solution obtained by the the CSQRP-LSQR method with $\mathbf{x}^{(0)}=\mathbf{0}$ for the two-dimensional elasticity problem.}
  \label{fig:double_holes_elasticity}
\end{figure}
\begin{figure}[!htbp]
    \centering
    \begin{subfigure}[b]{0.65\textwidth}
      \includegraphics[width=\textwidth]{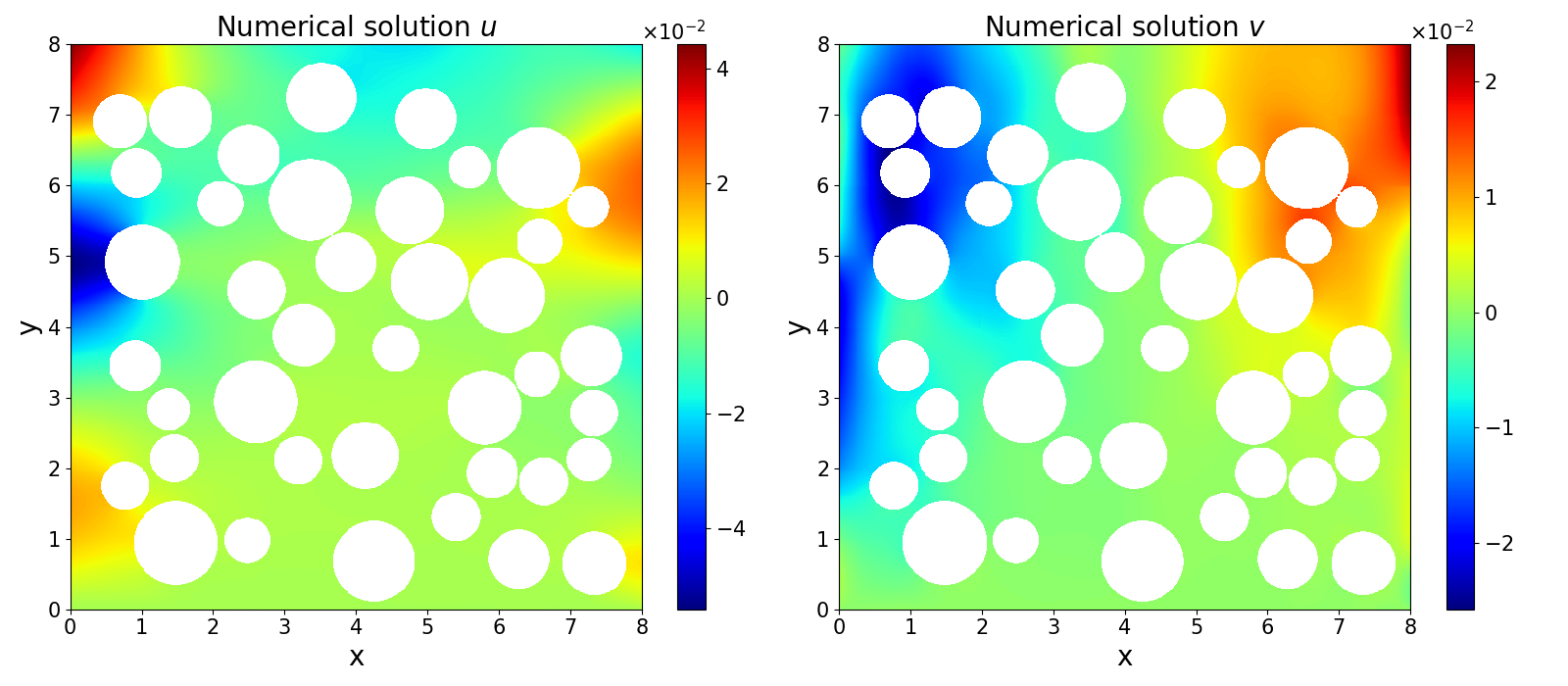}
    \end{subfigure}
    \medskip 
    \begin{subfigure}[b]{1.00\textwidth}
      \includegraphics[width=\textwidth]{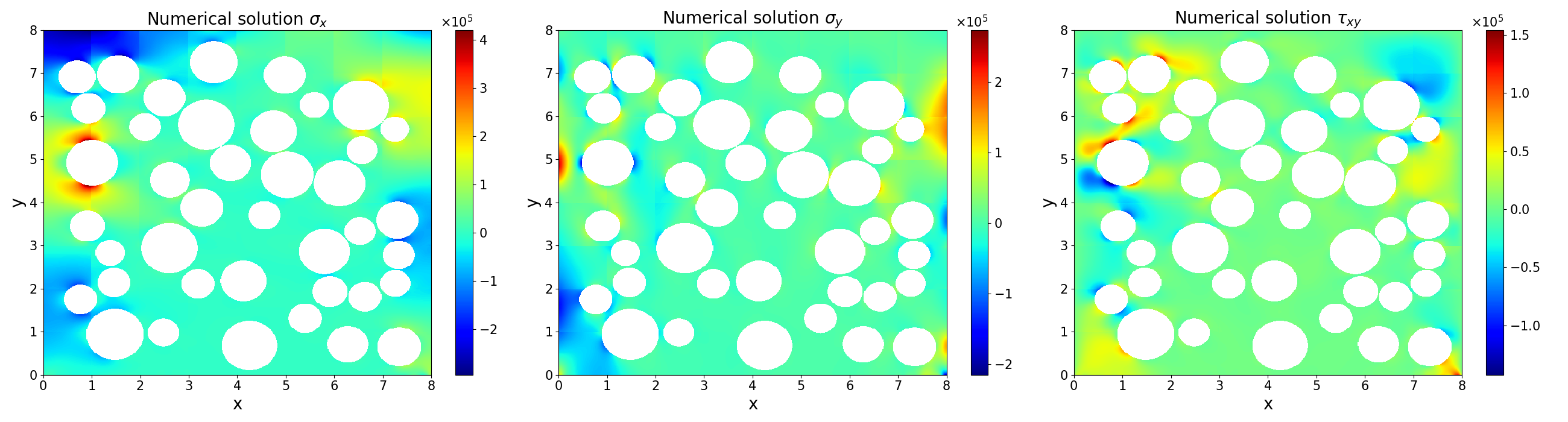}
    \end{subfigure}
    \caption{Reference solution obtained by the CSQRP-LSQR method with $\mathbf{x}^{(0)}=\mathbf{0}$ for the two-dimensional elasticity problem over a complex  geometry in Figure \ref{fig:complex_geometry_domain_elasticity}.}
    \label{fig:complex_geometry_domain}
  \end{figure}

Table \ref{tab:two-dimensional_the_homogenization_problem} compares the numerical performance of CSQRP-LSQR, CSSVDP-LSQR, CSSVD-PLSQR, LSRN-LSQR, and QR methods for the two-dimensional homogenization problem. The key metrics include the solution accuracy, the condition number of the preconditioned linear systems for CSQRP-LSQR and CSSVDP-LSQR, the iteration count, and the preprocessing time. It is observed that the condition numbers of $\mathbf{A}, \mathbf{A} \mathbf{R}^{-1}$, and $\mathbf{A P}$ range from $2.29 \times 10^{19}$ to $1.63 \times 10^{20}, 2.64 \times 10^3$ to $2.80 \times 10^5$, and 2.8233 to 5.4632, respectively. This indicates that while CSQRP-LSQR reduces condition numbers, some ill-conditioning remains, leading to slow LSQR convergence. However, the CSSVDP-LSQR method effectively transforms the highly ill-conditioned problem into a well-conditioned one during the preprocessing stage, greatly enhancing the stability and accuracy during iterations.
As the size of $\mathbf{A}$ increases, the CSSVDP-LSQR and CSSVD-PLSQR methods require significantly less computational time than the QR method. The preprocessing time for the CSSVD-PLSQR method is less than that of the CSSVDP-LSQR method; however, the latter achieves a solution accuracy that is 5 to 6 orders of magnitude higher than the former. Relative errors for CSSVDP-LSQR, QR, CSSVD-PLSQR, and LSRN-LSQR are approximately $10^{-10}, 10^{-5}, 10^{-4}$, and $10^{-4}$, respectively. Compared to the CSSVDP-LSQR method, the QR, LSRN-LSQR, and CSSVD-PLSQR methods exhibit numerical instability when solving severely ill-conditioned systems. This further indicates that explicitly forming a preconditioned matrix is a viable strategy for obtaining high-precision solutions to such highly ill-conditioned least squares problems.
In Table \ref{tab:two-dimensional_the_homogenization_problem}, ``$-$'' indicates memory overflow during the LSRN-LSQR preprocessing stage.

In Table \ref{tab:two-dimensional_the_homogenization_problem}, using the numerical solution from the largest parameter setting as the reference, the relative $L^2$ errors for $u$, $u_x$, and $u_y$ obtained by the CSSVDP-LSQR method are $3.23 \times 10^{-2}, 3.04 \times 10^{-2}$, and $3.12 \times 10^{-2}$, respectively, demonstrating clear numerical convergence. Figure \ref{fig:homogenization_A_AP_AN_singular_values_distrition} shows the singular value distribution of $\mathbf{A}$  with $n=303246$ and $m=19200$, $\mathbf{A R}^{-1}$, and $\mathbf{A P}$. The singular values of $\mathbf{A}$ are widespread, indicating a highly ill-conditioned system. However, the CSSVDP-LSQR method significantly improves this distribution through preconditioning, which is reflected in Table \ref{tab:two-dimensional_the_homogenization_problem}. Figure \ref{fig:the_two-dimensional_the_homogenization_problem} presents the numerical solutions and their first-order derivatives for the homogenization equation obtained using the CSSVDP-LSQR method.
\begin{table}[!htbp]
    \begin{center}
    \caption{Results of the two-dimensional homogenization problem using CSQRP-LSQR, CSSVDP-LSQR, CSSVD-PLSQR, LSRN-LSQR, and QR methods. The number of columns $n$ of $\mathbf{A}$ is 19200.}\label{tab:two-dimensional_the_homogenization_problem}%
    \setlength{\tabcolsep}{2.00mm} 
    \begin{tabular}{llllllll}
    \toprule
    Method& $m$ & $303246$ & $479644$ & $819604$ &$953034$ & $1023478$ & $1172032$\\
    \midrule 
    CSQRP-LSQR &$\kappa(\mathbf{A})$& $1.47e+20$ & $9.05e+19$ &$1.63e+20$ & $6.72e+19$ & $2.92e+19$ & $ 5.23e+19$\\
    &$\kappa(\mathbf{A}\mathbf{R}^{-1})$& $2.80e+05$ & $5.32e+03$ &$4.54e+03$ & $5.58e+03$ & $4.44e+03$ & $2.46e+03$\\
    \midrule
    CSSVDP-LSQR &$\kappa(\mathbf{A}\mathbf{P})$& $5.4632$ & $2.9929$ &$ 3.0157$ & $ 2.8735$ & $ 3.1643$ & $ 2.8233$\\
    &PCPU & $781.30$ & $861.85$ &$1036.53$ & $1161.91$ & $1171.12$ & $1238.09$\\
    &CPU & $59.00$ & $90.25$ &$148.46$ & $172.53$ & $188.36$ & $212.38$\\
    &TCPU & $840.30$ & $952.10$ &$1184.99$ & $1334.44$ & $1359.48$ & $1450.47$\\
    &IT & $30$ & $29$ &$28$ & $28$ & $28$ & $28$\\
    &$\frac{\|\mathbf{r}_k\|_2^2}{\|\mathbf{b}\|_2^2}$& $1.12e-09$ & $8.28e-10$ &$3.64e-10$ & $3.11e-10$ & $2.88e-10$ & $2.51e-10$\\
    &$u$ error & $3.54e-01$ & $9.51e-02$ &$6.93e-02$ & $3.85e-02$ & $3.23e-02$ & Reference\\
    &$u_{x}$ error & $3.15e-01$ & $1.31e-01$ &$5.88e-02$ & $3.49e-02$ & $3.04e-02$ & Reference\\
    &$u_{y}$ error & $3.24e-01$ & $1.38e-01$ &$5.93e-02$ & $3.58e-02$ & $3.12e-02$ & Reference\\
    \midrule
    CSSVD-PLSQR &PCPU & $681.39$ & $700.04$ &$761.89$ & $802.85$ & $951.13$ & $1051.13$\\
    &CPU & $95.42$ & $119.79$ &$168.13$ & $207.96$ & $228.46$ & $250.78$\\
    &TCPU & $776.81$ & $819.83$ &$930.02$ & $1010.75$ & $1179.59$ & $1301.91$\\
    &IT & $37$ & $29$ &$29$ & $28$ & $29$ & $28$\\
    &$\frac{\|\mathbf{r}_k\|_2^2}{\|\mathbf{b}\|_2^2}$& $3.09e-04$ & $3.45e-04$ &$3.29e-04$ & $2.74e-04$ & $2.70e-04$ & $2.94e-04$\\
    \midrule
    LSRN-LSQR &PCPU & $1538.29$ & $1810.62$ &$2370.41$ & $2599.71$ & $-$ & $-$\\
    &CPU & $87.85$ & $140.83$ &$167.93$ & $225.19$ & $-$ & $-$\\
    &TCPU & $1626.14$ & $1951.45$ &$2538.34$ & $2824.90$ & $-$ & $-$\\
    &IT & $28$ & $28$ &$28$ & $28$ & $-$ & $-$\\
    &$\frac{\|\mathbf{r}_k\|_2^2}{\|\mathbf{b}\|_2^2}$& $3.08e-04$ & $3.45e-04$ &$3.00e-04$ & $2.74e-04$ & $-$ & $-$\\
    \midrule 
    QR &TCPU & $624.86$ & $1123.06$ &$1566.07$ & $2804.10$ & $ 3103.46$ & $3163.95$\\
    &$\frac{\|\mathbf{r}_k\|_2^2}{\|\mathbf{b}\|_2^2}$& $5.92e-05$ & $5.73e-05$ &$7.97e-05$ & $5.65e-05$ & $ 5.17e-05$ & $ 7.81e-05$\\
    \bottomrule
    \end{tabular}
    \end{center}
\end{table}
\begin{figure}[!htbp]
  \renewcommand\figurename{Figure}
    \centering
    \includegraphics[scale=0.28]{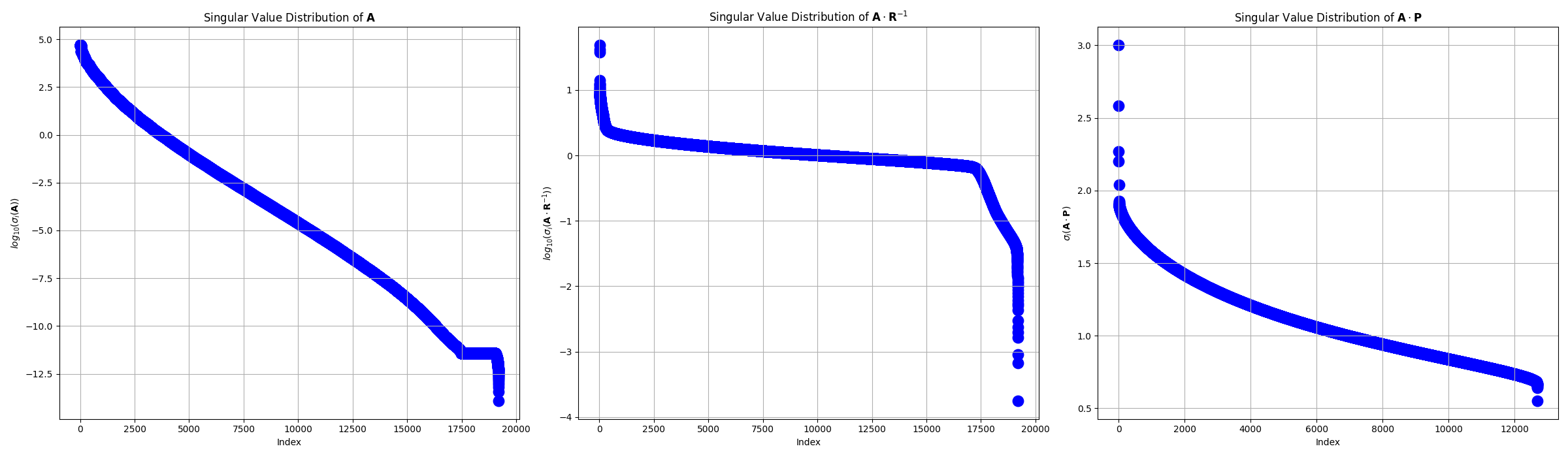}
  \caption{Distribution of singular values for $\mathbf{A}$, $\mathbf{A}\mathbf{R}^{-1}$, and $\mathbf{A}\mathbf{P}$ in the two homogenization problem.}
  \label{fig:homogenization_A_AP_AN_singular_values_distrition}
\end{figure}
\begin{figure}[!htbp]
  \renewcommand\figurename{Figure}
    \centering
    \includegraphics[scale=0.28]{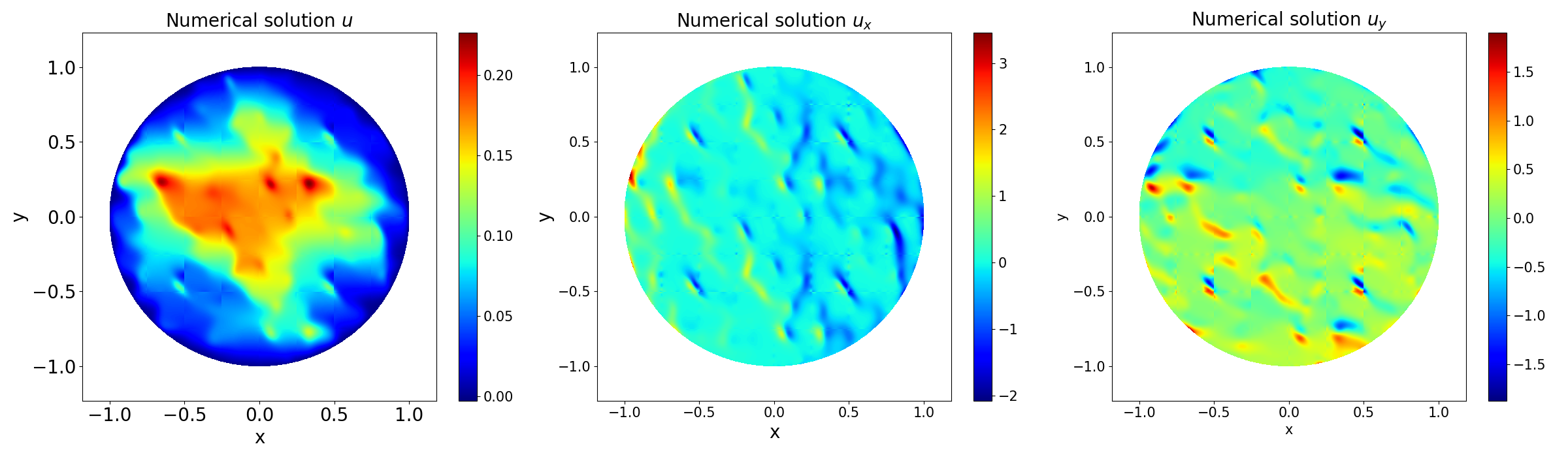}
  \caption{Reference solution and its first-order derivatives obtained by the CSSVDP-LSQR method for the two-dimensional homogenization problem.}
  \label{fig:the_two-dimensional_the_homogenization_problem}
\end{figure}

Results for two-dimensional Stokes flow with four sets of complex obstacles under inhomogeneous boundary conditions are presented in Tables in Tables \ref{tab:16_common_stokes_table}, \ref{tab:16_holes_complex_Stokes}, \ref{tab:20_complex_stokes_table}, and \ref{tab:15_holes_complex_Stokes}. These results include the preprocessing time, iteration counts, and the relative $L^2$ errors in the velocity fields $u, v$, and pressure $p$. Table \ref{tab:20_complex_stokes_table} shows that the condition numbers of $\mathbf{A}$ range from $7.85 \times$ $10^{17}$ to $3.46 \times 10^{18}$, indicating severe ill-conditioning. The condition numbers for $\mathbf{A} \mathbf{R}^{-1}$ and $\mathbf{A P}$ range from 117.42 to 218.40 and 2.7238 to 3.8653, respectively, indicating that the CSQRP-LSQR method is slightly less effective than the CSSVDP-LSQR method in handling highly ill-conditioned problems. Although the CSQRP-LSQR method has a shorter preprocessing time compared to the CSSVDP-LSQR method, the condition number of $\mathbf{A} \mathbf{R}^{-1}$ remains high (ranging from 117.42 to 218.40), leading to 652 to 750 iterations. In contrast, the CSSVDP-LSQR method shows excellent preconditioning, resulting in significantly fewer iterations and shorter times compared to CSQRP-LSQR. Table \ref{tab:16_common_stokes_table} lists the relative $L^2$ errors of the velocity field $u, v$ and pressure $p$ for the CSQRP-LSQR and CSSVDP-LSQR methods, showing clear numerical convergence. Figures \ref{fig:16_common_holes_Stokes_problem} and \ref{fig:16_com_holes_Stokes_problem} present the reference solutions obtained by the CSQR-PLSQR and CSSVDP-LSQR methods, respectively.

Tables \ref{tab:16_holes_complex_Stokes}, \ref{tab:20_complex_stokes_table}, and \ref{tab:15_holes_complex_Stokes} present the numerical results of the CSSVDP-LSQR method for the other three obstacles, which are consistent with the observations in Table \ref{tab:16_common_stokes_table}. The method converges within 24 to 28 iterations. The relative $L^2$ errors for $u, v$, and $p$ show a clear convergence trend.  Figures \ref{fig:16_holes_complex_Stokes}, \ref{fig:20_holes_complex_Stokes}, and \ref{fig:15_holes_complex_Stokes} visualize the reference solutions for $u, v$, and $p$ in channel flows with non-uniformly distributed obstacles.
\begin{table}[!htbp]
  \begin{center}
  \caption{Results of the Stokes flow over complex geometry \ref{fig:16_com_holes_Stokes_problem} using CSQRP-LSQR and CSSVDP-LSQR methods. The column size of $\mathbf{A}$ is 22500.}\label{tab:16_common_stokes_table}%
  \setlength{\tabcolsep}{2.50mm}  
  \begin{tabular}{llllllll}
  \toprule
  Method&  $m$ & $313476$ & $445373$ & $599468$  &$776977$  & $840340$ & $906461$\\
  \midrule 
  CSQRP-LSQR &$\kappa(\mathbf{A})$& $1.32e+18$ & $1.19e+18$ &$1.01e+18$  & $ 1.22e+18$ & $7.85e+17$ & $3.46e+18$\\
  &$\kappa(\mathbf{A}\mathbf{R}^{-1})$& $218.40$ & $128.78$ &$ 127.88$  & $129.89$ & $117.42$ & $ 136.87$\\
  &PCPU & $ 440.16$ & $568.63$ &$711.32$  & $ 855.66$ & $ 921.18$ & $993.60$\\
  &CPU &$2860.43$ & $3619.91$ &$ 4617.92$  & $6197.32$ & $6215.43$ & $ 7058.17$\\
  &TCPU & $3300.59$ & $4188.53$ &$5329.24$  & $7052.98$ & $7136.61$ & $8051.77$\\
  &IT & $780$ & $706$ &$674$  & $700$ & $652$ & $684$\\
  &$u$ error  & $2.29e-02$ & $2.82e-02$ &$2.03e-02$  & $1.68e-02$ & $1.86e-02$ & Reference\\
  &$v$ error & $2.78e-02$ & $2.60e-02$ &$2.60e-02$  & $2.49e-02$ & $2.52e-02$ & Reference\\
  &$p$ error & $1.39e-02$ & $1.14e-02$ &$7.83e-03$  & $5.57e-03$ & $3.00e-03$ & Reference\\
  \midrule
  CSSVDP-LSQR &$\kappa(\mathbf{A}\mathbf{P})$& $3.8653$ & $3.0579$ &$2.9653$  & $2.7856$ & $2.7238$ & $2.8207$\\
  &PCPU & $1135.58$ & $1184.20$ &$1304.90$  & $1424.90$ & $ 1534.55$ & $1557.37$\\
  &CPU & $62.34$ & $81.20$ &$ 119.86$  & $ 129.89$ & $144.77$ & $ 156.01$\\
  &TCPU & $1197.92$ & $1265.40$ &$1424.76$  & $1554.79$ & $1679.32$ & $1713.38$\\
  &IT & $28$ & $25$ &$25$  & $24$ & $24$ & $24$\\
  &$u$ error & $8.25e-02$ & $5.82e-02$ &$4.44e-02$  & $9.83e-03$ & $1.37e-02$ & Reference\\
  &$v$ error & $1.15e-01$ & $8.26e-02$ &$7.49e-02$  & $3.09e-02$ & $4.09e-02$ & Reference\\
  &$p$ error & $1.20e-01$ & $6.91e-02$ &$6.17e-02$  & $3.25e-02$ & $2.30e-02$ & Reference\\
  \bottomrule
  \end{tabular}
  \end{center}
\end{table}
\begin{table}[!htbp]
  \begin{center}
    \caption{Results of the Stokes flow over a complex geometry in Figure \ref{fig:16_holes_complex_Stokes} using the CSSVDP-LSQR method. The column size of $\mathbf{A}$ is 22500.}\label{tab:16_holes_complex_Stokes}%
  \setlength{\tabcolsep}{4.50mm} 
  \begin{tabular}{lllllll}
  \toprule
  $m$ & $336111$ & $477800$ & $643604$  &$835369$  & $903583$ & $975620$\\
  \midrule
  PCPU & $1201.83$ & $1249.37$ & $1334.05$  &$1467.43$  & $1562.45$ & $1575.99$\\
  CPU & $71.24$ & $94.06$ & $121.75$  &$147.00$  & $166.39$ & $174.85$\\
  TCPU & $1273.07$ & $1343.43$ & $1455.80$  &$1614.43$  & $1728.84$ & $1750.84$\\
  IT & $28$ & $27$ & $26$  &$24$  & $24$ & $24$\\
  $u$ error & $9.99e-02$ & $7.58e-02$ &$5.38e-02$  & $1.69e-02$ & $2.61e-02$ & Reference\\
  $v$ error & $1.11e-01$ & $8.29e-02$ &$5.46e-02$  & $2.20e-02$ & $3.60e-02$ & Reference\\
  $p$ error & $1.95e-01$ & $1.47e-01$ &$1.20e-01$  & $3.01e-02$ & $5.67e-02$ & Reference\\
  \bottomrule
  \end{tabular}
  \end{center}
\end{table}
\begin{table}[!htbp]
  \begin{center}
  \caption{The numerical results obtained by solving the Stokes flow over complex geometry \ref{fig:20_holes_complex_Stokes} using the CSSVDP-LSQR method with $\gamma=3$. The number of columns $n$ of the coefficient matrix $\mathbf{A}$ is 22500.}\label{tab:20_complex_stokes_table}%
  \setlength{\tabcolsep}{4.50mm}  
  \begin{tabular}{lllllll}
  \toprule
  $m$ & $349709$ & $496998$ & $669864$  &$869509$  & $940488$ & $1015462$\\
  \midrule
  PCPU & $1177.86$ & $1233.16$ &$1330.98$  & $1528.53$ & $1587.18$ & $1660.68$\\
  CPU & $68.32$ & $108.52$ &$121.59$  & $158.98$ & $172.02$ & $182.69$\\
  TCPU & $1246.18$ & $1341.68$ &$1452.57$  & $1678.51$ & $1759.20$ & $1843.37$\\
  IT & $26$ & $25$ &$24$  & $24$ & $24$ & $24$\\
  $u$ error & $1.18e-01$ & $8.02e-02$ &$6.55e-02$  & $4.10e-02$ & $1.83e-02$ & Reference\\
  $v$ error & $2.24e-01$ & $1.52e-01$ &$1.19e-01$  & $8.93e-02$ & $6.70e-02$ & Reference\\
  $p$ error & $9.41e-01$ & $8.41e-01$ &$5.57e-01$  & $9.78e-01$ & $2.57e-01$ & Reference\\
  \bottomrule
  \end{tabular}
  \end{center}
\end{table}
\begin{table}[!htbp]
  \begin{center}
    \caption{Results of the Stokes flow over a complex geometry in Figure \ref{fig:15_holes_complex_Stokes} using the CSSVDP-LSQR method. The column size of $\mathbf{A}$ is 22500.}\label{tab:15_holes_complex_Stokes}%
  \setlength{\tabcolsep}{4.50mm}  
  \begin{tabular}{lllllll}
  \toprule
  $m$ & $409756$ & $583696$ & $788104$  &$1023784$  & $1108346$ & $1196940$\\
  \midrule
  PCPU & $1312.34$ & $1345.58$ & $1579.68$  &$1617.98$  & $1663.33$ & $1763.69$\\
  CPU & $86.40$ & $118.25$ & $154.32$  &$191.86$  & $204.38$ & $241.20$\\
  TCPU & $1398.74$ & $1463.83$ & $1734.00$  &$1809.84$  & $1867.71$ & $2004.89$\\
  IT & $26$ & $26$ & $25$  &$24$  & $24$ & $24$\\
  $u$ error & $1.28e-02$ & $1.10e-02$ &$8.45e-03$  & $7.53e-03$ & $4.78e-03$ & Reference\\
  $v$ error & $4.33e-02$ & $3.06e-02$ &$2.46e-02$  & $2.86e-02$ & $2.13e-02$ & Reference\\
  $p$ error & $2.58e-02$ & $1.69e-02$ &$7.29e-03$  & $4.94e-02$ & $3.49e-02$ & Reference\\
  \bottomrule
  \end{tabular}
  \end{center}
\end{table}
\begin{figure}[!htbp]
  \renewcommand\figurename{Figure}
    \centering
    \includegraphics[scale=0.26]{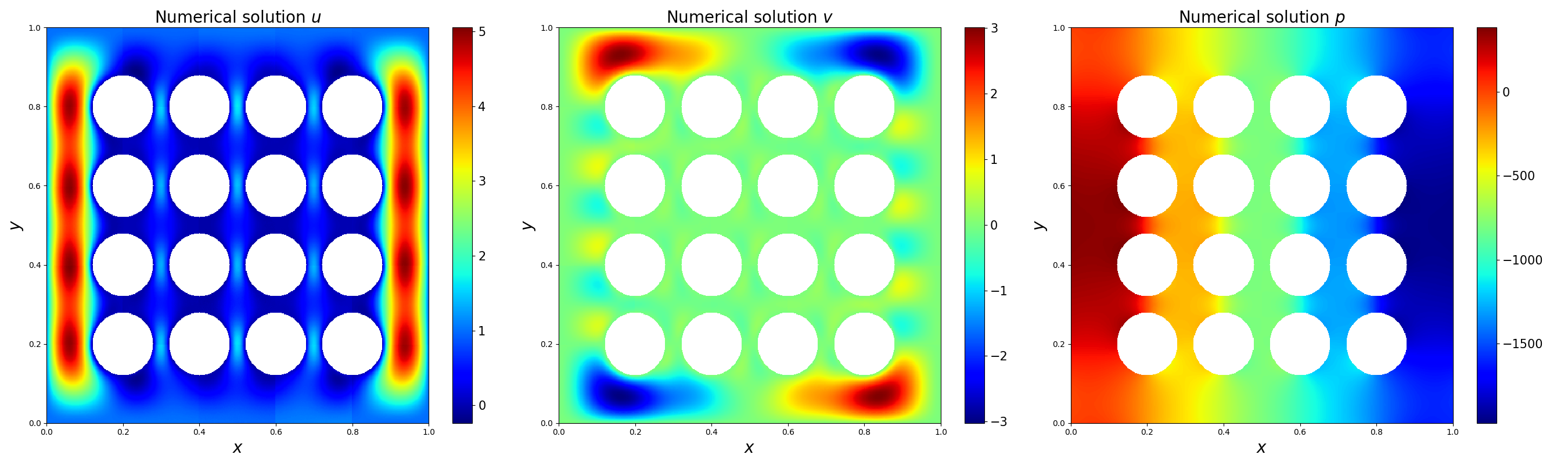}
  \caption{Reference solution by the CSQRP-LSQR method for Stokes flow problem over the complex geometry.}
  \label{fig:16_common_holes_Stokes_problem}
\end{figure}
\begin{figure}[!htbp]
  \renewcommand\figurename{Figure}
    \centering
    \includegraphics[scale=0.26]{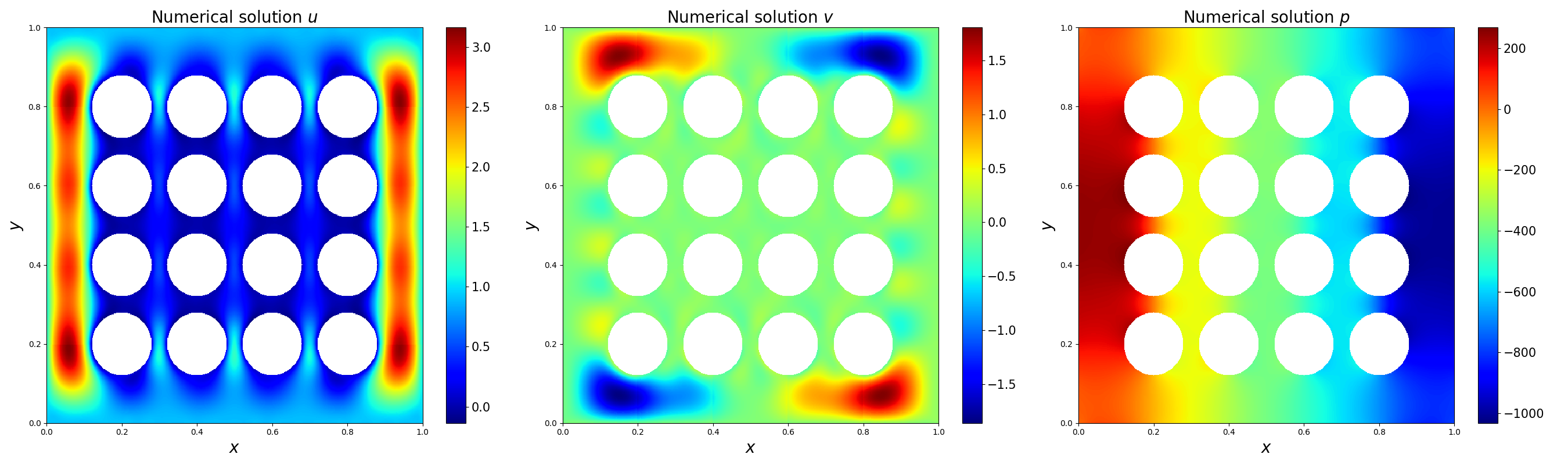}
  \caption{Reference solution obtained by the CSSVDP-LSQR method for Stokes flow problem over the complex geometry.}
  \label{fig:16_com_holes_Stokes_problem}
\end{figure}
\begin{figure}[!htbp]
  \renewcommand\figurename{Figure}
    \centering
    \includegraphics[scale=0.26]{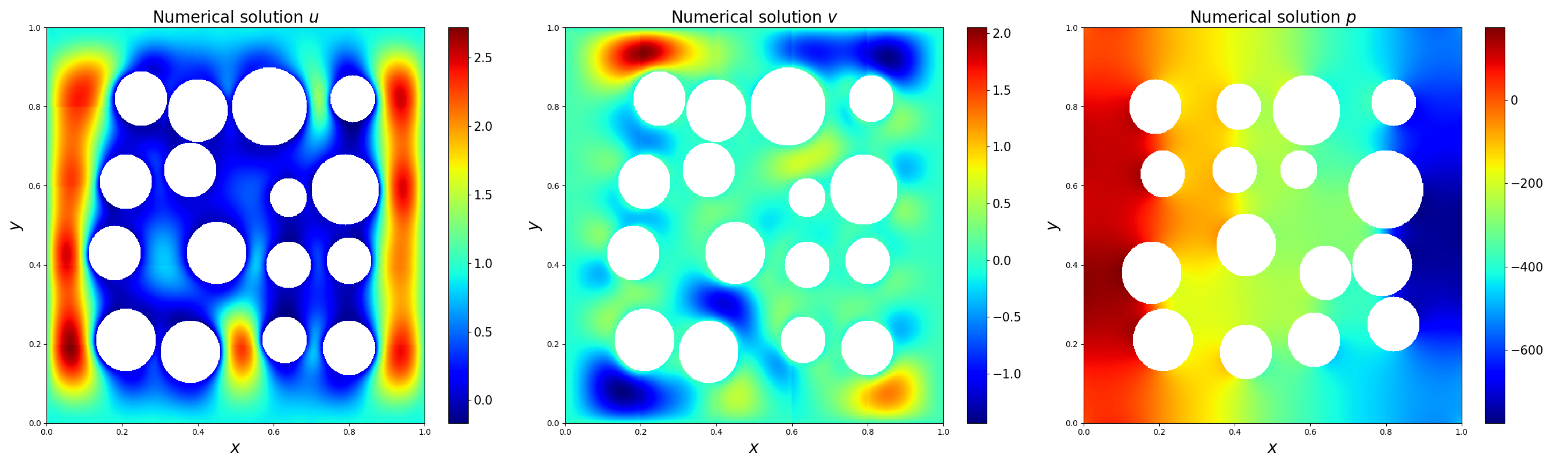}
  \caption{Reference solution by the CSSVDP-LSQR method  for Stokes flow problem over the complex geometry.}
  \label{fig:16_holes_complex_Stokes}
\end{figure}
\begin{figure}[!htbp]
  \renewcommand\figurename{Figure}
    \centering
    \includegraphics[scale=0.26]{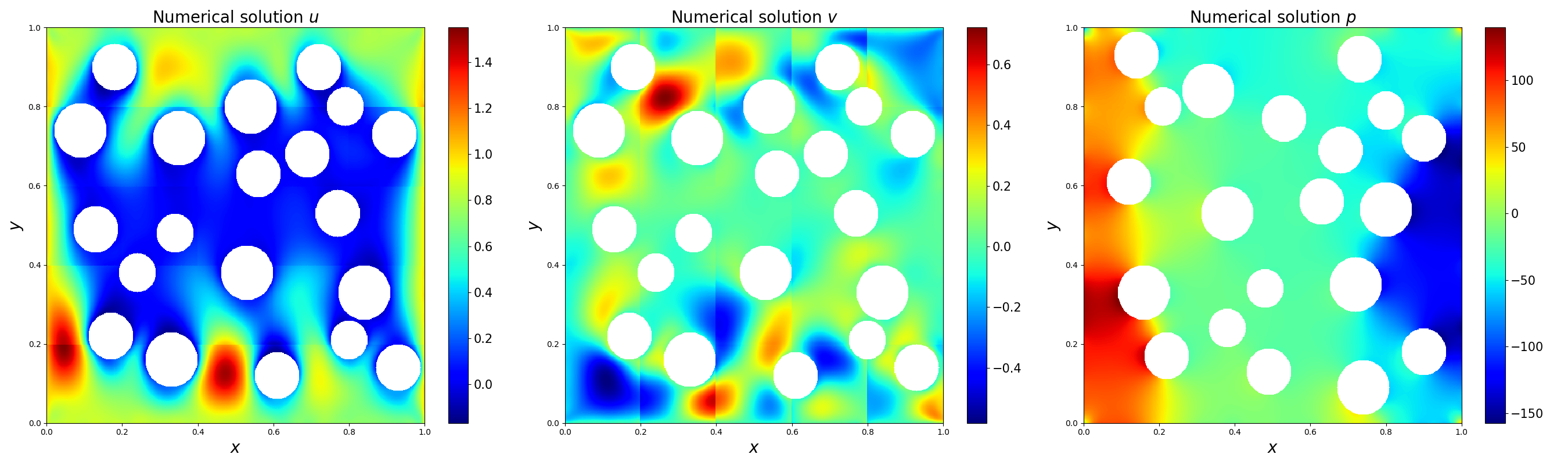}
  \caption{Reference solution by the CSSVDP-LSQR method for Stokes flow problem over the complex geometry.}
  \label{fig:20_holes_complex_Stokes}
\end{figure}
\begin{figure}[!htbp]
  \renewcommand\figurename{Figure}
    \centering
    \includegraphics[scale=0.26]{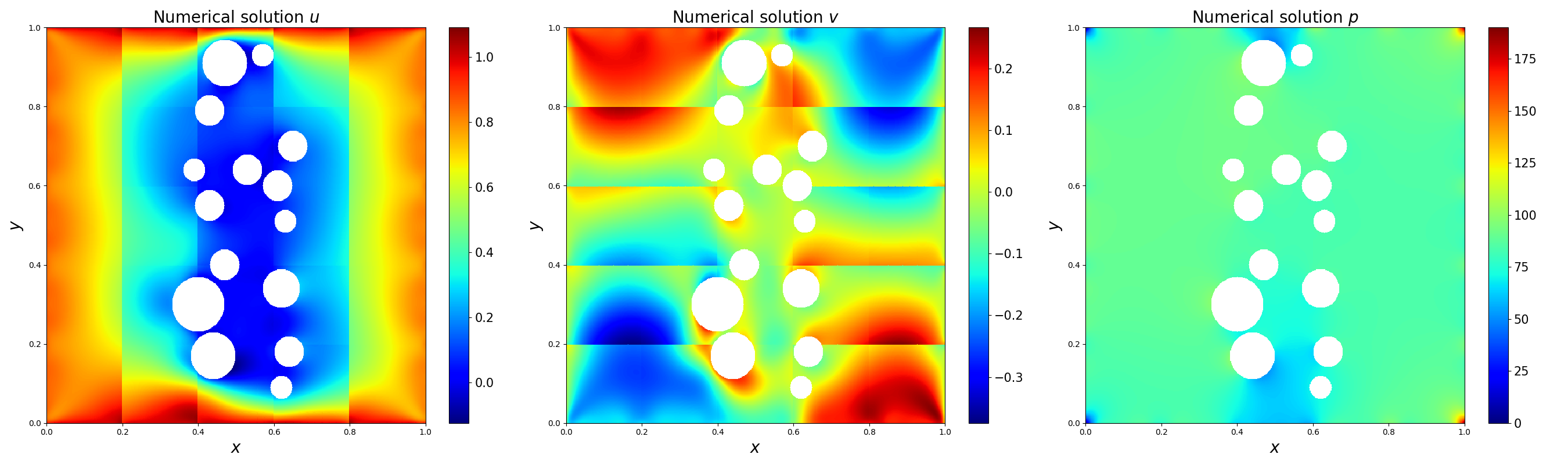}
  \caption{Reference solution by the CSSVDP-LSQR method for Stokes flow problem over the complex geometry.}
  \label{fig:15_holes_complex_Stokes}
\end{figure}
\par

Table \ref{tab:the three-dimensional Homogenization equation} presents the numerical performance of the CSQRP-LSQR method for three-dimensional homogenization problems with different $R$ values. As $R$ increases, the preprocessing time and the total computation time generally increase, while the iteration count slightly decreases, ranging from 30 to 41 iterations. The relative $L^2$ and $H^1$ errors for $u$ also increase. For  $R=2, 2.5,3, 4$, the relative $L^2$ errors for $u$ are $2.10 \times 10^{-3}, 8.69 \times 10^{-2}, 1.64 \times 10^{-1}$, and $3.01 \times 10^{-1}$, respectively, while the relative $H^1$ errors are $8.39 \times 10^{-3}, 3.48 \times 10^{-1}, 6.57 \times 10^{-1}$, and 1.20. We attribute this to the higher-frequency components in the exact solution as $R$ increases. Despite this, the CSQRP-LSQR method works well with a reasonable number of iterations.

Three different hyper-parameter settings in the RFM for the three-dimensional Poisson equation are used, with results reported in Table \ref{tab:the three-dimensional Poisson equation}. As the number of columns increases, both the preprocessing time and the total computation time increase significantly. The iteration count remains relatively stable, ranging from 34 to 36 iterations. The iteration counts indicate that the CSQRP-LSQR method works well. In Case II, using the solution with the maximum parameters as the reference, the relative $L^2$ and $H^1$ errors for $u$ are $8.36 \times 10^{-3}$ and $3.34 \times 10^{-2}$, respectively, while in the other two cases, both errors are around $10^{-2}$. Figure \ref{fig:three-dimensional_Poisson_equation_solution} shows the reference solution on a specific cross-section in Case II.
\begin{table}[!htbp]
  \begin{center}
  \caption{Results of the CSQRP-LSQR method for the three-dimensional homogenization problem with different $R$ values. The number of rows $m$ of $\mathbf{A}$ is 1264329.}\label{tab:the three-dimensional Homogenization equation}%
  \setlength{\tabcolsep}{2.50mm}  
  \begin{tabular}{lllllllll}
  \toprule
  $R$ &  $n$ & $5400$ & $10800$ & $16200$  &$21600$  & $27000$ & $32400$ & $37800$\\
  \midrule 
  $2$ &PCPU & $168.48$ & $418.87$ &$771.33$  & $1241.96$ & $1795.40$ & $2714.18$ & $3761.33$\\
  &CPU &$40.95$ & $101.06$ &$132.53$  & $167.96$ & $276.19$ & $304.87$ & $328.90$\\
  &TCPU & $209.43$ & $519.93$ &$903.86$  & $1409.92$ & $2071.59$ & $3019.05$ & $4090.23$\\
  &IT & $34$ & $36$ &$37$  & $38$ & $39$ & $41$ & $41$\\
  &$u$ error & $1.97e-01$ & $4.11e-02$ &$1.21e-02$  & $5.94e-03$ & $3.38e-03$ & $2.10e-03$ & Reference\\
  &$u$ $H^1$ & $7.86e-01$ & $1.64e-01$ &$4.85e-02$  & $2.37e-02$ & $1.35e-02$ & $8.39e-03$ & Reference\\
  \midrule
  $2.5$ &PCPU & $169.62$ & $459.77$ &$796.77$  & $1309.05$ & $1846.12$ & $2742.13$ & $3707.50$\\
  &CPU &$37.04$ & $97.62$ &$102.89$  & $155.72$ & $261.37$ & $272.62$ & $293.11$\\
  &TCPU & $206.66$ & $557.39$ &$899.66$  & $1464.77$ & $2107.49$ & $3014.75$ & $4000.61$\\
  &IT & $32$ & $32$ &$33$  & $34$ & $34$ & $35$ & $35$\\
  &$u$ error & $5.85e-01$ & $4.96e-01$ &$3.77e-01$  & $2.54e-01$ & $1.62e-01$ & $8.69e-02$ & Reference\\
  &$u$ $H^1$ & $2.34e+00$ & $1.98e+00$ &$1.51e+00$  & $1.02e+00$ & $6.48e-01$ & $3.48e-01$ & Reference\\
  \midrule
  $3.0$ &PCPU & $170.50$ & $426.84$ &$786.97$  & $1294.97$ & $1872.55$ & $2681.33$ & $3735.70$\\
  &CPU &$35.16$ & $78.86$ &$101.03$  & $145.94$ & $225.93$ & $255.92$ & $290.42$\\
  &TCPU & $205.66$ & $505.70$ &$888.00$  & $1440.91$ & $2098.48$ & $2937.25$ & $4026.12$\\
  &IT & $31$ & $31$ &$32$  & $33$ & $33$ & $33$ & $34$\\
  &$u$ error & $5.99e-01$ & $4.91e-01$ &$4.03e-01$  & $3.08e-01$ & $2.32e-01$ & $1.64e-01$ & Reference\\
  &$u$ $H^1$ & $2.40e+00$ & $1.97e+00$ &$1.61e+00$  & $1.23e+00$ & $9.29e-01$ & $ 6.57e-01$ & Reference\\
  \midrule
  $4.0$ &PCPU & $170.91$ & $420.62$ &$789.71$  & $1345.97$ & $1872.02$ & $2698.14$ & $3798.79$\\
  &CPU &$32.16$ & $75.01$ &$97.09$  & $125.53$ & $194.83$ & $ 257.60$ & $287.49$\\
  &TCPU & $203.07$ & $495.63$ &$886.80$  & $1471.50$ & $2066.85$ & $2955.74$ & $4086.28$\\
  &IT & $30$ & $30$ &$31$  & $31$ & $31$ & $32$ & $32$\\
  &$u$ error & $7.69e-01$ & $6.34e-01$ &$5.42e-01$  & $4.60e-01$ & $3.69e-01$ & $3.01e-01$ & Reference\\
  &$u$ $H^1$ & $3.07e+00$ & $2.54e+00$ &$2.17e+00$  & $1.84e+00$ & $1.47e+00$ & $1.20e+00$ & Reference\\
  \bottomrule
  \end{tabular}
  \end{center}
\end{table}
\begin{table}[!htbp]
  \begin{center}
  \caption{Results of the CSQRP-LSQR method for the three-dimensional Poisson equation over a complex geometry.}\label{tab:the three-dimensional Poisson equation}%
  \setlength{\tabcolsep}{3.50mm}
  \begin{tabular}{llllllll}
  \toprule
  Case &$(m, n)$ & PCPU & CPU & TCPU & IT & $u$ error & $u$ $H^{1}$\\
  \midrule
  I &$(1052784,1600)$ & $76.90$ & $ 12.25$ & $89.15$ & $34$ & $7.90e-02$ & $3.16e-01$\\
  &$(1052784,3200)$ & $163.16$ & $23.79$ & $186.95$ & $35$ & $5.05e-02$ & $2.02e-01$\\
  &$(1052784,4800)$ & $268.08$ & $33.46$ & $301.54$ & $35$ & $3.62e-02$ & $1.45e-01$\\
  &$(1052784,6400)$ & $366.65$ & $ 48.56$ & $415.21$ & $36$ & $3.13e-02$ & $1.25e-01$\\
  &$(1052784,8000)$ & $503.47$ & $46.97 $ & $550.44$ & $36$ & $2.80e-02$ & $1.12e-01$\\
  &$(1052784,9600)$ & $620.88$ & $66.59$ & $687.47$ & $36$ & $2.39e-02$ & $9.56e-02$\\
  &$(1052784,11200)$ & $746.17$ & $75.90$ & $822.07$ & $36$ & $1.69e-02$ & $6.78e-02$\\
  &$(1052784,12800)$ & $908.47$ & $109.85$ & $1018.32$ & $36$ & \multicolumn{2}{c}{Reference}\\
  \midrule
  II &$(1408262,1600)$ & $106.81$ & $16.51$ & $123.32$ & $34$ & $7.73e-02$ & $3.09e-01$\\
  &$(1408262,3200)$ & $226.14$ & $33.18$ & $259.32$ & $35$ & $4.59e-02$ & $1.83e-01$\\
  &$(1408262,4800)$ & $354.96$ & $45.55$ & $400.51$ & $36$ & $2.60e-02$ & $1.04e-01$\\
  &$(1408262,6400)$ & $507.11$ & $54.79$ & $561.90$ & $36$ & $1.84e-02$ & $7.36e-02$\\
  &$(1408262,8000)$ & $660.80$ & $74.82 $ & $735.62$ & $36$ & $1.30e-02$ & $5.20e-02$\\
  &$(1408262,9600)$ & $807.09$ & $78.40$ & $885.49$ & $36$ & $8.36e-03$ & $3.34e-02$\\
  &$(1408262,10400)$&$878.99$ & $83.86$ & $962.85$ & $36$ & \multicolumn{2}{c}{Reference}\\
  \midrule
  III &$(1704044,1600)$ & $126.56$ & $20.75$ & $147.31$ & $34$ & $7.71e-02$ & $3.09e-01$\\
  &$(1704044,3200)$ & $265.53$ &  $37.25$ & $302.78$ & $35$ & $4.43e-02$ & $1.77e-01$\\
  &$(1704044,4800)$ & $435.81$ & $54.68$ & $490.49$ & $36$ & $2.31e-02$ & $9.25e-02$\\
  &$(1704044,6400)$ & $587.09$ & $79.66$ & $666.75$ & $36$ & $1.38e-02$ & $5.51e-02$\\
  &$(1704044,7200)$ & $670.08$ & $85.39$ & $755.47$ & $36$ & $1.00e-02$ & $4.01e-02$\\
  &$(1704044,8000)$ & $764.33$ & $91.50$ & $855.83$ & $36$ &  \multicolumn{2}{c}{Reference}\\
  \bottomrule
  \end{tabular}
  \end{center}
\end{table}
\begin{figure}[!htbp]
  \renewcommand\figurename{Figure}
    \centering
    \includegraphics[scale=0.40]{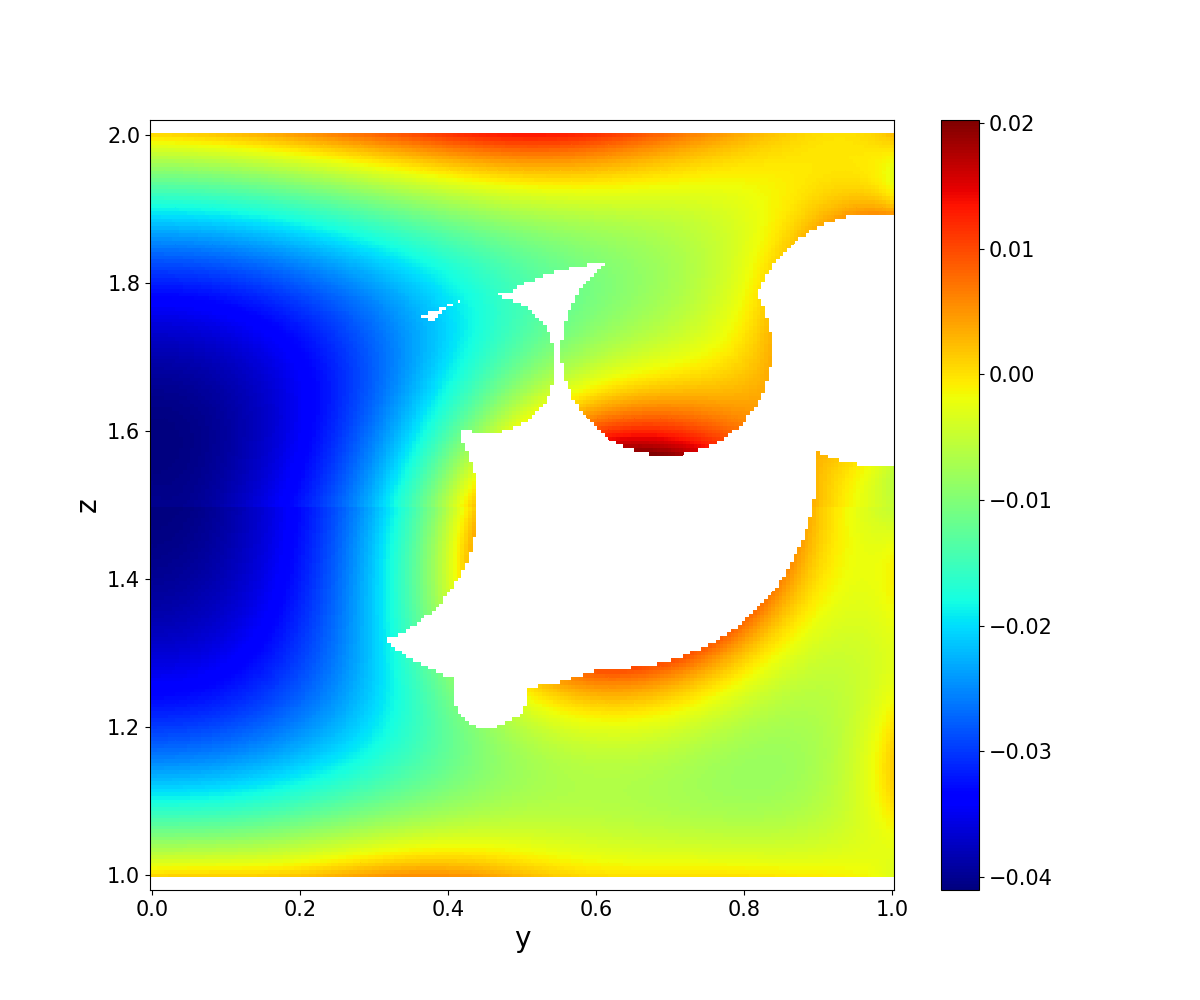}
  \caption{The reference solution (on the semi-section) of the three-dimensional Poisson problem \eqref{eq:complex_Poisson} obtained using the CSQRP-LSQR method.}
  \label{fig:three-dimensional_Poisson_equation_solution}
\end{figure}

In summary, we apply the proposed CSQRP-LSQR method to solve ill-conditioned least squares problems obtained by the RFM for two- and three-dimensional PDEs without explicit solutions. While results confirm that our method effectively solves these ill-conditioned least squares problems, it is evident that RFM faces challenges for these complex problems, particularly in three-dimensional cases. 

\section{Hyper-parameters in the RFM}\label{my_hyperparameters_RFM}
In all experiments, $R_{n j}=1$, $tanh$ is selected as the activation function, and the partition of unity (PoU) function employed is $\psi^a(\mathbf{x})$. When $d=2$,  the domain $\Omega$ is divided into $M_p = N_xN_y$ uniformly partitioned sub-domains along the $x$ and $y$ axes, with $N_x(N_x \geqslant 1)$ and $N_y(N_y \geqslant 1)$ representing the number of sub-domains, respectively. We choose a set of points $\left\{\mathbf{x}_n\right\}_{n=1}^{M_p} \subset \Omega$, where each point $\mathbf{x}_n$ serves as the center of its respective sub-domain $\Omega_n$ for constructing the PoU functions. For each point $\mathbf{x}_n$, we construct $J_n$ random feature functions with radius $\mathbf{r}_n$. We then generate $Q$ equally spaced collocation points within each sub-domain $\Omega_n$, discarding any point outside of $\Omega$. A similar setup is used when $d=3$. When $M_p>1$, $\mathbf{A}$ is sparse, and its sparsity increases as $M_p$ increases. 

\begin{table}[!htbp]
  \begin{center}
  \caption{Hyper-parameters in the RFM for the elasticity and Stokes flows problems in two dimensions.}\label{tab:Numerical_parameters_table}%
  \setlength{\tabcolsep}{4.50mm}  
  \begin{tabular}{llllll}
  \toprule
  Problem & $(N_x, N_y)$ & $\sqrt{Q}$ & $J_n$  & $\kappa(\mathbf{A})$ & $(m, n)$  \\
  \midrule
  Timoshenko beam (Table \ref{Timoshenko_beam_table}) & $(5,5)$ & $50$ &$200$  & $9.1689e+13$  & $(139000, 10000)$\\
  2D Stokes flow (Table \ref{CSSVD-PLSQR_QR_SVD_for_Stokes_flow}) & $(4,4)$ & $60$ &$300$  & $8.2740e+17$ &  $(169564, 14400)$\\
  2D Stokes flow (Table \ref{complex_geometric_2D_Stokes_flow_problem}) & $(5,5)$ & $40$ &$120$  &$1.8662e+13$ &  $(122238, 9000)$ \\
  2D Elasticity (Table \ref{complex_elasticity_table}) & $(6,6)$ & $60$ &$150$  &  $5.5545e+12$ & $(140060, 10800)$ \\
  Timoshenko beam (Table \ref{Randomized_iterative_LSQR_LSMR_for_elasticity_numerical_results}) & $(5,5)$ & $50$ &$200$  & $9.1689e+13$  & $(139000, 10000)$ \\
  2D Stokes flow (Table \ref{Randomized_iterative_LSQR_LSMR_for_Stokes_flow}) & $(4,4)$ & $60$ &$300$  & $8.2740e+17$ &  $(169564, 14400)$\\
  \toprule
  Problem & \multicolumn{4}{c}{$\mathbf{x}_n$} & $ \mathbf{r}_n $  \\
  \midrule
  Timoshenko beam (Table \ref{Timoshenko_beam_table}) & \multicolumn{4}{c}{$\left\{(-4+2i, 1+2j), i,j \in \left\{0,1,2,3,4\right\}\right\}$} &  $(1, 1)$ \\
  2D Stokes flow (Table \ref{CSSVD-PLSQR_QR_SVD_for_Stokes_flow}) & \multicolumn{4}{c}{$\left\{\left(\frac{1+2i}{8}, \frac{1+2j}{8}\right), i,j \in \left\{0,1,2,3\right\}\right\}$} &  $\left(\frac{1}{8}, \frac{1}{8}\right)$\\
  2D Stokes flow (Table \ref{complex_geometric_2D_Stokes_flow_problem}) & \multicolumn{4}{c}{$\left\{\left(\frac{1+2i}{10}, \frac{1+2j}{10}\right), i,j \in \left\{0,1,2,3, 4\right\}\right\}$}&  $(\frac{1}{10}, \frac{1}{10})$ \\
  2D Elasticity (Table \ref{complex_elasticity_table}) & \multicolumn{4}{c}{$\left\{\left(\frac{2+4i}{3}, \frac{2+4j}{3}\right), i,j \in \left\{0,1,2,3,4, 5\right\}\right\}$} & $(\frac{2}{3}, \frac{2}{3})$ \\
  Timoshenko beam (Table \ref{Randomized_iterative_LSQR_LSMR_for_elasticity_numerical_results}) & \multicolumn{4}{c}{$\left\{(-4+2i, 1+2j), i,j \in \left\{0,1,2,3,4\right\}\right\}$} &  $(1, 1)$ \\
  2D Stokes flow (Table \ref{Randomized_iterative_LSQR_LSMR_for_Stokes_flow}) & \multicolumn{4}{c}{$\left\{\left(\frac{1+2i}{8}, \frac{1+2j}{8}\right), i,j \in \left\{0,1,2,3\right\}\right\}$} &  $\left(\frac{1}{8}, \frac{1}{8}\right)$\\
  \bottomrule
  \end{tabular}
  \end{center}
\end{table}
\begin{table}[!htbp]
  \begin{center}
  \caption{Hyper-parameters in the RFM for the three-dimensional PDEs.}\label{tab:hyper-parameters_three-dimensional_PDEs}%
  \setlength{\tabcolsep}{2.50mm} 
  \begin{tabular}{lllll}
  \toprule
  Problem  & $\sqrt[3]{Q} $ & $J_n$  & $(N_x, N_y, N_z)$\\
  \midrule
  3D Poisson (Table \ref{tab:three-dimensional_Poisson_equation})& $\left\{20,20,25,25, 30, 30\right\}$& $\left\{400, 800, 800, 1200, 1600, 2000\right\}$ &$(2, 2, 2)$\\
  3D Helmholtz (Table \ref{tab:the three-dimensional Helmholtz equation}) 
  & $\left\{20,20,30,30,35, 35\right\}$& $\left\{400, 800, 800, 1000, 1200, 1400\right\}$ &$(2, 2, 2)$\\
  3D Elasticity (Table \ref{tab:three-dimensional_elasticity_problem})& $\left\{20, 20, 25, 25\right\}$ & $\left\{600, 600, 800, 1000, 1200\right\}$ & $(2,2,2)$\\
  3D Elasticity (Table \ref{tab:three-dimensional_elasticity_problem}-\ref{tab:three-dimensional_elasticity})& $\left\{20, 30\right\}$ & $\left\{400, 600\right\}$ & $(1,1,1)$\\
  3D Stokes (Table \ref{tab:three-dimensional Stokes flow})& $\left\{20, 20, 25, 25\right\}$ & $\left\{600, 600, 800, 1000, 1200\right\}$ & $(2,2,2)$\\
  3D Stokes (Table \ref{tab:three-dimensional Stokes flow}-\ref{tab:3Dtokes})& $\left\{20, 30\right\}$ & $\left\{800, 800\right\}$ & $(1,1,1)$\\
  \toprule
  Problem & \multicolumn{2}{c}{$\mathbf{x}_n$} & $ \mathbf{r}_n $  \\
  \midrule
  3D Poisson (Table \ref{tab:three-dimensional_Poisson_equation})& \multicolumn{2}{c}{$\left\{\left(\frac{1+2i}{4}, \frac{1+2j}{4}, \frac{1+2k}{4}\right), i,j, k\in \left\{0,1\right\}\right\}$} &  $\left(\frac{1}{4},\frac{1}{4},\frac{1}{4} \right)$ \\
  3D Helmholtz (Table \ref{tab:the three-dimensional Helmholtz equation})& \multicolumn{2}{c}{$\left\{\left(\frac{1+2i}{4}, \frac{1+2j}{4}, \frac{1+2k}{4}\right), i,j, k\in \left\{0,1\right\}\right\}$} &  $\left(\frac{1}{4},\frac{1}{4},\frac{1}{4} \right)$ \\
  3D Elasticity (Table \ref{tab:three-dimensional_elasticity_problem})& \multicolumn{2}{c}{$\left\{\left(\frac{5+10i}{2}, \frac{5+10j}{2}, \frac{5+10k}{2}\right), i,j, k\in \left\{0,1\right\}\right\}$} &  $\left(\frac{5}{2},\frac{5}{2},\frac{5}{2} \right)$ \\
  3D Elasticity (Table \ref{tab:three-dimensional_elasticity_problem}-\ref{tab:three-dimensional_elasticity}) & \multicolumn{2}{c}{$\left(5, 5, 5\right)$} &  $\left(5,5,5\right)$ \\
  3D Stokes (Table \ref{tab:three-dimensional Stokes flow})& \multicolumn{2}{c}{$\left\{\left(\frac{5+10i}{2}, \frac{5+10j}{2}, \frac{5+10k}{2}\right), i,j, k\in \left\{0,1\right\}\right\}$} &  $\left(\frac{5}{2},\frac{5}{2},\frac{5}{2} \right)$ \\
  3D Stokes (Table\ref{tab:three-dimensional Stokes flow}-\ref{tab:3Dtokes}) & \multicolumn{2}{c}{$\left(5, 5, 5\right)$} &  $\left(5,5,5\right)$ \\
  \bottomrule
\end{tabular}
\end{center}
\end{table}
\begin{table}[!htbp]
  \begin{center}
  \caption{Hyper-parameters in the RFM for two-dimensional PDEs without explicit solutions.}\label{tab:hyper-parameters_without_analytical_solutions_2DPDEs}%
  \setlength{\tabcolsep}{2.00mm} 
  \begin{tabular}{lllll}
  \toprule
  Problem  & $\sqrt{Q} $ & $J_n$  & $(N_x, N_y)$\\
  \midrule
  2D Elasticity (Table \ref{tab:my_double_holes})& $\left\{50, 65, 75, 90, 110, 120\right\}$& $200$ &$(8, 4)$\\
  2D Elasticity (Table \ref{tab:complex_domain_elasticity})& $\left\{60, 65, 75, 90, 115, 120\right\}$& $200$ &$(8, 8)$\\
  2D Homogenization (Table \ref{tab:two-dimensional_the_homogenization_problem}) & $\left\{75, 95, 125, 135, 140, 150\right\}$& $300$ &$(8, 8)$\\
  2D Stokes (Table \ref{tab:16_common_stokes_table}-\ref{tab:15_holes_complex_Stokes}) &$\left\{75, 90, 105, 120, 125, 130\right\}$& $300$ &$(5, 5)$\\
  \toprule
  Problem & \multicolumn{2}{c}{$\mathbf{x}_n$} & $ \mathbf{r}_n $  \\
  \midrule
  2D Elasticity (Table \ref{tab:my_double_holes}) & \multicolumn{2}{c}{$\left\{\left(-1+\frac{1+2i}{8}, -\frac{1}{2}+\frac{1+2j}{8}\right), \right.$} & $\left(\frac{1}{8}, \frac{1}{8}\right)$ \\
   & \multicolumn{2}{c}{$\left. i \in \{0, 1, 2, 3, 4, 5, 6, 7\}, j \in \{0, 1, 2, 3\} \right\}$} &  \\
  2D Elasticity (Table \ref{tab:complex_domain_elasticity})& \multicolumn{2}{c}{$\left\{\left(\frac{1+2i}{2}, \frac{1+2j}{2}\right), i,j\in \left\{0,1,2,3,4,5, 6, 7\right\}\right\}$} &  $\left(\frac{1}{2},\frac{1}{2} \right)$ \\
  2D Homogenization (Table \ref{tab:two-dimensional_the_homogenization_problem}) & \multicolumn{2}{c}{$\left\{\left(-1+\frac{1+2i}{8}, -1+\frac{1+2j}{8}\right), i,j\in \left\{0,1, 2, 3, 4, 5, 6, 7\right\}\right\}$} &  $\left(\frac{1}{8},\frac{1}{8} \right)$ \\
  2D Stokes (Table \ref{tab:16_common_stokes_table}-\ref{tab:15_holes_complex_Stokes})  & \multicolumn{2}{c}{$\left\{\left(\frac{1+2i}{10}, \frac{1+2j}{10}\right), i,j\in \left\{0,1, 2, 3, 4\right\}\right\}$} &  $\left(\frac{1}{10},\frac{1}{10} \right)$ \\
  \bottomrule
\end{tabular}
\end{center}
\end{table}
\begin{table}[!htbp]
  \begin{center}
  \caption{Hyper-parameters in the RFM for the three-dimensional PDEs without explicit solution. }\label{tab:hyper-parameters_without_analytical_solutions_3DPDEs}%
  \setlength{\tabcolsep}{1.50mm} 
  \begin{tabular}{lllll}
  \toprule
  Problem  & $\sqrt[3]{Q} $ & $J_n$  & $(N_x, N_y, N_z)$\\
  \midrule
  3D Homogenization (Table \ref{tab:the three-dimensional Homogenization equation})& $33$ & $\left\{200, 400, 600, 800, 1000, 1200, 1400\right\}$  &$(3, 3, 3)$\\
  3D Poisson (Table \ref{tab:the three-dimensional Poisson equation} case I)
  & $63$ & $\left\{200, 400, 600, 800, 1000, 1200, 1400, 1600\right\}$ &$(2, 2, 2)$\\ 
  3D Poisson (Table \ref{tab:the three-dimensional Poisson equation} case II)
  & $70$ & $\left\{200, 400, 600, 800, 1000, 1200, 1300\right\}$ &$(2, 2, 2)$\\
  3D Poisson (Table \ref{tab:the three-dimensional Poisson equation} case III)
  & $75$ & $\left\{200, 400, 600, 800, 900, 1000\right\}$ &$(2, 2, 2)$\\
  \toprule
  Problem & \multicolumn{2}{c}{$\mathbf{x}_n$} & $ \mathbf{r}_n $  \\
  \midrule
  3D Homogenization (Table \ref{tab:the three-dimensional Homogenization equation}) & \multicolumn{2}{c}{$\left\{\left(\frac{1+2i}{6}, \frac{1+2j}{6}, \frac{1+2k}{6}\right), i,j, k\in \left\{0,1, 2\right\}\right\}$} &  $\left(\frac{1}{6},\frac{1}{6},\frac{1}{6} \right)$ \\
  3D Poisson (Table \ref{tab:the three-dimensional Poisson equation})  & \multicolumn{2}{c}{$\left\{\left(-1+\frac{1+2i}{2}, -1+\frac{1+2j}{2}, 1+ \frac{1+2k}{4}\right), i,j, k\in \left\{0,1\right\}\right\}$} &  $\left(\frac{1}{2},\frac{1}{2},\frac{1}{4} \right)$ \\
  \bottomrule
\end{tabular}
\end{center}
\end{table}

\begin{table}[!htbp]
  \begin{center}
  \caption {Geometric setup in Figure \ref{fig:three-dimensional_Poisson_equation_solution} }\label{tab:Filled_Removed_holes}%
  \setlength{\tabcolsep}{4.50mm} 
  \begin{tabular}{llll}
  \toprule
    & Center coordinates & Radius  \\
  \midrule
  Removed & $(1.9575995206832886, 1.2456529587507248)$& $0.0502443732693791$ \\
          & $(2.1249395608901978, 1.5534365177154540)$& $0.2743096835911274$ \\
          & $(2.4759104847908020, 1.7237649857997894)$& $0.1705943234264851$ \\
          & $(1.9882896542549133, 1.5191256999969482)$& $0.2666174061596394$ \\
  \midrule
  Filled  & $(1.7935708165168762, 1.6423535346984863)$& $0.1327976770699024$ \\
          & $(2.1955102682113647, 1.7093743383884430)$& $0.1412965357303619$ \\
          & $(2.0465531945228577, 1.0623437911272050)$& $0.1595700718462467$ \\
          & $(1.7482689023017883, 1.4980989694595337)$& $0.1915273256599903$ \\
          & $(1.9530203938484192, 1.6899727284908295)$& $0.0919794291257858$ \\
  \bottomrule
\end{tabular}
\end{center}
\end{table}

\section{Information of test matrices}\label{test_matrices_information}
\begin{table}[!htbp]
  \begin{center}
  \caption{Information of the matrices from Florida space matrix collection}\label{florida_sparse_matrix_information}%
  \setlength{\tabcolsep}{4.500mm} 
  \begin{tabular}{lllll}
  \toprule
  Name & $(m,n)$ & $\kappa(\mathbf{A})$& $rank(\mathbf{A})$ & density  \\
  \midrule
  GL7d11 & $(1019, 60)$ & $1.65e+16$ & $59$ & $2.4746\%$  \\
  rel6 & $(2340, 157)$ & Inf & $137$ & $1.3885 \%$ \\
  aa4  & $(426, 7195)$ & $9.28e+16$ & $367$ & $1.7005\%$  \\
  relat6 & $(2340, 157)$ & Inf & $137$ & $2.2070 \%$ \\
  air05 & $(426, 7195)$ & $9.28e+16$ & $367$ & $0.1700\%$ \\
  ch8-8-b1 &$(1568, 64)$ & $3.48e+14$ & $63$ & $3.1250\%$ \\
  shar\_te2-b1 &$(17160, 286)$ & $1.56e+13$ & $285$ & $0.6993\%$  \\
  us04 & $(163, 28016)$ & Inf & $115$ & $6.5155\%$ \\
  ch8-8-b2 &$(18816, 1568)$ & $1.63e+15$ & $1505$ & $0.1913\%$  \\
  rel7 & $(21924, 1045)$ & Inf & $1012$ & $0.0024\%$  \\
  kl02 & $(71, 36699)$ & $1.40e+16$ & $64$ & $8.1568\%$  \\
  relat7b & $(21924, 1045)$ & Inf & $1012$ & $0.3551\%$  \\
  stat96v5  & $(2307, 75779)$ & Inf & $2305$ & $0.1338 \%$  \\
  shar\_te2-b2 & $(200200, 17160)$ & $1.05e+15$ & $16875$ & $0.0175\%$\\
  connectus & $(512, 394792)$ & Inf & $456$ & $0.5578\%$\\
  rel8 & $(345688, 12347)$ & Inf & $12289$ & $0.0193\%$ \\
  relat8 & $(345688, 12347)$ & Inf & $12289$ & $0.0313\%$ \\
  12month1 & $(12471, 872622)$ & Inf & $12417$ & $0.2079\%$ \\
  \bottomrule
  \end{tabular}
  \end{center}
\end{table}
\end{appendices}

\end{document}